\newif\ifdraft
\drafttrue
\documentclass[5p, a4paper]{article}

\usepackage[utf8]{inputenc}
\usepackage[T1]{fontenc}
\usepackage{amssymb}
\usepackage{amsfonts}
\usepackage{amsmath}
\usepackage{amsthm}
\usepackage{graphicx}
\usepackage{tikz}
\usepackage{pgfplots}
\usepackage{subfig} 
\usepackage{textcomp}
\usepackage{mathtools}
\usepackage{fullpage}

\newcommand{\eremk}{\hbox{}\hfill\rule{0.8ex}{0.8ex}}
\newcommand{\Omegaext}{\Omega^{\rm ext}}

\newcommand{\R}{\mathbb{R}}
\newcommand{\N}{\mathbb{N}}

\newcommand{\Bv}{\mathbf{v}}

\newcommand{\CB}{\mathcal{B}}

\newtheorem{definition}{Definition}[section]
\newtheorem{theorem}[definition]{Theorem}
\newtheorem{lemma}[definition]{Lemma}
\newtheorem{Proposition}[definition]{Proposition}
\newtheorem{remark}[definition]{Remark}

\numberwithin{equation}{section}

\ifdraft
\usepackage{color}
\newcommand{\todo}[1]{\textcolor{red}{#1}}
\else
\newcommand{\todo}[1]{}
\fi

\newcommand{\diam}{\operatorname*{diam}}
\newcommand{\ext}{\operatorname*{ext}}

\renewcommand{\H}{\mathcal{H}}
\newcommand{\lefttriplenorm}{\ensuremath{\left| \! \left| \! \left|}}
\newcommand{\righttriplenorm}{\ensuremath{\right| \! \right| \! \right|}}
\newcommand{\triplenorm}[1]{\lefttriplenorm #1 \righttriplenorm}
\newcommand{\abs}[1]{\left\vert #1 \right\vert}
\newcommand{\skp}[1]{\left< #1 \right>}
\newcommand{\norm}[1]{\left\| #1 \right\|}
\newcommand{\ra}[0]{\rightarrow}
\newcommand{\T}{\mathcal{T}}

\newcommand\blfootnote[1]{%
  \begingroup
  \renewcommand\thefootnote{}\footnote{#1}%
  \addtocounter{footnote}{-1}%
  \endgroup
}

\begin{document}

\hskip 10 pt

\begin{center}
{\fontsize{14}{20}\bf 
Caccioppoli-type estimates and $\H$-Matrix approximations to inverses for FEM-BEM couplings}
\end{center}

\begin{center}
\textbf{Markus Faustmann, Jens Markus Melenk, Maryam Parvizi}\\
\bigskip
{Institute for Analysis and Scientific Computing}\\
TU Wien\\
Wiedner Hauptstr. 8-10, 1040 Wien, Austria\\markus.faustmann@tuwien.ac.at, melenk@tuwien.ac.at, maryam.parvizi@tuwien.ac.at\\
\bigskip
\end{center}

\begin{abstract}
We consider three different methods for the coupling of the finite element method and the 
boundary element method, the Bielak-MacCamy coupling, 
the symmetric coupling, and the Johnson-N\'ed\'elec coupling. For each coupling we provide discrete interior regularity 
estimates. As a consequence, we are able to prove the existence of exponentially convergent 
$\H$-matrix approximants 
to the inverse matrices corresponding to the lowest order Galerkin discretizations of the couplings.
\blfootnote{{\bf Acknowledgement.} MP was funded by the Austrian Science Fund (FWF) project P 28367 and JMM was supported 
by the Austrian Science Fund (FWF) by the special
research program Taming complexity in PDE systems (grant SFB F65).}
\end{abstract}

\section{Introduction}
Transmission problems are usually posed on unbounded domains, where a (possibly nonlinear) equation is given 
on some bounded domain, and another linear equation is posed on the complement of the bounded domain. 
While the interior problem can be treated numerically by the finite element method (FEM), the unbounded nature
of the exterior problem makes this problematic. A suitable method to treat unbounded problems is provided by 
the boundary element method (BEM), where the differential equation in the unbounded domain 
is reformulated via an integral equation posed just on the boundary. 
In order to combine both methods for transmission problems, additional conditions on the interface have to be 
fulfilled, which leads to different approaches for the coupling of the FEM and the BEM. 
We study three different FEM-BEM couplings, the Bielak-MacCamy coupling \cite{BielakMacCamy}, 
Costabel's symmetric coupling \cite{Costabel88,CostabelStephan90}, and the Johnson-N\'ed\'elec coupling
\cite{JohnsonNedelec80}. Well-posedness and unique solvability of these formulation have been studied in, 
e.g., \cite{Steinbach11,Sayas13,AFFKMP13}, where a main observation is that the couplings are equivalent 
to an elliptic problem.

Elliptic problems typically feature interior regularity known as Caccioppoli estimates, where
stronger norms can be estimated by weaker norms on larger domains. In this paper, we provide 
such Caccioppoli-type estimates for the discrete problem. More precisely, we obtain simultaneous
interior regularity estimates for the finite element solution as well as for the single- and double-layer potential
of the boundary element solution (cf.\ Theorems~\ref{th:CaccioppoliBMcC}, \ref{th:CaccioppoliSymm}, \ref{th:CaccioppoliJN}). 
Discrete Caccioppoli-type estimates for the FEM and the BEM separately can 
be found in our previous works \cite{FMP15,AFM20,FMP16,FMP17}. While the techniques for the FEM and the BEM part are 
similar therein, some essential modifications have to be made to treat the coupling terms on the boundary.

An important consequence of Caccioppoli-type estimates is the existence 
of low-rank approximants to inverses of FEM or BEM matrices, as these inverses are usually dense matrices 
\cite{BebendorfHackbusch03,Boerm10,FMP15,FMP16,FMP17}.
In particular, FEM and BEM inverses can be approximated in the data-sparse $\H$-matrix format, introduced in 
\cite{Hackbusch99}. In comparison with other compression methods, $\H$-matrices have the advantage that they 
come with an additional approximative arithmetic that allows for addition, multiplication, inversion or 
$LU$-decompositions in the $\H$-matrix format; for more details we refer to 
\cite{GrasedyckDissertation,GrasedyckHackbusch03,HackbuschBuch}. In this work, we present an
approximation result 
for the inverses of stiffness matrices corresponding to the lowest order FEM-BEM discretizations.
On admissible blocks, determined by standard admissibility conditions, 
the inverses can be approximated by a low-rank factorization, where the error converges exponentially in the 
rank employed.

The paper is structured as follows: In Chapter~\ref{sec:mainresults}, we present our model problem and state the main results
of the article, the discrete Caccioppoli-type interior regularity estimates for each coupling, and the existence 
of exponentially convergent $\H$-matrix approximants to the inverse matrices corresponding to the 
FEM-BEM discretizations of the couplings. Chapter~\ref{sec:proofsCacc} is concerned with the proofs 
of the Caccioppoli-type estimates. Chapter~\ref{sec:proofHmatrix} provides an abstract framework for the proof of  
low-rank approximability to inverse matrices, which can be applied for other model problems as well. 
Finally, Chapter~\ref{sec:numerics} provides some numerical examples. 

\medskip

\section{Main Results}\label{sec:mainresults}
On a Lipschitz domain $\Omega \subset \mathbb{R}^d$, $d=2,3$ with 
polygonal (for $d = 2$) or polyhedral (for $d = 3$) boundary $\Gamma := \partial \Omega$,
we study the transmission problem
\begin{subequations}\label{eq:model}
\begin{alignat}{5}
-\operatorname*{div}( \mathbf{C}\cdot \nabla u)&=f \quad \hspace{1cm}&&\text{in} \,\,\Omega,\\
\label{e2}
-\Delta u ^ {\ext} &= 0&&\text{in} \,\,\,\Omega ^ {\text{\text{ext}}},\\
\label{e3}
u-u ^ {\ext}&= u_0&&\text{on} \,\, \Gamma ,\\
\label{e4}
\left(\mathbf{C} \nabla u - \nabla u ^ {\ext}\right) \cdot \nu &= \varphi _0\,\,\,\, &&\text{on} \,\, \Gamma, \\
\label{e5}
u ^ {\ext} &= \mathcal{O} ( \left|x \right| ^ {-1}) \,\,\, &&\text{as}\, \left|x \right|  \rightarrow \infty. 
\end{alignat}  
\end{subequations}
Here, $\Omegaext := \R^d \setminus\overline{\Omega}$ denotes the exterior of $\Omega$, and
$\nu$ denotes the outward normal vector. For the data, we assume
$f \in L^2(\Omega)$, $u_0 \in H^{1/2}(\Gamma)$, $\varphi_0 \in H^{-1/2}(\Gamma)$, and
$\mathbf{C} \in L^{\infty}(\Omega;\R^d)$ to be pointwise symmetric and positive definite, i.e.,
there is a constant $C_{\rm ell}>0$ such that
\begin{equation}
\label{eq:Cell}
\skp{\mathbf{C}x,x}_2\geq C_{\rm ell} \norm{x}_2^2.
\end{equation}
%For $d = 2$, we additionally require the compatibility condition $\skp{f,1}_{L^2(\Omega)} + \skp{\varphi_0,1}_{L^2(\Gamma)} = 0$
%to satisfy the radiation condition (\ref{e5}) for the 
%solvability of the variational formulation. Furthermore, 
For $d = 2$, we assume $\operatorname{diam} \Omega < 1$ 
for the single-layer operator $V$ introduced below to be elliptic.
\begin{remark}
In the following, we consider three different variational formulations, 
namely, the symmetric coupling, the Bielak-MacCamy coupling, 
and the Johnson-N\'ed\'elec coupling  
for our model problem. 
All three are well-posed without compatibility assumptions on the data. 
The compatibility condition $\skp{f,1}_{L^2(\Omega)} + \skp{\varphi_0,1}_{L^2(\Gamma)} = 0$ for $d=2$ ensures the radiation condition (\ref{e5}); 
lifting the compatibility condition yields a solution that satisfies 
a different radiation condition, namely, 
$u^{\rm ext}  = b \log |x|+ \mathcal{O} ( \left|x \right| ^ {-1})$ as $|x| \rightarrow \infty$ for some $b \in \R$
for the three coupling strategies considered. 
Our analysis requires only 
the unique solvability of the variational formulations.  
\eremk
\end{remark}
With the Green's function for the Laplacian 
$G(x) = -\frac{1}{2\pi} \log\abs{x}$ for $d=2$ and $G(x) = \frac{1}{4\pi}\frac{1}{\abs{x}}$ for $d=3$,
we introduce the single-layer boundary integral operator
$V \in L(H^{-1/2}(\Gamma),H^{1/2}(\Gamma))$ by 
\begin{eqnarray*}
V\phi(x) := \int_{\Gamma} G(x-y)\phi(y)ds_y, \quad x \in \Gamma.
\end{eqnarray*}
The double-layer operator
$K \in L(H^{1/2}(\Gamma),H^{1/2}(\Gamma))$ has the form 
\begin{eqnarray*}
K\phi(x) := \int_{\Gamma} (\partial_{\nu(y)}G(x-y))\phi(y)ds_y, \quad x \in \Gamma,
\end{eqnarray*}
where $\partial_{\nu(y)}$ denotes the normal derivative at the point $y$.
The adjoint of $K$ is denoted by $K'$.
Finally, the hyper-singular operator $W\in L(H^{1/2}(\Gamma),H^{-1/2}(\Gamma))$ 
is given by 
\begin{eqnarray*}
W\phi(x) := -\partial_{\nu(x)}\int_{\Gamma} (\partial_{\nu(y)}G(x-y))\phi(y)ds_y, \quad x \in \Gamma.
\end{eqnarray*}
The single-layer operator $V$ is elliptic for $d=3$ and for $d=2$ provided $\diam(\Omega) < 1$.
The hyper-singular operator $W$ is  semi-elliptic with a kernel of dimension
being the number of components of connectedness of $\Gamma$.

In addition to the boundary integral operators, 
we need the volume potentials $\widetilde{V}$ and $\widetilde{K}$ defined by
\begin{align*}
\widetilde{V} \phi ( x) &:= \int_{\Gamma}G( x-y) \phi  ( y) d s _y,  &x \in \mathbb{R}^d \backslash \Gamma,\\
\widetilde{K} \phi ( x) &:= \int_{\Gamma} \partial _{{\nu} ( y) }G( x-y) \phi( y) d s_y,  &x \in \mathbb{R}^d \backslash \Gamma.
\end{align*}

In this paper, we study discretizations of
weak solutions of the model problem reformulated via three different FEM-BEM couplings: 
the Bielak-MacCamy coupling, Costabel's symmetric coupling, and the Johnson-N\'ed\'elec coupling. 
All these couplings lead to a variational formulation of finding 
$(u,\varphi) \in H^1(\Omega) \times H^{-1/2}(\Gamma) =: \mathbf{X}$
such that
\begin{align}\label{eq:abstractFEMBEM}
a(u,\varphi;\psi,\zeta) = g(\psi,\zeta) \qquad \forall (\psi,\zeta) \in \mathbf{X},
\end{align}
where $a: \mathbf{X}\times \mathbf{X} \rightarrow \R$ is a bilinear form and 
$g: \mathbf{X} \rightarrow \R$ is continuous linear functional. \\

For the discretization, we assume that $\Omega$ is triangulated by a quasi-uniform mesh 
${\mathcal T}_h=\{T_1,\dots,T_{\widehat n}\}$ of mesh width 
$h := \max_{T_j\in \mathcal{T}_h}{\rm diam}(T_j)$. 
The elements $T_j \in \mathcal{T}_h$ are open triangles ($d=2$) or tetrahedra ($d=3$). 
Additionally, we assume that the mesh $\T_h$ is regular in the sense of Ciarlet and 
$\gamma$-shape regular in the sense that 
 we have ${\rm diam}(T_j) \le \gamma\,|T_j|^{1/2}$ for all $T_j\in\mathcal{T}_h$, 
where $|T_j|$ denotes the Lebesgue measure of $T_j$. 
By $\mathcal{K}_h:=\{K_1,\dots,K_  {\widehat{m}}\}$, we denote the restriction of ${\mathcal T}_h$ to the boundary 
$\Gamma$, which is a regular and shape-regular triangulation of the boundary. \\

For simplicity, we consider lowest order Galerkin discretizations in 
$S^{1,1}(\T_h) \times S^{0,0}(\mathcal{K}_h)$, where
\begin{align*}
 S^{1,1}({\mathcal T}_h) &:= \{u \in C(\Omega)\, \colon \, u|_T \in P_1(T) \quad  \forall T \in \T_h \}, \\
 S^{0,0}(\mathcal{K}_h) &:= \{u \in L^2(\Gamma)\, \colon \, u|_K \in P_0(K) \quad \forall K \in \mathcal{K}_h \},
\end{align*}
with $P_p(T)$ denoting the space of polynomials of maximal degree $p$ on an element $T$.
We let ${\mathcal B}_h:= \{\xi_j\, :\, j = 1,\dots, n\}$ be the basis of 
$S^{1,1}({\mathcal T}_h)$ consisting of the standard hat functions, and we let 
 ${\mathcal W}_h:= \{\chi_j\, :\, j = 1,\dots, m\}$ be the basis of 
$S^{0,0}({\mathcal K}_h)$ that consists of the
 characteristic functions of the surface elements. 
These bases feature the following norm equivalences: 
\begin{subequations}
\begin{align}\label{eq:basisa}
c_1h^{d/2}\norm{\mathbf{x}}_2 &\leq \norm{\Phi\mathbf{x}}_{L^2(\Omega)} 
\leq c_2 h^{d/2}\norm{\mathbf{x}}_2
\quad\quad\;\; \forall\, \mathbf{x} \in \R^n, \\  \label{eq:basisb}
c_3h^{(d-1)/2}\norm{\mathbf{y}}_2 &\leq \norm{\Psi\mathbf{y}}_{L^2(\Gamma)} 
\leq c_4 h^{(d-1)/2}\norm{\mathbf{y}}_2
\quad \forall\, \mathbf{y} \in \R^m 
\end{align}
\end{subequations}
for the 
isomorphisms $\Phi:\R^n\ra S^{1,1}({\mathcal T}_h)$, $\mathbf{x} \mapsto \sum_{j=1}^n\mathbf{x}_j\xi_j$ and
$\Psi:\R^m\ra S^{0,0}({\mathcal K}_h)$, $\mathbf{y} \mapsto \sum_{j=1}^m\mathbf{y}_j\chi_j$.\\

Finally, we need the notion of concentric boxes.
\begin{definition}
\textbf{(Concentric boxes)}
Two  (quadratic) boxes $B_R$ and $B_{R^\prime}$ of side length $R$ and ${R^\prime}$
are said to be concentric if they have the same barycenter and $B_R$
can be obtained by a stretching of $B_{R^\prime}$
by the factor $R/R^\prime$ taking their common barycenter as the origin.
\end{definition}

Before we can state our first main results, the interior regularity estimates, 
we specify the norm we are working with, 
an $h$-weighted $H^1$-equivalent norm.
For a box $B_R$ with side length $R$, an open set 
$\omega \subset \R^d$, and $v \in H^1(B_R\cap \omega)$, we introduce
\begin{align}\label{eq:def:triplenorm}
\triplenorm{  v }  ^2 _{h, R, \omega} :=
h^2\left\| \nabla v \right\|^2_{L^2(B_{R}\cap \omega)} 
+ \left\| v \right\|^2_{L^2(B_{R}\cap \omega)}.
\end{align}
For the case $\omega = \R^d$, we abbreviate $\triplenorm{\cdot}_{h,R,\R^d} =: \triplenorm{\cdot}_{h,R}$ and 
for the case $\omega = \R^d\backslash \Gamma$
we write $\triplenorm{\cdot}_{h,R,\R^d\backslash\Gamma} =: \triplenorm{\cdot}_{h,R,\Gamma^c}$ and understood the norms over 
$B_R\backslash \Gamma$ as a sum over integrals $B_R\cap\Omega$ and $B_R\cap \Omega^{\rm ext}$.
Moreover, for triples $(u,v,w) \in H^1(B_R\cap\Omega) \times H^1(B_R)\times H^1(B_R\backslash\Gamma)$, we set
\begin{align}\label{eq:def:triplenormVec}
\triplenorm{(u,v,w)}_{h,R}^2 := \triplenorm{u}_{h,R,\Omega}^2+\triplenorm{v}_{h,R}^2+\triplenorm{w}_{h,R,\Gamma^c}^2.
\end{align}
We note that $u$ will be the interior solution, $v$ be chosen as a single-layer potential and $w$ as a double-layer potential 
(which jumps across $\Gamma$), which explains the different requirements for the set $\omega$.

\subsection{The Bielak–MacCamy coupling} \label{S1}
The Bielak–MacCamy coupling is derived by making a single-layer ansatz for the exterior 
solution, i.e., $u^{\rm ext} = \widetilde{V}\varphi$ in $\Omegaext$ with an 
unknown density $\varphi \in H^{-1/2}(\Gamma)$. For more details, we refer to \cite{BielakMacCamy}.
This approach leads to the bilinear form 
\begin{subequations}\label{eq:BMCcouplingBLF}
\begin{alignat}{2}
a_{\rm bmc}(u,\varphi;\psi,\zeta) &:= 
 \label{e6}\skp{\mathbf{C}\nabla u , \nabla  \psi}_{L^2(\Omega) } + \skp{ (1/2-{K^\prime})\varphi,\psi} _ {L^{2}(\Gamma)} 
-
\skp{u ,\zeta}_{L^2(\Gamma)} +  \skp{V \varphi,\zeta}_{L^2(\Gamma)}, \\
g_{\rm bmc}(\psi,\zeta) &:= 
\skp{f,\psi}_{L^2(\Omega)} + \skp{\varphi_0,\psi}_{L^2(\Gamma)} - \skp{u_0 , \zeta}_{L^2(\Gamma)}.
\end{alignat}
\end{subequations}
Replacing $H^1 (\Omega) \times H^{-{1}/{2}}(\Gamma)$ by the finite dimensional subspace 
$S^{1,1}(\mathcal{T}_h) \times S^{0,0}( \mathcal{\mathcal{K}}_h)$, we arrive at the Galerkin discretization of \eqref{eq:BMCcouplingBLF} 
of finding $( u _h , \varphi _h) \in S^{1,1}(\mathcal{T}_h) \times S^{0,0}(\mathcal{K}_h)$ such that
\begin{subequations}\label{eq:BMcC}
\begin{alignat}{2}
\skp{\mathbf{C} \nabla u_h , \nabla \psi_h}_{L^2(\Omega)} + \skp{(1/2 -{K^\prime})\varphi_h,\psi_h}_{L^2(\Gamma)} 
&= \skp{f,\psi_h}_{L^2(\Omega)} + \skp{\varphi_0,\psi_h}_{L^2(\Gamma)}  \,\, \forall \psi_h \in S^{1,1}(\T_h),\\
\skp{u_h,\zeta_h}_{L^2(\Gamma)} -  \skp{V \varphi_h, \zeta_h}_{L^2(\Gamma)}
&= \skp{u_0, \zeta_h}_{L^2(\Gamma)} \,\, \forall \zeta _h  \in 
S^{0,0}(\mathcal{K}_h)  .
\end{alignat}
\end{subequations}

If the ellipticity constant of $\mathbf{C}$ satisfies
$C_{\rm ell} > 1/4$, then \cite[Thm.~9]{AFFKMP13} shows that the Bielak-MacCamy coupling is equivalent to an elliptic
problem with the use of a (theoretical) implicit stabilization. Therefore, \eqref{eq:BMcC} 
is uniquely solvable.

The following theorem is one of the main results of our paper. It states that for the interior finite element 
solution and the single-layer potential of the boundary element solution, a Caccioppoli type estimate holds, 
i.e., the stronger $H^1$-seminorm can be estimated by a weaker $h$-weighted $H^1$-norm on a larger domain.

\begin{theorem}\label{th:CaccioppoliBMcC}
Assume that $C_{\rm ell} > 1/4$ in (\ref{eq:Cell}).
Let $\varepsilon \in (0,1) $ and $R \in (0,2\diam(\Omega))$ be such that $\frac{h}{R} < \frac{\varepsilon}{16}$, 
and let $B_R$ and $B_{(1+\varepsilon)R}$ be two concentric boxes. Assume that the data is localized away from 
$B_{(1+\varepsilon)R}$, i.e., 
$(\operatorname*{supp} f \cup \operatorname*{supp} \varphi_0 \cup  \operatorname*{supp} u_0) \cap B_{(1+\varepsilon)R} = \emptyset$.
Then, there exists a constant $C$ depending only on $\Omega$, $d$, 
and the $\gamma$-shape regularity of the quasi-uniform triangulation $\T_h$, such that for the solution  
$(u_h,\varphi_h)$ of \eqref{eq:BMcC}  we have
\begin{align}
\nonumber
\left\| \nabla u_h \right\|_ {L^2(B_{R}\cap\Omega)}
+\left\| \nabla \widetilde{V} \varphi_h \right\|_ {L^2(B_{R})} 
&\le 
 \frac{C}{\varepsilon R}  \left( 
\triplenorm{  u_h }  _{h,( 1+\varepsilon) R, \Omega} + 
\triplenorm { \widetilde{V}\varphi_h } _{h,(1+\varepsilon) R} \right), 
\end{align}
where the norms on the right-hand side are defined in \eqref{eq:def:triplenorm}.
\end{theorem}

With the bases $\mathcal{B}_h$ of $S^{1,1}(\T_h)$ and $\mathcal{W}_h$ of $S^{0,0}(\mathcal{K}_h)$, the 
Galerkin discretization \eqref{eq:BMcC} 
leads to a block matrix $\mathbf{A}_{\rm bmc} \in \mathbb{R}^ {(n+m) \times( n+m)}$
\begin{align}\label{eq:matrixBMcC}
\mathbf{A}_{\rm bmc}:= \begin{pmatrix}
\mathbf{A}&\frac{1}{2}\mathbf{M}^T-\mathbf{K}^T\\
\mathbf{M}&\mathbf{V}
\end{pmatrix},
\end{align}
where $\mathbf{A} \in \mathbb{R} ^ {n\times n}$ is given by
$\mathbf{A}_{ij} = \skp{\mathbf{C}\nabla \xi_j,\nabla \xi_i }_{L^2(\Omega)}$, 
$\mathbf{M} \in \mathbb{R} ^ {m\times n}$ by $\mathbf{M}_{ij} = \skp{\xi_i,\chi_j}_{L^2(\Gamma)}$, 
$\mathbf{K} \in \mathbb{R} ^ {m\times n}$ by $\mathbf{K}_{ij} = \skp{K\xi_i,\chi_j}_{L^2(\Gamma)}$,
and $\mathbf{V} \in \mathbb{R} ^ {m\times m}$ by $\mathbf{V}_{ij} = \skp{V\chi_j,\chi_i}_{L^2(\Gamma)}$.

%%%%%%%%%%%%%%%%%%%%%%%%%%%%%%%%%%%%%%%%%%%%%%%%%%%%%%%%%%%%%%%%%%%%%%

\subsection{Costabel's symmetric coupling}
Using the representation formula, or more precisely, both single- and double-layer potential, 
for the exterior solution, one obtains an expression
$u^{\rm ext} = - \widetilde{V}\varphi + \widetilde{K}u^{\rm ext}$ with
$\varphi = \nabla u^{\rm ext} \cdot \nu$, \cite[Eq.~(55)]{AFFKMP13}. By coupling the interior and exterior solution in a symmetric way
(which uses all four boundary integral operators), this leads to Costabel's symmetric coupling, introduced
in \cite{Costabel88} and \cite{han90}.
Here, the bilinear form and right-hand side are given by
\begin{subequations}\label{eq:symmcouplingBLF}
\begin{alignat}{2}
 a_{\rm sym}(u,\varphi;\psi,\zeta) &:= 
 \skp{\mathbf{C}\nabla u, \nabla\psi}_{L^2(\Omega)} + \skp{({K^\prime}-1/2)\varphi,\psi}_{L^2(\Gamma)} 
 + \skp{W u, \psi}_{L^2(\Gamma)}  \nonumber \\
 &\qquad +\skp{(1/2-K) u,\zeta}_{L^2(\Gamma)} +  \skp{V \varphi,\zeta}_{L^2(\Gamma)},
 \\ \nonumber
 g_{\rm sym}(\psi,\zeta) &:= 
 \skp{f,\psi}_{L^2(\Omega)} + \skp{\varphi_0 + Wu_0,\psi}_{L^2(\Gamma)} +\skp{(1/2-K) u_0,\zeta}_{L^2(\Gamma)} 
 \\
 &=: 
 \skp{f,\psi}_{L^2(\Omega)} + \skp{v_0,\psi}_{L^2(\Gamma)} +\skp{w_0,\zeta}_{L^2(\Gamma)}. 
\end{alignat}
\end{subequations}
The Galerkin discretization leads to the problem of
finding $(u_h , \varphi_h) \in S^{1,1}(\T_h) \times S^{0,0}(\mathcal{K}_h)$ such that
\begin{subequations}\label{eq:symmcoupling}
\begin{alignat}{2}
  \skp{\mathbf{C}\nabla u_h, \nabla\psi_h}_{L^2(\Omega)} 
  \! +\! \skp{ ({K^\prime}\!-\!1/2)\varphi_h, \psi_h}_{L^2(\Gamma)}\! +\! \skp{W u_h, \psi_h}_{L^2(\Gamma)} 
  &\! =\! \skp{f, \psi_h}_{L^2(\Omega)}\! +\! \skp{v_0, \psi_h}_{L^2(\Gamma)}	
  \label{eq40a}  \\
  \label{eq40b}\skp{(1/2-K)u_h ,\zeta_h}_{L^2(\Gamma)} 
  +  \skp{V \varphi_h, \zeta_h}_{L^2(\Gamma)} 
  &\!= \! \skp{w_0 , \zeta_h}_{L^2(\Gamma)}
\end{alignat}
\end{subequations}
for all $(\psi_h,\zeta_h) \in S^{1,1}(\T_h)\times S^{0,0}(\mathcal{K}_h)$. \\

With similar arguments as for the Bielak-MacCamy coupling, 
\cite{AFFKMP13} prove unique solvability for the symmetric coupling for any $C_{\rm ell}>0$. \\

The following theorem is similar to Theorem~\ref{th:CaccioppoliBMcC} and provides a simultaneous Caccioppoli-type 
estimate for the interior solution as well as for the single-layer potential of the boundary solution 
and the double-layer potential of the trace of the interior solution. Here, the double-layer potential additionally
appears since all boundary integral operators, especially the hyper-singular operator %(cf.~\cite{FMP17}),
appear in the coupling.

\begin{theorem}\label{th:CaccioppoliSymm}
Let $\varepsilon \in (0,1) $ and $R \in (0,2\diam(\Omega)) $ be such that $\frac{h}{R} < \frac{\varepsilon}{32}$, 
and let $B_R$ and $B_{(1+\varepsilon)R}$ be two concentric boxes. Assume that the data is localized away from 
$B_{(1+\varepsilon)R}$, i.e., 
$(\operatorname*{supp} f \cup  \operatorname*{supp}v_0 \cup  \operatorname*{supp} w_0) \cap B_{(1+\varepsilon)R} = \emptyset$.
Then, there exists a constant $C$ depending only on $\Omega$, $d$, 
and the $\gamma$-shape regularity of the quasi-uniform triangulation $\T_h$, such that for the solution  
$(u_h,\varphi_h)$ of \eqref{eq:symmcoupling}  we have
\begin{align}
\left\| \nabla u_h \right\|_ {L^2(B_{R}\cap\Omega)} + \left\| \nabla \widetilde{V} \varphi_h \right\|_ {L^2(B_{R})} +
\left\| \nabla \widetilde{K}u_h \right\|_ {L^2(B_{R}\backslash\Gamma)} \le 
 \frac{C}{\varepsilon R} \triplenorm{(u_h,\widetilde V \phi_h,\widetilde K u_h)}_{h,(1+\varepsilon) R},
%\left( 
% \triplenorm{  u_h }_{h,( 1+\varepsilon) R, \Omega} + 
%\triplenorm{  \widetilde{V}\varphi_h } _{h,(1+\varepsilon) R}  +
%\triplenorm{  \widetilde{K}u_h } _{h,(1+\varepsilon) R,\Gamma^c} 
%\right). 
\end{align}
where the norm on the right-hand side is defined in \eqref{eq:def:triplenormVec}.
\end{theorem}	

With the bases $\mathcal{B}_h$ of $S^{1,1}(\T_h)$ and $\mathcal{W}_h$ of $S^{0,0}(\mathcal{K}_h)$,
the Galerkin discretization \eqref{eq:symmcoupling} 
leads to a block matrix $\mathbf{A}_{\rm sym} \in \mathbb{R}^ {(n+m) \times( n+m)}$
\begin{align}\label{eq:matrixSymm}
\mathbf{A}_{\rm sym}:= \begin{pmatrix}
\mathbf{A}+\mathbf{W}&\mathbf{K}^T-\frac{1}{2}\mathbf{M}^T\\
\frac{1}{2}\mathbf{M}-\mathbf{K}&\mathbf{V}
\end{pmatrix},
\end{align}
where $\mathbf{A}$, $\mathbf{M}$, $\mathbf{K}$ are defined in \eqref{eq:matrixBMcC}, and 
$\mathbf{W} \in \R^{n\times n}$ is given by $\mathbf{W}_{ij} = \skp{W\xi_j,\xi_i}_{L^2(\Gamma)}$.

%%%%%%%%%%%%%%%%%%%%%%%%%%%%%%%%%%%%%%%%%%%%%%%%%%%%%%%%%%%%%%%%%%%%%%

\subsection{The Johnson-N\'ed\'elec coupling}
The Johnson-N\'ed\'elec coupling, introduced in \cite{JohnsonNedelec80} 
again uses the representation formula for the exterior solution, but 
differs from the symmetric coupling in the way how the interior and exterior solutions 
are coupled on the boundary. Instead of all four boundary integral operators, only the 
single-layer and the double-layer operator are needed.
The bilinear form for the Johnson-N\'ed\'elec coupling is given by
\begin{subequations}
\begin{alignat}{2}
a_{\rm jn}(u,\varphi;\psi,\zeta) &:=
\skp{\mathbf{C}\nabla u,\nabla \psi}_{L^2(\Omega)} - \skp{\varphi,\psi}_{L^2(\Gamma)} +
\skp{(1/2-K)u,\zeta}_{L^2(\Gamma)} + \skp{V\varphi,\zeta}_{L^2(\Gamma)}, \\
g_{\rm jn}(\psi,\zeta) &:= 
\skp{f,\psi}_{L^2(\Omega)} + \skp{\varphi_0,\psi}_{L^2(\Gamma)} + \skp{(1/2-K)u_0,\zeta}_{L^2(\Gamma)} \nonumber \\
&=: 
\skp{f,\psi}_{L^2(\Omega)} + \skp{\varphi_0,\psi}_{L^2(\Gamma)} + \skp{w_0,\zeta}_{L^2(\Gamma)}.
\end{alignat}
\end{subequations}
The Galerkin discretization in $S^{1,1}(\mathcal{T}_h) \times S^{0,0}( \mathcal{\mathcal{K}}_h)$
leads to the problem of finding $(u_h,\varphi_h) \in S^{1,1}(\mathcal{T}_h) \times S^{0,0}( \mathcal{\mathcal{K}}_h)$
such that
\begin{subequations}\label{eq:JN}
\begin{alignat}{2}
\skp{\mathbf{C}\nabla u_h,\nabla \psi_h}_{L^2(\Omega)} - \skp{\varphi_h,\psi_h}_{L^2(\Gamma)}  &= 
\skp{f,\psi_h}_{L^2(\Omega)} + \skp{\varphi_0,\psi_h}_{L^2(\Gamma)} \quad \forall \psi_h \in S^{1,1}(\mathcal{T}_h), \\
\skp{(1/2-K)u_h,\zeta_h}_{L^2(\Gamma)} + \skp{V\varphi_h,\zeta_h}_{L^2(\Gamma)} &= 
\skp{(1/2-K)u_0,\zeta_h}_{L^2(\Gamma)} \quad \!\! \forall \zeta_h \in S^{0,0}( \mathcal{\mathcal{K}}_h).
\end{alignat}
\end{subequations}

As in the case of the Bielak-MacCamy coupling, the Johnson-N\'ed\'elec coupling has an unique solution 
provided $C_{\rm ell}>1/4$, see \cite{AFFKMP13}. \\

The following theorem gives the analogous result to Theorem~\ref{th:CaccioppoliBMcC} and 
Theorem~\ref{th:CaccioppoliSymm} for the Johnson-N\'ed\'elec coupling. Similarly to the symmetric coupling, 
we simultaneously control a stronger norm of the interior solution and both layer potentials by a 
weaker norm on a larger domain.

\begin{theorem}\label{th:CaccioppoliJN}
Assume that $C_{\rm ell} > 1/4$ in (\ref{eq:Cell}).
Let $\varepsilon \in (0,1) $ and $R \in (0,2\diam(\Omega)) $ be such that $\frac{h}{R} < \frac{\varepsilon}{32}$, 
and let $B_R$ and $B_{(1+\varepsilon)R}$ be two concentric boxes. Assume that the data is localized away from 
$B_{(1+\varepsilon)R}$, i.e., 
$(\operatorname*{supp} f\cup \operatorname*{supp} \varphi_0 \cup \operatorname*{supp} (1/2-K)u_0) \cap B_{(1+\varepsilon)R} = \emptyset$.
Then, there exists a constant $C$ depending only on $\Omega$, $d$, 
and the $\gamma$-shape regularity of the quasi-uniform triangulation $\T_h$, such that for the solution  
$(u_h,\varphi_h)$ of \eqref{eq:symmcoupling}  we have
\begin{align}
\left\| \nabla u_h \right\|_ {L^2(B_{R}\cap\Omega)} + \left\| \nabla \widetilde{V} \varphi_h \right\|_ {L^2(B_{R})} +
\left\| \nabla \widetilde{K}u_h \right\|_ {L^2(B_{R}\backslash\Gamma)} \le C
\frac{R}{(\varepsilon R)^2} 
\triplenorm{(u_h,  \widetilde{V}\varphi_h, \widetilde{K}u_h )  } _{h,( 1+\varepsilon) R},
\end{align}
where the norm on the right-hand side is defined in \eqref{eq:def:triplenormVec}.
\end{theorem}	

With the bases $\mathcal{B}_h$ of $S^{1,1}(\T_h)$ and $\mathcal{W}_h$ of $S^{0,0}(\mathcal{K}_h)$,
the Galerkin discretization \eqref{eq:JN} leads to a matrix 
$\mathbf{A}_{\rm jn} \in \mathbb{R}^ {(n+m) \times( n+m)}$
\begin{align}\label{eq:matrixJN}
\mathbf{A}_{\rm jn}:= \begin{pmatrix}
\mathbf{A} & -\mathbf{M}^T\\
\frac{1}{2}\mathbf{M}-\mathbf{K}&\mathbf{V}
\end{pmatrix},
\end{align}
where $\mathbf{A}$, $\mathbf{M}$, $\mathbf{K}$, $\mathbf{V}$ are defined in \eqref{eq:matrixBMcC}.

%%%%%%%%%%%%%%%%%%%%%%%%%%%%%%%%%%%%%%%%%%%%%%%%%%%%%%%%%%%%%%%%%%%%%%		
		
\subsection{$\H$-Matrix approximation of inverses}
As a consequence of the Caccioppoli-type inequalities, we are able to prove the existence 
of $\H$-matrix approximants to the inverses of the stiffness matrices corresponding to 
the discretized FEM-BEM couplings.

We briefly introduce the matrix compression format of $\mathcal{H}$-matrices. For more detailed information,
we refer to \cite{Hackbusch99,BebendorfBuch,HackbuschBuch,BoermBuch}.

The main idea of $\mathcal{H}$-matrices is to store certain far field blocks of the matrix 
efficiently as a low-rank matrix. In order to choose blocks that are suitable for compression, 
we need to introduce the concept of admissibility. 

\begin{definition}[bounding boxes and $\eta$-admissibility]
\label{def:admissibility}
A cluster $\tau$ is a subset of the index set 
	$\mathcal{I}=\{1,2,...,n+m\}$. 
	For a cluster $\tau \subset \mathcal{I}$, the axis-parallel $B_{R_\tau} \subseteq \mathbb{R}^d$ is called a bounding box 
	if $B_{R_\tau}$ is a hyper cube with side length $R_\tau$ and 
	$\cup _ {i \in \tau} \operatorname*{supp} \xi _i \subseteq B_{R_\tau}$ 
	as well as 	$\cup _ {i \in \tau} \operatorname*{supp} \chi _i \subseteq B_{R_\tau}$.
\\
For $\eta >0$, a pair of clusters $(\tau , \sigma)$ with 
$\tau , \sigma \subset \mathcal{I}$ is called $\eta$-admissible if 
there exist bounding boxes $B_{R_\tau}$ and $B_{R_\sigma}$ such that 
\begin{align*}
\min \{\operatorname*{diam}(B_{R_\tau}), \operatorname*{diam}(B_{R_\sigma})\} \leq 
\eta \, \operatorname*{dist} (B_{R_\tau},B_{R_\sigma}).
\end{align*}
\end{definition}

\begin{remark}
Definition~\ref{def:admissibility} clusters the degrees of freedom associated with 
triangulation ${\mathcal T}_h$ of $\Omega$ and the triangulation ${\mathcal K}_h$ 
of $\Gamma$ simultaneously.
\eremk
\end{remark}
	
The block-partition of $\mathcal{H}$-matrices is based on so-called cluster trees.

\begin{definition}[cluster tree]
A \emph{cluster tree} with \emph{leaf size} $n_{\rm leaf} \in \mathbb{N}$ is a binary tree 
$\mathbb{T}_{\mathcal{I}}$ with root $\mathcal{I}$  
such that each cluster $\tau \in \mathbb{T}_{\mathcal{I}}$ is either a leaf of the tree 
and satisfies $\abs{\tau} \leq n_{\rm leaf}$, or there exist disjoint subsets 
$\tau'$, $\tau'' \in \mathbb{T}_{\mathcal{I}}$ of $\tau$, so-called sons, with 
$\tau = \tau' \cup \tau''$. 
Here and below, $\abs{\tau}$ denotes the cardinality of the finite set $\tau$.
The \emph{level function} ${\rm level}: \mathbb{T}_{\mathcal{I}} \rightarrow \mathbb{N}_0$ 
is inductively defined by 
${\rm level}(\mathcal{I}) = 0$ and ${\rm level}(\tau') := {\rm level}(\tau) + 1$ for $\tau'$ a son of $\tau$. 
The \emph{depth} of a cluster tree
is ${\rm depth}(\mathbb{T}_{\mathcal{I}}) := \max_{\tau \in \mathbb{T}_{\mathcal{I}}}{\rm level}(\tau)$.  
\end{definition}

\begin{definition}[far field, near field, and sparsity constant]
 A partition $P$ of $\mathcal{I} \times \mathcal{I}$ is said to be based on the cluster tree 
 $\mathbb{T}_{\mathcal{I}}$, 
if $P \subset \mathbb{T}_{\mathcal{I}}\times\mathbb{T}_{\mathcal{I}}$. For such a partition $P$ 
and a fixed admissibility parameter $\eta > 0$, we define the \emph{far field} and the \emph{near field} 
as 
\begin{equation}\label{eq:farfield}
P_{\rm far} := \{(\tau,\sigma) \in P \; : \; (\tau,\sigma) \; \text{is $\eta$-admissible}\}, 
\quad P_{\rm near} := P\setminus P_{\rm far}.
\end{equation}
The \emph{sparsity constant} $C_{\rm sp}$ of such a partition 
was introduced in \cite{HackbuschKhoromskij2000a,GrasedyckDissertation} as 
\begin{equation}\label{eq:sparsityConstant}
C_{\rm sp} := \max\left\{\max_{\tau \in \mathbb{T}_{\mathcal{I}}}
\abs{\{\sigma \in \mathbb{T}_{\mathcal{I}} \, : \, \tau \times \sigma \in P_{\rm far}\}},
\max_{\sigma \in \mathbb{T}_{\mathcal{I}}}\abs{\{\tau \in \mathbb{T}_{\mathcal{I}} \, : \, 
\tau \times \sigma \in P_{\rm far}\}}\right\}.
\end{equation}
\end{definition}

\begin{definition}[$\mathcal{H}$-matrices]
Let $P$ be a partition of $\mathcal{I} \times \mathcal{I}$ that is based on a cluster tree $\mathbb{T}_{\mathcal{I}}$ 
and $\eta >0$. 
A matrix $\mathbf{A}  \in \mathbb{R} ^ {(n+m)\times(n\times m)}$ is an $\mathcal{H}$-matrix with blockwise rank $r$, if for every
$\eta$-admissible block $(\tau , \sigma) \in P_{\rm far}$, we have a low-rank factorization 
$$\mathbf{A}|_{\tau \times \sigma}=\mathbf{X}_{\tau \sigma}  \mathbf{Y}^ T_{\tau \sigma},$$
where $\mathbf{X}_{\tau \sigma} \in \mathbb{R} ^ {|\tau|\times r}$ and  
$\mathbf{Y}_{\tau \sigma} \in \mathbb{R} ^ {|\sigma|\times r }$.
\end{definition}

Due to the low-rank structure on far-field blocks, the memory requirement to store an  
$\mathcal{H}$ matrix is given by $\sim C_{\rm sp}\operatorname*{depth}(\mathbb{T}_{\mathcal{I}})r(n+m)$. Provided 
$C_{\rm sp}$ is bounded and the cluster tree is balanced, i.e., 
$\operatorname*{depth}(\mathbb{T}_{\mathcal{I}})\sim \log(n+m)$, which can be ensured by suitable clustering methods 
(e.g. geometric clustering, \cite{HackbuschBuch}), we get a storage complexity of 
$\mathcal{O}(r(n+m)\log(n+m))$. \\

The following theorem shows that the inverse matrices 
$\mathbf{A}_{\rm bmc}^{-1}$, $\mathbf{A}_{\rm sym}^{-1}$, and $\mathbf{A}_{\rm jn}^{-1}$
corresponding to the three mentioned FEM-BEM couplings
can be approximated in the $\mathcal{H}$-matrix format, and the error 
converges exponentially in the maximal block rank employed.

\begin{theorem}\label{th:H-Matrix approximation of inverses}
For a fixed admissibility parameter $\eta >0$, let a partition 
$P$ of $\mathcal{I} \times \mathcal{I}$ that is based on the cluster tree $\mathbb{T}_{\mathcal{I}}$ be given.
Then, there exists an $\mathcal{H}$-matrix $\mathbf{B}_{\mathcal{H}}$ 
with maximal blockwise rank $r$ such that 
\begin{align*}
\left\|\mathbf{A}_{\rm bmc}^{-1} -\mathbf{B}_{\mathcal{H}}
\right\|_2 \le C_{\rm  apx} C_{\rm sp} \operatorname*{depth}(\mathbb{T}_{\mathcal{I}}) 
h^{-(2+d)} e^{-br^{1/(2d+1)}}
\end{align*}
for the Bielak-MacCamy coupling. In the same way, there exists a blockwise rank-$r$
$\mathcal{H}$-matrix $\mathbf{B}_{\mathcal{H}}$  such that 
\begin{align*}
\left\|\mathbf{A}_{\rm sym}^{-1} -\mathbf{B}_{\mathcal{H}}
\right\|_2 \le C_{\rm  apx} C_{\rm sp} \operatorname*{depth}(\mathbb{T}_{\mathcal{I}}) 
h^{-(2+d)} e^{-br^{1/(3d+1)}}
\end{align*}
for the symmetric coupling and 
\begin{align*}
\left\|\mathbf{A}_{\rm jn}^{-1} -\mathbf{B}_{\mathcal{H}}
\right\|_2 \le C_{\rm  apx} C_{\rm sp} \operatorname*{depth}(\mathbb{T}_{\mathcal{I}}) 
h^{-(2+d)} e^{-br^{1/(6d+1)}}
\end{align*}
for the Johnson-N\'ed\'elec coupling. The constants 
$C_{\rm  apx}>0$ and $b>0$ depend only on 
$\Omega$, $d$, $\eta$, and the $\gamma$-shape regularity of the quasi-uniform triangulations 
$\mathcal{T}_h$ and $\mathcal{K}_h$.
\end{theorem}

\section{The Caccioppoli-type inequalities}\label{sec:proofsCacc}
In this section, we provide the proofs of the interior regularity estimates of
Theorems~\ref{th:CaccioppoliBMcC}--\ref{th:CaccioppoliJN}.

We start with some well-known facts about the volume potential operators $\widetilde{V}$, $\widetilde K$ and the boundary integral operators
$V,K,K',W$.
For details, we refer to \cite[Ch. 3]{SauterSchwab} and \cite[Ch. 6]{steinbach2007numerical}. 

\begin{itemize}
 \item With the interior trace operator $\gamma_0^{\rm int}$ (for $\Omega$) and exterior trace operator $\gamma_0^{\rm ext}$ (for $\R^d\backslash \overline{\Omega}$), we have 
\begin{align}\label{eq:tracePotentials}
&\gamma_0^{\rm int} \widetilde{V}\varphi = V\varphi = \gamma_0^{\rm ext} \widetilde{V}\varphi, \nonumber\\
&\gamma_0^{\rm int} \widetilde{K}u = (-1/2 + K)u \qquad
\text{and} \qquad \gamma_0^{\rm ext} \widetilde{K}u = (1/2 + K) u,
\end{align}
which implies the jump conditions across $\Gamma$
 \begin{align}
\label{eq:jump-relation-1}
[\gamma_0 \widetilde{V}\varphi] := 
\gamma_0^{\rm ext} \widetilde{V}\varphi-\gamma_0^{\rm int} \widetilde{V}\varphi = 0, \qquad 
[\gamma_0 \widetilde{K}u] = u.
 \end{align}
 \item Similarly, with the interior $\gamma_1^{\rm int} u := \gamma_0^{\rm int}\nabla u \cdot \nu$ 
and exterior conormal derivative 
$\gamma_1^{\rm ext} u := \gamma_0^{\rm ext}\nabla u \cdot \nu$ ($\nu$ is the 
outward normal vector of $\Omega$), we have 
\begin{align}\label{eq:conormalPotentials}
&\gamma_1^{\rm int} \widetilde{V}\varphi = (1/2+K')\varphi \qquad \text{and} \qquad \gamma_1^{\rm ext} \widetilde{V}\varphi = (-1/2 + K')\varphi, \nonumber\\
&\gamma_1^{\rm int} \widetilde{K}u = -Wu = \gamma_1^{\rm ext} \widetilde{K}u, 
\end{align}
and consequently the jump conditions
 \begin{align}
\label{eq:jump-relation-2}
[\gamma_1 \widetilde{V}\varphi] := 
\gamma_1^{\rm ext} \widetilde{V}\varphi-\gamma_1^{\rm int} \widetilde{V}\varphi = -\varphi, \qquad 
[\gamma_1 \widetilde{K}u] = 0.
 \end{align}
 \item The potentials $\widetilde{V} \varphi$ and $\widetilde{K}u$ are harmonic in $\R^d \backslash \Gamma$ and are bounded 
 operators (see \cite[Ch.~3.1.2]{SauterSchwab})
\begin{align}\label{eq:mapping-Ktilde}
\widetilde V : H^{-1/2+s}(\Gamma) \rightarrow H_{\rm loc}^{1+s}(\R^d), \qquad 
\widetilde K : H^{1/2+s}(\Gamma) \rightarrow H_{\rm loc}^{1+s}(\R^d\backslash\Gamma), \qquad \abs{s}\leq 1/2.
\end{align}
Consequently, we have the boundedness for the boundary integral operators as
\begin{align}\label{eq:mapping-K}
V : H^{-1/2+s}(\Gamma) \rightarrow H^{1/2+s}(\Gamma), \quad 
K : H^{1/2+s}(\Gamma) \rightarrow H^{1/2+s}(\Gamma), \quad 
W : H^{1/2+s}(\Gamma) \rightarrow H^{-1/2+s}(\Gamma)
\end{align}
for $s\in\R$ with $\abs{s}\leq 1/2$.
\end{itemize}

In the following, the notation $\lesssim$ abbreviates $\leq$ up to a 
constant $C>0$ which depends only on $\Omega$, the dimension $d$, and the
$\gamma$-shape regularity of $\mathcal{T}_h$. 
Moreover, we use $\simeq$ to indicate that both estimates
$\lesssim$ and $\gtrsim$ hold.

\subsection{The Bielak-MacCamy coupling}
This section is dedicated to the proof of Theorem~\ref{th:CaccioppoliBMcC}.
The techniques employed are fairly similar to \cite{FMP15,FMP16}, where Caccioppoli-type 
estimates for FEM and BEM are proven. Nonetheless, in the case of the FEM-BEM couplings, the 
additional terms in the bilinear forms arising from the coupling on the boundary 
need to be treated carefully.

We start with a classical approximation result, so-called super-approximation, see, e.g., 
\cite{nitsche-schatz74,wahlbin91}. 

\begin{lemma}\label{lem:superapprox}
Let $I_h^{\Gamma}:L^2(\Gamma) \rightarrow S^{0,0}(\mathcal{K}_h)$ 
be the $L^2(\Gamma)$-orthogonal projection. Then, there is $C > 0$ depending only on the 
shape-regularity of the triangulation and $\Gamma$ such that 
for any discrete function
$\psi_h \in S^{0,0}(\mathcal{K}_h)$ and any $\eta \in W^{1,\infty}(\Gamma)$ 
\begin{align}
 \norm{\eta\psi_h - I_h^{\Gamma}(\eta \psi_h)}_{H^{-1/2}(\Gamma)} \leq C h^{3/2} 
 \norm{\nabla \eta}_{L^{\infty}(\Gamma)}\norm{\psi_h}_{L^2(\Gamma\cap \operatorname*{supp}(\eta))}.
\end{align}
\end{lemma}
\begin{proof}
For details, we refer to \cite{FMP16}.
The main observation is that, on each element $K \in \mathcal{K}_h$, we have 
$\nabla \psi_h|_K \equiv 0$. Therefore, the standard approximation result reduces to
$$
 \norm{\eta\psi_h - I_h^{\Gamma}(\eta \psi_h)}_{L^2(K)} \lesssim h \norm{\nabla(\eta\psi_h)}_{L^2(K)}
 \lesssim h \norm{\nabla(\eta)\psi_h}_{L^2(K)}.
 $$
 Since $I_h^{\Gamma}$ is the $L^2$-projection, we obtain an additional factor $h^{1/2}$ in the 
 $H^{-1/2}(\Gamma)$-norm. 
\end{proof}

A similar super-approximation result holds for the nodal interpolation operator 
$I_h^{\Omega} : C(\overline{\Omega}) \rightarrow S^{1,1}(\mathcal{T}_h)$
\begin{equation}
\label{eq:super-approximation-Omega}
\norm{\eta v_h - I  _h^{\Omega}(\eta v_h)}_{H^k(\Omega)} \lesssim h^{2-k} \left( \norm{\nabla \eta}_{L^\infty(\Omega)} \norm{\nabla v_h}_{L^2(\Omega\cap \operatorname*{supp}(\eta))} 
+ \norm{D^2 \eta}_{L^\infty(\Omega)} \norm{v_h}_{L^2(\Omega\cap \operatorname*{supp}(\eta))}  \right)
\end{equation}
for any discrete function
$v_h \in S^{1,1}(\mathcal{T}_h)$, any $\eta \in W^{2,\infty}(\Omega)$, and $k=0,1$, where $H^{0}(\Omega):= L^2(\Omega)$.

In the proof of the Caccioppoli type inequality, we need the following inverse-type inequalities from 
\cite[Lem.~{3.8}]{FMP16} and \cite[Lem.~{3.6}]{FMP17}.
\begin{lemma}[ {\cite[Lem.~{3.8}]{FMP16}, \cite[Lem.~{3.6}]{FMP17}}]
\label{lem:investV}
Let $B_R \subset B_{R'}$ be concentric boxes with $\operatorname*{dist}(B_R,\partial B_{R'}) \geq 4h$. Then, for every $\psi_h \in S^{0,0}(\mathcal{K}_h)$, we have
\begin{align*}
\norm{\psi_h}_{L^2(B_R\cap \Gamma)} \lesssim h^{-1/2}\norm{\nabla \widetilde{V}\psi_h}_{L^2(B_{R'})}.
\end{align*}
Moreover,  for every $v_h \in S^{1,1}(\mathcal{T}_h)$, we have
	\begin{align}\label{eq:potentialK}
	\left\| \gamma_1 \widetilde{K}v _h \right\|_{L^{2}(B_R\cap\Gamma)}
	\lesssim h^{-1/2}\left(\left\| \nabla \widetilde{K}v _h \right\|_{L^{2}(B_{R'})}
	+\frac{1}{\operatorname*{dist}(B_R,\partial B_{R'})}\left\| \widetilde{K}v_h \right\|_{L^{2}(B_{R'})} \right). 
	\end{align}
\end{lemma}
Combining Lemma~\ref{lem:superapprox} with Lemma~\ref{lem:investV} 
(assuming $\operatorname*{supp}\eta \subset B_R$), we obtain estimates of the form
\begin{align}\label{eq:estErrorBEM}
 \norm{\eta\psi_h - I_h^{\Gamma}(\eta \psi_h)}_{H^{-1/2}(\Gamma)} \lesssim h 
 \norm{\nabla \eta}_{L^{\infty}(\Gamma)}\norm{\nabla \widetilde{V}\psi_h}_{L^2(B_{R'})}.
\end{align}

\begin{remark}
 An inspection of the proof of \eqref{eq:potentialK} (\cite[Lem.~{3.6}]{FMP17}) shows that the main observation is that 
 $\widetilde{K}v_h$ is harmonic. The remaining arguments therein only use mapping properties and jump conditions 
 for the potential $\widetilde K$ and can directly be modified such that the same result holds for the single-layer 
 potential as well, i.e., for every $\psi_h \in S^{0,0}(\mathcal{T}_h)$, we have
	\begin{align}\label{eq:potentialV}
	\left\| \gamma_1 \widetilde{V}\psi_h \right\|_{L^{2}(B_R\cap\Gamma)}
	\lesssim h^{-1/2}\left(\left\| \nabla \widetilde{V}\psi _h \right\|_{L^{2}(B_{R'})}
	+\frac{1}{\operatorname*{dist}(B_R,\partial B_{R'})}\left\| \widetilde{V}\psi_h \right\|_{L^{2}(B_{R'})} \right). 
	\end{align}
\end{remark}

Now, with the help of a local ellipticity result, the discrete variational formulation, 
and super-approximation, we
are able to prove Theorem~\ref{th:CaccioppoliBMcC}.

\begin{proof}[Proof of Theorem~\ref{th:CaccioppoliBMcC}]
In order to reduce unnecessary notation, we write $(u,\varphi)$ for the Galerkin solution $(u_h,\varphi_h)$.
The assumption on the support of the data implies the local orthogonality
\begin{align}\label{eq:orthoBMcC}
a_{\rm bmc}(u,\varphi;\psi_h,\zeta_h) = 0  \quad \forall (\psi_h,\zeta_h) \in S^{1,1}(\T_h) \times S^{0,0}(\mathcal{K}_h)
\quad \text{with}\quad \operatorname*{supp}\psi_h, \operatorname*{supp}\zeta_h \subset B_{(1+\varepsilon)R}.
\end{align}
Let  $\eta \in C_0^\infty(\R^d)$ be a cut-off function with 
$\text{supp}\, \eta \subseteq B_{(1+\delta/4)R}$, $\eta \equiv 1$ on $B_R$, 
$0 \le \eta \le 1$, and $ \left\| D^j \eta  \right\|_{L^\infty(B_{(1+\delta)R})} 
\lesssim \frac{1}{(\delta R)^j}$ for $j=1,2$. Here, $0<\delta \leq \varepsilon$ is such that $\frac{h}{R} \leq \frac{\delta}{8}$.
We note that this choice of $\delta$ implies that 
$\bigcup\{K \in \mathcal{K}_h \, : \, \operatorname*{supp}\eta \cap K \neq \emptyset\} \subset B_{(1+\delta/2)R}$.
In the final step of the proof, we will choose two different values for $\delta$ ($\leq \varepsilon$) depending on $\varepsilon$ - 
one of them, $\delta=\frac{\varepsilon}{2}$, explains the assumption made on $\varepsilon$ in the theorem. \\

\noindent
{\bf Step 1:} We provide a ``localized'' ellipticity estimate, i.e., we prove an inequality of the form
\begin{align*}
\left\| \nabla (\eta u ) \right\|_{L^2(\Omega)}^2 +
\left\| \nabla (\eta \widetilde{V}  \varphi) \right\|_{L^2(\mathbb{R}^d)}^2 
\lesssim a_{\rm bmc}(u,\varphi;\eta^2 u,\eta^2\varphi) \quad +
\quad \text{terms in weaker norms}.
\end{align*}
(See (\ref{eq:proofBMcC7b}) for the precise form.)
Since the ellipticity constant $C_{\rm ell}$ of $\mathbf{C}$ satisfies $C_{\rm ell}>1/4$, we may
choose a $\rho >0$ such that $1/4 < \rho/2 < C_{\rm ell}$.
This implies $C_{\rho}:=\min\{1-\frac{1}{2\rho},C_{\rm ell}-\frac{\rho}{2}\} > 0$, and 
we start with
\begin{align}\label{eq:proofBMcC1}
\left(C_{\rm ell}-\frac{\rho}{2}\right)\left\| \nabla (\eta u ) \right\|_{L^2(\Omega)}^2   +
\left(1-\frac{1}{2\rho}\right)\left\| \nabla (\eta \widetilde{V}  \varphi) \right\|_{L^2(\mathbb{R}^d)}^2 
&\le C_{\rm ell}\left\| \nabla (\eta u) \right\|_{L^2(\Omega)}^2+
\left\| \nabla (\eta  \widetilde{V}\varphi) \right\|_{L^2(\mathbb{R}^d)}^2 
 \nonumber \\ 
& \quad - \frac{1}{2\rho}\left\| \nabla (\eta \widetilde{V}\varphi) \right\|_{L^2(\Omega)}^2 -
\frac{\rho}{2}\left\| \nabla (\eta u) \right\|_{L^2(\Omega)}^2. 
\end{align}
Young's inequality implies
\begin{align}\label{eq:proofBMcC2}
-\frac{1}{2\rho}\left\| \nabla (\eta \widetilde{V} \varphi) \right\|_{L^2(\Omega)}^2 
-\frac{\rho}{2}\left\| \nabla (\eta u) \right\|_{L^2(\Omega)}^2 
&\le -\left\| \nabla (\eta \widetilde{V} \varphi) \right\|_{L^2(\Omega)}
\left\| \nabla (\eta u) \right\|_{L^2(\Omega)} \nonumber \\
& \le - \skp{\nabla (\eta \widetilde{V}\varphi), \nabla (\eta u )}_{L^2(\Omega)}.
\end{align}
Inserting \eqref{eq:proofBMcC2} into \eqref{eq:proofBMcC1} leads to 
\begin{align}\label{eq:proofBMcC3}
\nonumber
C_{\rho}\left\| \nabla (\eta u) \right\|_{L^2(\Omega)}^2 +
C_{\rho}\left\| \nabla (\eta \widetilde{V}\varphi)  \right\|_{L^2(\R^d)}^2   
&\le \left\| \nabla (\eta \widetilde{V}\varphi) \right\|_{L^2(\mathbb{R}^d)}^2
+C_{\rm ell}\left\| \nabla (\eta u)  \right\|_{L^2(\Omega)}^2 
\\
& \quad 
-  \skp{\nabla (\eta \widetilde{V}\varphi), \nabla (\eta u)}_{L^2(\Omega)}.
\end{align}
An elementary calculation shows 
\begin{align}\label{eq:proofBMcC4}
\nonumber
 \skp{\nabla (\eta \widetilde{V}\varphi), \nabla (\eta u)}_{L^2(\Omega)}
&=
 \skp{\nabla \widetilde{V}\varphi, \nabla (\eta^2 u)}_{L^2(\Omega)} \\
& \quad +\skp{(\nabla \eta) \widetilde{V}\varphi, \nabla (\eta u)}_{L^2(\Omega)}
  - \skp{ \nabla \widetilde{V}\varphi, \eta (\nabla \eta) u}_{L^2(\Omega)}.
 %\quad- \skp{ \nabla(\eta\widetilde{V}\varphi), (\nabla \eta) u}_{L^2(\Omega)} + 
 %\skp{ (\nabla \eta)\widetilde{V}\varphi, (\nabla \eta) u}_{L^2(\Omega)}. 
\end{align}
Since the single-layer potential is harmonic in $\Omega$, integration by parts (in $\Omega$) and 
$\gamma_1^{\rm int}\widetilde V = 1/2+K'$ lead to
\begin{align}\label{eq:proofBMcC5}
\skp{\nabla \widetilde{V}\varphi, \nabla (\eta^2 u)}_{L^2(\Omega)}
= \skp{\gamma_1^{\rm int} \widetilde{V}\varphi, \eta^2 u}_{L^2(\Gamma)}
= \skp{(1/2 +{K^\prime})\varphi, \eta^2 u}_{L^2(\Gamma)}.
\end{align}
Similarly, with integration by parts (in $\Omega$ and $\Omegaext$) 
and the jump condition of the single-layer potential we obtain
\begin{align}\label{eq:proofBMcC6}
\nonumber
\left\| \nabla  (\eta \widetilde{V}\varphi) \right\|_{L^2(\mathbb{R}^d)}^2 
&= \skp{\nabla \widetilde{V}\varphi, \nabla (\eta^2 \widetilde{V}\varphi)}_{L^2(\mathbb{R}^d)} 
+ \skp{\nabla \eta \widetilde{V}\varphi, \nabla \eta\widetilde{V} \varphi}_{L^2(\mathbb{R}^d)} \\ \nonumber 
&= - \skp{\left[ \gamma_1 \widetilde{V}\varphi\right] , \eta ^2 V\varphi}_{L^2(\Gamma)}
+ \skp{\nabla \eta \widetilde{V} \varphi, \nabla   \eta\widetilde{V} \varphi}_{L^2(  \mathbb{R}^d) }\\
&= \skp{V\varphi, \eta^2 \varphi}_{L^2(\Gamma)}
+ \skp{\nabla \eta \widetilde{V}\varphi, \nabla \eta\widetilde{V}\varphi}_{L^2(\mathbb{R}^d)}.
\end{align}
Moreover, the symmetry of $\mathbf{C}$ implies
\begin{align}\label{eq:proofBMcC7}
C_{\rm ell}\left\| \nabla  ( \eta u )  \right\|_{L^2(\Omega) }^2
\leq \skp{\mathbf{C}\nabla(\eta u), \nabla( \eta u )}_{L^2(\Omega) } 
= \skp{\mathbf{C}\nabla u, \nabla( \eta^2 u )}_{L^2(\Omega) } 
+ \skp{\mathbf{C}\nabla \eta u, \nabla   \eta u}_{L^2(\Omega) }.
\end{align}
Plugging \eqref{eq:proofBMcC4}--\eqref{eq:proofBMcC7} into \eqref{eq:proofBMcC3}, we infer
\begin{align}\label{eq:proofBMcC7a}
\nonumber
C_{\rho}\left\| \nabla (\eta u) \right\|_{L^2(\Omega)}^2+
C_{\rho}\left\| \nabla (\eta \widetilde{V}\varphi) \right\|_{L^2(\mathbb{R}^d)}^2 
&\le   \skp{\mathbf{C}\nabla u, \nabla (\eta^2 u)}_{L^2(\Omega)} + \skp{\mathbf{C}\nabla \eta u, \nabla \eta u}_{L^2(\Omega)}
 + \skp{ V\varphi, \eta^2 \varphi}_{L^2(\Gamma)} 
\\ \nonumber
&
\quad + \norm{\nabla\eta \widetilde{V}\varphi}_{L^2(\mathbb{R}^d)}^2
- \skp{(1/2+K')\varphi, \eta^2 u}_{L^2(\Gamma)} 
 \\
& \quad + \skp{\nabla \widetilde{V}\varphi, (\nabla\eta) \eta u}_{L^2(\Omega)}
- \skp{\nabla \eta \widetilde{V}\varphi, \nabla (\eta u)}_{L^2(\Omega)} \nonumber\\
& = a_{\rm bmc}(u,\varphi;\eta^2u,\eta^2\varphi) + \skp{\mathbf{C}\nabla \eta u, \nabla   \eta u}_{L^2(\Omega) }
+  \norm{\nabla\eta \widetilde{V}\varphi}_{L^2(\mathbb{R}^d)}^2 
 \nonumber \\
&\quad + \skp{\nabla \widetilde{V}\varphi, (\nabla\eta) \eta u}_{L^2(\Omega)}- \skp{\nabla \eta \widetilde{V}\varphi, \nabla (\eta u)}_{L^2(\Omega)}.
\end{align}
 Young’s inequality and $\norm{\nabla \eta}_{L^{\infty}(\R^d)}\lesssim \frac{1}{\delta R}$ imply
\begin{align}\label{eq:proofBMcCYoung1}
\nonumber 
\abs{\skp{\nabla \widetilde{V}\varphi, (\nabla\eta) \eta u}_{L^2(\Omega)}} &\leq 
\left| \skp{\nabla(\eta\widetilde{V} \varphi), \nabla  \eta u}_{L^2(\Omega)}\right| 
+ \abs{\skp{\nabla\eta \widetilde{V}\varphi, \nabla\eta u}_{L^2(\Omega)}} \\ \nonumber
&\le 
\left\|\nabla(\eta\widetilde{V} \varphi)\right\|_{L^2(\Omega)} \left\|\nabla \eta u \right\|_{L^2(\Omega)}
+\frac{C}{(\delta R)^2}\norm{u}_{L^2(B_{(1+\delta)R}\cap\Omega)}\norm{\widetilde{V}\varphi}_{L^2(B_{(1+\delta)R})}
\\
& \le \frac{C}{(\delta R)^2}\left(\left\| u \right\|^2_{L^2(B_{(1+\delta)R}\cap \Omega)} +
\norm{\widetilde{V}\varphi}_{L^2(B_{(1+\delta)R})}^2\right)
 +\frac{C_{\rho}}{4} \left\|  \nabla(\eta \widetilde{V}\varphi)\right\|_{L^2(\R^d)},
\end{align}
as well as
\begin{align}\label{eq:proofBMcCYoung2}
\nonumber \left| \skp{\nabla \eta \widetilde{V} \varphi, \nabla (\eta u)}_{L^2(\Omega) }\right| 
& \le \left\| \nabla\eta \widetilde{V}\varphi \right\|_{L^2(\Omega)}\left\| \nabla (\eta u) \right\|_{L^2(\Omega)} \\
& \le
\frac{2C}{(\delta R)^2}\left\| \widetilde{V}\varphi \right\|^2_{L^2(B_{(1+\delta)R})}
+ \frac{C_{\rho}}{4}\left\| \nabla (\eta u) \right\|^2_{L^2(\Omega)}. 
\end{align}
Absorbing the gradient terms in \eqref{eq:proofBMcCYoung1}-\eqref{eq:proofBMcCYoung2} 
in the left-hand side of \eqref{eq:proofBMcC7a}, we arrive at 
\begin{align}\label{eq:proofBMcC7b}
\left\| \nabla (\eta u) \right\|_{L^2(\Omega)}^2+
\left\| \nabla (\eta \widetilde{V}\varphi) \right\|_{L^2(\mathbb{R}^d)}^2 
  & \lesssim 
 a_{\rm bmc}(u,\varphi;\eta^2u,\eta^2\varphi) \nonumber \\
 &\quad + \frac{1}{(\delta R)^2} \norm{\widetilde{V}\varphi}_{L^2(B_{(1+\delta)R})}^2 
+  \frac{1}{(\delta R)^2} \left\| u\right\|_{L^2(B_{(1+\delta)R}\cap \Omega)}^2.
\end{align}

\noindent
{\bf Step 2:} We apply the local orthogonality of $(u,\varphi)$ 
to piecewise polynomials and use approximation properties.

Let $I_h^{\Omega}: C(\overline{\Omega})  \rightarrow S^{1,1}(\T_h)$ 
be the nodal interpolation operator and $I_h^{\Gamma}$ the $L^2(\Gamma)$-orthogonal projection mapping onto
$S^{0,0}(\mathcal{K}_h)$. Then, the orthogonality \eqref{eq:orthoBMcC} leads to
\begin{align}\label{eq:proofBMcC8}
\nonumber a_{\rm bmc}(u,\varphi;\eta^2u,\eta^2\varphi) &= 
  a_{\rm bmc}(u,\varphi;\eta^2u - I_h^{\Omega}(\eta^2 u),\eta^2\varphi-I_h^{\Gamma}(\eta^2 \varphi)) \\ \nonumber
  &= \skp{\mathbf{C}\nabla u, \nabla( \eta^2 u - I_h^{\Omega}(\eta^2 u)) }_{L^2(\Omega)}
+ \skp{(1/2-K')\varphi, I_h^{\Omega}(\eta^2 u) - \eta^2 u}_{L^2(\Gamma)} \\ 
&\quad+
  \skp{V\varphi, \eta^2 \varphi - I_h^{\Gamma}(\eta^2 \varphi)}_{L^2(\Gamma)} 
 - \skp{u, I_h^{\Gamma}(\eta^2 \varphi) -\eta^2 \varphi}_{L^2(\Gamma)} \nonumber \\
 &=: T_1+T_2+T_3+T_4.
\end{align}
The term $T_3$ can be estimated in exactly the same way as in \cite{FMP16}.
More explicitly, we need a second cut-off function $\widetilde{\eta}$ 
with $0 \le \widetilde{\eta} \le 1$, $\widetilde{\eta} \equiv 1$ on 
$B_{(1+\delta/2)R}\supseteq \operatorname*{supp} (\eta^2 \varphi - I_h^{\Gamma}(\eta^2 \varphi))$,  
$\operatorname*{supp} \widetilde{\eta} \subseteq \overline{B_{(1+\delta)R}}$ 
and $ \left\| \nabla \widetilde{\eta} \right\|_{L^\infty(B_{(1+\delta)R})} \lesssim \frac{1}{\delta R}$. 
Here, the support property of $I_h^{\Gamma}(\eta^2\varphi)$ follows from the assumption on $\delta$.
The trace inequality together with
the super-approximation properties of $I_h^{\Gamma}$, expressed in \eqref{eq:estErrorBEM}, lead to
\begin{align}\label{eq:proofBMcC9}
\left|\skp{V\varphi, \eta^2 \varphi - I_h^{\Gamma}(\eta^2 \varphi)}_{L^2(\Gamma)}\right| & =
\left|\skp{\widetilde{\eta}V\varphi, \eta^2 \varphi - I_h^{\Gamma}(\eta^2 \varphi)}_{L^2(\Gamma)}\right|
\le \left\|\widetilde{\eta} {V}\varphi \right\|_{H^{1/2}(\Gamma)} 
\left\|\eta^2 \varphi -I_h^{\Gamma}( \eta^2\varphi ) \right\|_{H^{-1/2}(\Gamma)} \nonumber \\
& 
\lesssim 
\left\|\widetilde{\eta} \widetilde{V}\varphi \right\|_{H^{1}(\Omega)} \frac{h}{\delta R}
\left\| \nabla \widetilde{V}\varphi \right\|_{L^2(B_{(1+\delta)R})} \nonumber \\
&\lesssim 
\frac{h}{\delta R} \left\| \nabla \widetilde{V}\varphi \right\|^2_{L^2(B_{(1+\delta)R})} +
\frac{1}{(\delta R)^2} \left\| \widetilde{V}\varphi \right\|^2_{L^2(B_{(1+\delta)R})} .
\end{align}
With the same arguments, we obtain an estimate for $T_4$
\begin{align}\label{eq:proofBMcC10}
\left|\skp{u, I_h^{\Gamma}(\eta^2 \varphi)\! - \!\eta^2 \varphi}_{L^2(\Gamma)} \right| 
& \lesssim  
\frac{h}{\delta R} \left\| \nabla u \right\|^2_{L^2(B_{(1+\delta)R}\cap\Omega)} 
+\frac{h}{\delta R} \left\|\nabla \widetilde{V}\varphi \right\|^2_{L^2(B_{(1+\delta)R})} + 
\frac{1}{(\delta R)^2} \left\| u \right\|^2_{L^2(B_{(1+\delta)R}\cap\Omega)} .
\end{align}
The volume term $T_1$ in \eqref{eq:proofBMcC8} 
can be estimated as in \cite{FMP15}. 
Here, the super-approximation properties of $I_h^{\Omega}$ from \eqref{eq:super-approximation-Omega}, 
 Young’s inequality, and $\frac{h}{\delta R} \le 1$ lead to
\begin{align}\label{eq:proofBMcC11}
\nonumber &\left|\skp{\mathbf{C}\nabla u, \nabla (\eta^2 u -I_h^{\Omega}(\eta^2 u))}_{L^2(\Omega) } \right|  
\le
 \left\|\mathbf{C} \nabla u \right\|_{L^2(B_{(1+\delta)R}\cap\Omega)} \left\|\nabla (\eta^2 u -I_h^{\Omega}(\eta^2 u)) \right\|_{L^2(\Omega)}
 \\ \nonumber
 & \qquad \qquad
\lesssim   \left\| \nabla  u \right\|_{L^2(B_{(1+\delta)R}\cap\Omega)}\left(
\frac{h}{(\delta R)^2}\left\| u \right\|_{L^2(B_{(1+\delta)R}\cap\Omega)} \right. 
\left. +\frac{h}{ \delta R} \left\| \eta \nabla u\right\|_ {L^2(B_{(1+\delta)R}\cap\Omega)}\right )   \\
&
\qquad \qquad\le C\frac{h^2}{(\delta R)^2}\left\| \nabla u \right\|^2_{L^2(B_{(1+\delta)R}\cap\Omega)}
+\frac{1}{4}\left\| \eta \nabla u\right\|^2_{L^2(B_{(1+\delta)R}\cap\Omega)}.
\end{align}
It remains to treat 
the coupling term $T_2$ involving the adjoint double-layer operator in \eqref{eq:proofBMcC8}. 
With the support property $\operatorname*{supp}(I_h^{\Omega}(\eta^2 u) - \eta^2 u) \subset B_{(1+\delta/2)R}$,
which follows from $8h \leq \delta R$, and 
$(1/2-K')\varphi = -\gamma_1^{\rm ext} \widetilde{V}\varphi$, we obtain
% We define a third cut-off function $\widehat{\eta}$ 
% with $0 \le \widehat{\eta} \le 1$, $\widehat{\eta} \equiv 1$ on 
% $\operatorname*{supp} (I_h^{\Omega}(\eta^2 u) - \eta^2 u)$,  
% $\operatorname*{supp} \widehat{\eta} \subseteq \overline{B_{(1+\delta)R}}$ 
% and $ \left\| \nabla \widehat{\eta} \right\|_{L^\infty(B_{(1+\delta)R})} \lesssim \frac{1}{\delta R}$.  
% Then, we have
\begin{align}\label{eq:proofBMcC12}
\abs{\skp{ (1/2-K')\varphi, I_h^{\Omega}(\eta^2 u) - \eta^2 u}_{L^2(\Gamma)}}
&\le 
\left\|\gamma_1^{\rm ext} \widetilde{V}\varphi \right\|_{L^{2}(B_{(1+\delta/2)R}\cap\Gamma)} 
\left\| I_h^{\Omega}(\eta^2 u) - \eta^2 u \right\|_{L^{2}(\Gamma)}.
\end{align}
%%%%%%%%%%%%%%%%%%%%%%%%%%%%%%%%%%%%%%%%%%%%%%Ü%%%%%%%%%%%%%%%%%%%%%%%
The multiplicative trace inequality for
$\Omega$, see, e.g., \cite{BrennerScott}, the super-approximation property 
of $I_h^{\Omega}$ from \eqref{eq:super-approximation-Omega}, and $ \frac{h}{R} \le \frac{\delta}{8}$ lead to
(see also \ \cite[Eq.~(25), (26)]{FMP15} for more details) 
 \begin{align}\label{eq:proofBMcC13}
 \nonumber \left\| I_h^{\Omega}(\eta^2 u) - \eta^2 u \right\|_{L^{2}(\Gamma)} & \le 
  \left\| I_h^{\Omega}(\eta^2 u) - \eta^2 u \right\|_{L^2(\Omega)} 
  + \left\| I_h^{\Omega}(\eta^2 u) - \eta^2 u \right\|^{1/2}_{L^{2}(\Omega)}
   \left\| \nabla ( I_h^{\Omega}(\eta^2 u) - \eta^2 u) \right\|^{1/2}_{L^{2}(\Omega)}  \\
   & \lesssim
 \frac{h^{3/2}}{\delta R}  \left\| \nabla u \right\|_{L^2(B_{(1+ \delta)R}\cap\Omega)}
  + \frac{h^{3/2}}{(\delta R)^2} \left\| u \right\|_{L^2(B_{(1+\delta)R}\cap\Omega)}.
  \end{align}
We use estimate
	\eqref{eq:potentialV} and \eqref{eq:proofBMcC13} in \eqref{eq:proofBMcC12}, which implies
 \begin{align}\label{eq:proofBMcC15}
\nonumber \left| \skp{(1/2-K')\varphi, I_h^{\Omega}(\eta^2 u) - \eta^2 u}_{L^2(\Gamma)}\right| 
& \lesssim h ^{{-1/2}} \left(\left\| \nabla \widetilde{V}\varphi \right\|_{L^{2}(B_{(1+\delta)R})} 
+\frac{1}{\delta R} \left\| \widetilde{V}\varphi \right\|_{L^{2}(B_{(1+\delta)R})} \right) \cdot
\\ \nonumber 
 & \quad \left( \frac{h^{3/2}}{\delta R} \left\| \nabla u \right\|_{L^2(B_{(1+\delta)R}\cap\Omega)}
 + \frac{h^{3/2}}{(\delta R)^2} \left\| u \right\|_{L^2(B_{(1+\delta)R}\cap\Omega)} \right) 
\\ \nonumber
  & \lesssim \frac{h}{\delta R} \Bigl\{ \Big( \left\| \nabla \widetilde{V}\varphi \right\|^2_{L^{2}(B_{(1+\delta)R})}     
+ \left\| \nabla u \right\|^2_{L^2(B_{(1+\delta)R}\cap\Omega)} \Big)\\
& \quad+
 \frac{1}{(\delta R)^2} \Big( \left\| u \right\|^2_{L^2(B_{(1+\delta)R}\cap\Omega)}
 +\left\| \widetilde{V}\varphi \right\|^2_{L^2(B_{(1+\delta)R})} \Big) \Bigr\}. 
 \end{align}
 Finally, inserting \eqref{eq:proofBMcC9}, \eqref{eq:proofBMcC10}, \eqref{eq:proofBMcC11}, and  
 \eqref{eq:proofBMcC15} into \eqref{eq:proofBMcC8} and further into \eqref{eq:proofBMcC7b},
 and absorbing the term $\frac{1}{4}\left\| \eta \nabla u\right\|^2_{L^2(B_{(1+\delta)R})}$ on the 
 left-hand side implies
\begin{align}\label{eq:proofBMcC16}
\nonumber
\left\| \nabla u \right\|_{L^2(B_R\cap\Omega)}^2 +
\left\| \nabla \widetilde{V}\varphi \right\|_{L^2(B_R)}^2  
&\leq
\left\| \nabla (\eta  u) \right\|_{L^2(\Omega)}^2
+\left\| \nabla (\eta \widetilde{V}\varphi) \right\|_{L^2( \mathbb{R}^d)}^2  
  \\ \nonumber
&\lesssim \frac{h}{\delta R} \left( 
\left\| \nabla u \right\|^2_{L^2(B_{(1+\delta)R}\cap\Omega)}+
\left\| \nabla \widetilde{V}\varphi \right\|^2_{L^{2}(B_{(1+\delta)R})}     
 \right)\\
& \quad +
\frac{1}{(\delta R)^2} \left( \left\| u \right\|^2_{L^2(B_{(1+\delta)R}\cap\Omega)}
+\left\| \widetilde{V}\varphi \right\|^2_ {L^2(B_{(1+\delta)R})} \right). 
\end{align}

\noindent
{\bf Step 3:} We iterate \eqref{eq:proofBMcC16} to obtain the claimed powers of $h$ for the gradient terms.

We set $\delta = \frac{\varepsilon}{2}$, and 
use \eqref{eq:proofBMcC16} again for the gradient terms on the right-hand side with the boxes 
$B_{\widetilde{R}}$ and $B_{(1+\widetilde{\delta})\widetilde{R}}$, where
$\widetilde{\delta} = \frac{\varepsilon}{\varepsilon +2}$ and $\widetilde{R} = (1+\varepsilon/2)R$. 
We note that $16h\leq \varepsilon R$ implies $8h\leq \widetilde{\delta} \widetilde{R}$, so we may apply \eqref{eq:proofBMcC16}.
Considering $(1+\widetilde{\delta})(1+\frac{\varepsilon}{2}) = 1+\varepsilon$, we get
\begin{align}\label{eq:proofBMcC17}
\left\| \nabla u \right\|_{L^2(B_R\cap\Omega)}^2+
\left\| \nabla \widetilde{V}\varphi \right\|_{L^2(B_R)}^2  
 &\lesssim 
 \frac{h^2}{ (\varepsilon R)^2 }\left( 
 \left\| \nabla u \right\|^2_{L^{2}(B_{(1+\varepsilon)R}\cap \Omega)}+
 \left\| \nabla \widetilde{V}\varphi \right\|^2_{L^{2}(B_{(1+\varepsilon)R})} 
  \right)
\nonumber \\
&\quad  + \left(\frac{h}{(\varepsilon R )^3 }+ \frac{1}{(\varepsilon R )^2 }\right)
  \left(\left\| u \right\|^2_ {L^2(B_{(1+\varepsilon)R}\cap\Omega)}
 +\left\| \widetilde{V}\varphi \right\|^2_ {L^2(B_{(1+\varepsilon)R})}\right),
\end{align}
and with
$\frac{h}{\varepsilon R} <1$,
we conclude the proof.
\end{proof}

\subsection{The symmetric coupling}
In this section, we provide the proof of Theorem~\ref{th:CaccioppoliSymm}. While some 
parts of the proof are similar to the proof of Theorem~\ref{th:CaccioppoliBMcC} and 
are therefore shortened, there are some differences as well, mainly that it does not suffices to 
study the single-layer potential. Indeed, one has to add a term containing the double-layer potential 
to the Caccioppoli inequality in order to get a localized ellipticity estimate.  

\begin{proof}[Proof of Theorem~\ref{th:CaccioppoliSymm}]
Again, we write $(u,\varphi)$ for the Galerkin solution $(u_h,\varphi_h)$. The assumption on the support of the data implies the local orthogonality
\begin{align}\label{eq:orthosym}
a_{\rm sym}(u,\varphi;\psi_h,\zeta_h) = 0  \quad 
\forall (\psi_h,\zeta_h) \in S^{1,1}(\T_h) \times S^{0,0}(\mathcal{K}_h)
\quad \text{with}\quad \operatorname*{supp}\psi_h, \operatorname*{supp}\zeta_h \subset B_{(1+\varepsilon)R}.
\end{align}

\noindent

As in the proof of Theorem~\ref{th:CaccioppoliBMcC}
let $\eta \in C_0^\infty(\R^d)$ be a cut-off function with 
$\text{supp}\, \eta \subseteq B_{( 1+\delta/4) R}$, 
$\eta \equiv 1$ on $B_R $, $0 \le \eta \le 1$, 
and $ \left\| D^j \eta \right\|_{L^\infty(B_{(1+\delta)R})} \lesssim \frac{1}{\delta R}$ for $j=1,2$. 
Here, $0<\delta\leq \varepsilon$ is given such that $\frac{h}{R}\leq \frac{\delta}{16}$ 
and will be chosen in the last step of the proof.

%where $\operatorname*{supp} f, \operatorname*{supp}v_0, \operatorname*{supp} (1/2-K)u_0 \cap B_{(1+\delta)R} = \emptyset$.\\

\noindent
{\bf Step 1:} We start with a local ellipticity estimate. More precisely, we show 
\begin{align*}
 \left\|\nabla (\eta u) \right\|^2_{L^2(\Omega)} 
 &+\left\|\nabla (\eta \widetilde{V} \varphi) \right\|^2_{L^2(\R^d) }
 +\left\|\nabla (\eta \widetilde{K} u) \right\|^2_{L^2(\R^d\setminus\Gamma)} 
 \le a_{\rm sym}(u, \varphi; \eta ^2 u , \eta ^2 \varphi)+ {\rm  terms \, in \, weaker \, norms}.
\end{align*}
(See (\ref{eq:proofCSC3}) for the precise statement.)
%Employing \eqref{eq:proofBMcC2} - since no condition on $C_{\rm ell}$ is imposed, we may set 
%$\rho = 1$ -
%for $\widetilde{K}u$ instead of $u$ as well as 
{}From \eqref{eq:proofBMcC7} and the Cauchy-Schwarz inequality we get 
\begin{align}\label{eq:proofCSC1}
\nonumber C_{\rm ell} \left\|\nabla (\eta u) \right\|^2_{L^2(\Omega)} 
&+\frac{1}{2}\left\|\nabla (\eta \widetilde{V} \varphi) \right\|^2_{L^2(\R^d) }
+\frac{1}{2}\left\|\nabla (\eta \widetilde{K} u) \right\|^2_{L^2(\R^d\setminus\Gamma)} \\
&\qquad \leq \skp{\mathbf{C}\nabla u,\nabla(\eta^2 u)}_{L^2(\Omega)}+ 
\skp{\mathbf{C}\nabla\eta u,\nabla\eta u}_{L^2(\Omega)}  + \left\|\nabla (\eta \widetilde{V} \varphi) \right\|^2_{L^2(\R^d) }
\nonumber\\ 
& \qquad\quad
 +\left\|\nabla (\eta \widetilde{K} u) \right\|^2_{L^2(\R^d\setminus\Gamma)} 
-\skp{\nabla (\eta \widetilde{V} \varphi),\nabla (\eta \widetilde{K} u)}_{L^2(\mathbb{R}^d\setminus\Gamma)}. 
\end{align}
A direct calculation reveals that 
$\|\nabla (\eta\widetilde{K} u)\|^2_{L^2(\R^d\setminus\Gamma)} = \|(\nabla \eta) \widetilde{K} u\|^2_{L^2(\R^d\setminus\Gamma)} + \skp{\nabla \widetilde{K} u, \nabla (\eta ^2 \widetilde{K}u)}_{L^2(\R^d\setminus\Gamma)}$. 
Inserting this and \eqref{eq:proofBMcC6} in 
(\ref{eq:proofCSC1}) yields
%Using \eqref{eq:proofBMcC7} for $\widetilde{K}u$ instead of $u$ and \eqref{eq:proofBMcC6} for 
%$\left\|\nabla (\eta \widetilde{V} \varphi) \right\|^2_{L^2(\R^d) }$, we obtain
\begin{align}
\nonumber C_{\rm ell}\left\|\nabla (\eta u) \right\|^2_{L^2(\Omega)} 
&+\frac{1}{2}\left\|\nabla (\eta \widetilde{V} \varphi) \right\|^2_{L^2(\R^d) }
+\frac{1}{2}\left\|\nabla (\eta \widetilde{K} u) \right\|^2_{L^2(\R^d\setminus\Gamma)}
\\
&\qquad \le \skp{\mathbf{C}\nabla u, \nabla(\eta^2 u)}_{L^2(\Omega)} +\skp{\mathbf{C}\nabla\eta u,\nabla\eta u}_{L^2(\Omega)} 
 +  \skp{V\varphi, \eta^2 \varphi}_{L^2(\Gamma)} \nonumber \\ 
&\qquad\quad + \norm{(\nabla \eta) \widetilde{V}\varphi}_{L^2(\R^d)}^2
+ \skp{\nabla \widetilde{K} u, \nabla (\eta ^2 \widetilde{K}u)}_{L^2(\R^d\setminus\Gamma)} 
+ \norm{(\nabla \eta) \widetilde{K}u}_{L^2(\R^d\setminus\Gamma)}^2 \nonumber \\  
&\qquad\quad-\skp{\nabla (\eta \widetilde{V} \varphi),\nabla (\eta \widetilde{K} u)}_{L^2(\mathbb{R}^d\setminus\Gamma)}.
\end{align}
Integration by parts together with the jump conditions (\ref{eq:jump-relation-1}), \eqref{eq:jump-relation-2} for the double-layer potential gives
\begin{align}
  \skp{\nabla \widetilde{K} u, \nabla (\eta ^2 \widetilde{K}u)}_{L^2(\R^d\setminus\Gamma)} = \skp{Wu,\eta^2 u}_{L^2(\Gamma)}.
\end{align}
With a calculation analogous to \eqref{eq:proofBMcC4} (in fact, replace $u$ there with $\widetilde{K} u$), we get  
\begin{align*}
\skp{\nabla (\eta \widetilde{V} \varphi),\nabla (\eta \widetilde{K} u)}_{L^2(\mathbb{R}^d\setminus\Gamma)}  =
\skp{\nabla (\widetilde{V} \varphi),\nabla (\eta^2 \widetilde{K} u)}_{L^2(\mathbb{R}^d\setminus\Gamma)} + \text{l.o.t.},
\end{align*}
where the omitted terms (cf.~(\ref{eq:proofBMcC4})) 
\begin{equation*}
\text{l.o.t.} = 
\langle (\nabla \eta) \widetilde{V} \varphi, \nabla (\eta\widetilde K u)\rangle_{L^2(\mathbb{R}^d\setminus\Gamma)} - \langle \nabla \widetilde{V} \varphi,\eta (\nabla \eta) \widetilde K u\rangle_{L^2(\mathbb{R}^d\setminus\Gamma)}
\end{equation*}
can be estimated in weaker norms (i.e., $\|\widetilde{V}\varphi\|_{L^2(B_{(1+\delta/2)R})}$, $\|\widetilde K u\|_{L^2(B_{(1+\delta/2)R}\setminus \Gamma)}$)
or lead to terms that are absorbed in the left-hand side
as in the proof of Theorem~\ref{th:CaccioppoliBMcC} (see \eqref{eq:proofBMcCYoung1}, \eqref{eq:proofBMcCYoung2}).
With integration by parts on $\Omega$ and $\Omegaext$, we get
\begin{align}\label{eq:proofCSC3a}
\nonumber\skp{\nabla  \widetilde{V} \varphi,\nabla (\eta ^2\widetilde{K} u)}_{L^2(\mathbb{R}^d\setminus\Gamma)}
&=\skp{\gamma_1^{\text{int}}   \widetilde{V} \varphi,\gamma _0 ^ {\text{int}}  (\eta ^2\widetilde{K} u)}_{L^2(\Gamma)}
-\skp{\gamma_1 ^ {\ext}  \widetilde{V} \varphi,\gamma _0 ^ {\ext}  (\eta ^2\widetilde{K} u)}_{L^2(\Gamma)}\\
\nonumber
&=\skp{(K^\prime + 1/2) \varphi, \eta ^2(K- 1/2)u}_{L^2(\Gamma)}-\skp{(K^\prime - 1/2) \varphi, \eta ^2(K+ 1/2)u}_{L^2(\Gamma)}  \\
 &=\skp{\eta ^2\varphi,(K-1/2)u}_{L^2(\Gamma)}-\skp{(K^\prime- 1/2)\varphi, \eta ^2u}_{L^2(\Gamma)}.
\end{align}
Putting everything together and using $\norm{\nabla \eta}_{L^{\infty}(B_{(1+\delta)R})} 
\lesssim \frac{1}{\delta R}$, we obtain
\begin{align}\label{eq:proofCSC3}
\nonumber \left\|\nabla (\eta u) \right\|^2_{L^2(\Omega)} 
+\left\|\nabla (\eta \widetilde{V} \varphi) \right\|^2_{L^2(\R^d) }
&+\left\|\nabla (\eta \widetilde{K} u) \right\|^2_{L^2(\R^d\setminus\Gamma)} \lesssim 
a_{\rm sym}(u, \varphi, \eta ^2 u , \eta ^2 \varphi) + \frac{1}{(\delta R)^2} \norm{u}_{L^2(B_{(1+\delta)R}\cap\Omega)}^2
\\&\qquad\quad+ \frac{1}{(\delta R)^2} \norm{\widetilde{V}\varphi}_{L^2(B_{(1+\delta)R})}^2   
+ \frac{1}{(\delta R)^2} \norm{\widetilde{K}u}_{L^2(B_{(1+\delta)R}\setminus\Gamma)}^2. 
\end{align}
 %%%%%%%%%%%%%%%%%%%%%%%%%%%%%%%%%%%%%%%%%%%%%%%%%%%
 
 \noindent {\bf Step 2:} Applying the local orthogonality as well as approximation results.
 
 With the $L^2(\Gamma)$-orthogonal projection $I_h^{\Gamma}$ and the 
 nodal interpolation operator $I_h^{\Omega}$, the orthogonality \eqref{eq:orthosym} implies
\begin{align}\label{eq:proofCSC4}
\nonumber a_{\rm sym}(u, \varphi; \eta ^2 u , \eta ^2 \varphi)&= 
a_{\rm sym}(u, \varphi; \eta ^2 u - I_h^{\Omega}(\eta^2 u), \eta ^2 \varphi - I_h^{\Gamma}(\eta^2 \varphi))\\
&=\skp{\mathbf{C} \nabla u, \nabla (\eta^2 u - I_h ^ \Omega(\eta^2 u))}_{L^2(\Omega)} 
+\skp{Wu, \eta^2 u - I_h ^ \Omega(\eta^2 u)}_{L^2(\Gamma)} \nonumber \\
\nonumber
&\quad +\skp{(K^\prime - 1/2)\varphi, \eta^2 u- I_h ^ \Omega(\eta^2 u)}_{L^2(\Gamma)}
+\skp{V\varphi, \eta^2\varphi - I_h ^ \Gamma(\eta ^2 \varphi)}_{L^2(\Gamma)}\\ 
& \quad
+\skp{(1/2-K)u, \eta^2 \varphi - I_h ^ \Gamma(\eta^2 \varphi)}_{L^2(\Gamma)}\nonumber 
\\ &=: T_1 + T_2 + T_3 + T_4 + T_5.
\end{align}
The terms $T_1$, $T_3$, $T_4$ can be estimated with \eqref{eq:proofBMcC11}, \eqref{eq:proofBMcC15} and 
\eqref{eq:proofBMcC9} respectively as in the case for the Bielak-MacCamy coupling. Therefore, it remains 
to estimate $T_2$ and $T_5$.

We start with $T_2$, which can be treated in the same way as in \cite{FMP17}. 
In fact, with techniques similar to \eqref{eq:proofBMcC9}, 
the proof of \cite[Lem.~3.8]{FMP17} (taking $v = \widetilde{K} u$ there and noting that \cite[Lemma~{3.6}]{FMP17} is employed,
which does not impose orthogonality conditions on $v$) provides the estimate
% \begin{align}\label{eq:proofCSC6}
% \nonumber\left| \skp{Wu, \eta^2 u - I_h ^ \Gamma(\eta^2 u)}_{L^2(\Gamma)}\right|
% &=\left| \skp{\widetilde{\eta}_1 Wu, \eta^2 u - I_h ^ \Gamma(\eta^2 u)}_{L^2(\Gamma)}\right| \\ \nonumber
% &\le \left\|\widetilde{\eta}_1 Wu\right\|_{H^{-1/2}(\Gamma)}\left\|\eta^2 u - I_h ^ \Gamma (\eta^2 u)\right\|_{H^{1/2}(\Gamma)}
% \\&
% \le \left\| \widetilde{\eta}_1 \widetilde{K}u\right\|_{H^1(\Omega)}\left\|\eta^2 u - I_h ^ \Gamma (\eta^2 u)\right\|_{H^{1/2}(\Gamma)}.
% \end{align}
% Moreover, we can define the function $\widetilde{\eta}_1 u \in H^1 ( \Omega) $ 
% with the support property $\text{suup}\,\widetilde{\eta}_1u \subset \overline{B_{( 1+\delta)R }}$, 
% \cite{bibid} and therefore
% \begin{align}\label{eq:proofCSC7}
% \left\| \widetilde{\eta}_1 \widetilde{K}u\right\| _{H ^ {1}( \Omega ) } \le \frac{1}{\delta R} \left\|  \widetilde{K}u\right\| _{L ^ {2}( B_{ ( 1+\delta )R }) }+\left\|  \nabla \widetilde{K}u\right\| _{L ^ {2}( B_{ ( 1+\delta )R })}.
% \end{align}
% The trace inequality  for $\Omega$ and Lemma \ref{l3} imply
% \begin{align}\label{eq:proofCSC8}
% \left\|\eta ^2 u - I_h ^ \Gamma (\eta ^2 u )\right\| _{H^{1/2}(\Gamma) } & \lesssim\left\|\eta ^2 u - I_h ^ \Gamma (\eta ^2 u )\right\| _{H ^ {1}( \Omega ) } 
% \\[3 mm]\nonumber
% &
% \lesssim \frac{h+h^2}{(\delta R )^2 } \left\|  u\right\| _{L ^ {2}( B_{ ( 1+\delta )R }) }+\frac{h+h^2}{\delta R  } \left\|\eta   \nabla u\right\| _{L ^ {2}( B_{ ( 1+\delta )R })}.
% \end{align}
% Inserting the estimate \eqref{eq:proofCSC7} and \eqref{eq:proofCSC8} into \eqref{eq:proofCSC6} lead to
\begin{align*}
\left| \skp{Wu, \eta^2 u - I_h ^ \Gamma (\eta^2 u)}_{L^2(\Gamma)}\right| &\le  C
\left( \frac{ h^2}{(\delta R)^2} \left\| \nabla \widetilde{K}u \right\| ^2_ {L^2(B_{(1+\delta)R}\setminus\Gamma)}+
\frac{1}{(\delta R)^2} \left\| \widetilde{K}u \right\|^2_{L^2(B_{(1+\delta)R}\setminus\Gamma)} \right) \\
&\quad+\frac{1}{4}\left\| \nabla(\eta\widetilde{K}u)\right\| ^2_ {L^2(B_{(1+\delta)R}\setminus\Gamma)}.
\end{align*}
We note that \cite[Lemma~3.8]{FMP17} imposes the condition $16h\leq \delta R$.

We finish the proof by estimating $T_5$.
To that end, we need another cut-off function $\widetilde{\eta} \in S^{1,1}(\mathcal{T}_h )$  
with $0 \le \widetilde{\eta} \le 1$, $\widetilde{\eta} \equiv 1$ 
on $B_{(1+\delta/2)R}\supseteq\text{supp} \left( I_h ^ \Gamma( \eta ^2 \varphi  ) -  \eta ^2 \varphi \right)  $,  
$\text{supp}\,\widetilde{\eta} \subseteq  B_{(1+ \delta)R}  $ and 
$ \left\| \nabla \widetilde{\eta} \right\|_{L^\infty (B_{( 1+\delta) R} ) } \lesssim \frac{1}{\delta R}$.  
Since $( 1/2-K)u= - \gamma _0^{\rm int} \widetilde{ K} u $, 
we get with a trace inequality and the approximation properties expressed in \eqref{eq:estErrorBEM} that
\begin{align}\label{eq:proofCSC11}
\nonumber
\abs{T_5}&=\left| \skp{\widetilde{\eta}\gamma _0^{\rm int} \widetilde{ K} u, \eta^2 \varphi - I_h ^ \Gamma (\eta ^2 \varphi ) }_{L^2(\Gamma)}\right| 
\lesssim \left\| \gamma_0^{\rm int}(\widetilde{\eta} \widetilde{K} u)\right\|_{H^{1/2}(\Gamma)}
\left\|\eta^2 \varphi - I_h^\Gamma (\eta^2 \varphi ) \right\|_{H^{-1/2}(\Gamma)}  
\\ & \nonumber \lesssim 
\frac{h}{\delta R} \left\| \widetilde{\eta} \widetilde{K} u\right\|_{H^{1}(\Omega\setminus\Gamma)}
\left\| \nabla \widetilde{V} \varphi \right\| _{L^2(B_{(1+\delta) R})}
\\ &  \lesssim 
\frac{h}{\delta R}\left(\left\|\nabla  \widetilde{K} u  \right\|^2_{L^2( B_{\left(1+ \delta \right)R} \setminus\Gamma   ) }  
+\left\|  \nabla \widetilde{ V} \varphi \right\| ^2_{L^2(B_{(1+{\delta})R})}\right)   
+ \frac{1}{( \delta R) ^2}  \left\|  \widetilde{K}u \right\|^2_ {L^2( B_{(1+ \delta)R }\setminus\Gamma) }.
\end{align}
Putting everything together in \eqref{eq:proofCSC4} and further in \eqref{eq:proofCSC3}, and
absorbing the terms $\frac{1}{4}\norm{\eta \nabla u}_{L^2(\Omega)}$, 
$\frac{1}{4}\norm{\nabla(\eta \widetilde{K}u)}_{L^2(\R^d)}$ in the left-hand side, finally yields
\begin{align}\label{eq:proofCSC12}
\nonumber \left\|\nabla u \right\|^2_{L^2(B_R\cap \Omega)} 
&+\left\| \nabla \widetilde{V}\varphi\right\|^2_{L^2(B_R)} +
\left\|\nabla \widetilde{K}u \right\|^2_{L^2(B_R\setminus\Gamma)}  \\
&\qquad\lesssim  \frac{h}{\delta R} \left( 
\left\| \nabla u \right\|^2_{L^2(B_{(1+\delta)R}\cap\Omega)}
+\left\| \nabla \widetilde{K}u \right\|^2_{L^2(B_{(1+\delta)R}\setminus\Gamma)}
+ \left\| \nabla \widetilde{V}\varphi \right\|^2_{L^2(B_{(1+\delta)R})} \right) \nonumber
\\& \qquad\quad+ \frac{1}{(\delta R)^2} \left(\left\|  u \right\|^2_ {L^2(B_{(1+\delta)R}\cap\Omega)}
\left\| \widetilde{V}\varphi  \right\|^2_ {L^2(B_{(1+\delta)R})}
+\left\| \nabla \widetilde{K}u \right\|^2_{L^2(B_{(1+\delta)R}\setminus\Gamma)} \right).
\end{align}

 \noindent {\bf Step 3:} By reapplying \eqref{eq:proofCSC12} to the gradient terms 
 with $\delta = \frac{ \varepsilon}{2}$ and suitable boxes, we get the desired result exactly
 as in step 3 of the proof of Theorem~\ref{th:CaccioppoliBMcC}.
\end{proof}

%%%%%%%%%%%%%%%%%%%%%%%%%%%%%%%%%%%%%%%%%%%%%%%%%%%%%%%%%%%%%%%%%%%%%%\frac{h^2}{( \delta R) ^2} 
\subsection{The Johnson-N\'ed\'elec coupling}
In this section we prove the Caccioppoli-type inequality from Theorem~\ref{th:CaccioppoliJN}
for the Johnson-N\'ed\'elec coupling.
Most of the appearing terms have already been treated in the previous sections. The main difference
is that the double-layer potential appears naturally due to the boundary coupling terms, but 
the local orthogonality is not suited to provide an approximation for it, since the hypersingular 
operator does not appear in the bilinear form. 
A remedy for this problem is to localize the double-layer potential by splitting it into 
a local near-field and a non-local, but smooth far-field. This techniques follows
\cite{FM18}, where a similar localization using commutators is employed and 
a more detailed description of the method can be found.

% To capture the orthogonal property that is identified in \eqref{68}, we introduce the following space
% \begin{align*}
% \mathcal{H}_{3,h}( D):= & \left\lbrace \left(  u, \widetilde{V} \varphi , \widetilde{K}u\right)  \in H^1 \left( D \setminus \Gamma \right) \times H^1 ( D \setminus \Gamma ) \times H^1 ( D \setminus \Gamma ):\exists ( \widetilde{u}, \widetilde{\varphi})  \in S^{1,1}( \mathcal{\tau}   _h) \times S^{0,0}( \mathcal{K}   _h) \right. \\[3 mm]\nonumber
% &  s.t. \,\, \left. u\right| _{D\cap\Omega}=  \left. \widetilde{u} \right| _{D\cap\Omega},
% \,\left.  \left[ \partial _{\nu} \widetilde{V} \varphi \right] \right| _{D \cap \Gamma } = -\left. \widetilde{\varphi}\right| _{D \cap \Gamma } ,
%  \left.  \left[ \gamma _0 \, \widetilde{K} u \right] \right| _{D \cap \Gamma } = \left. \widetilde{u}\right| _{D \cap \Gamma },\,
%   a_2( u, \varphi, \psi , \zeta)=0  
% \\[3 mm]\nonumber
% & \forall (\psi, \zeta)  \in S^{1,1}\left( \mathcal{\tau}   _h\right) \times S^{0,0}( \mathcal{K}   _h),  \left. \text{supp}( \psi , \zeta )  \subseteq D \right\rbrace.  
% \end{align*}

\begin{proof}[Proof of Theorem~\ref{th:CaccioppoliJN}] 
Once again, we write $(u,\varphi)$ for the Galerkin solution $(u_h,\varphi_h)$.
The assumption on the support of the data implies the local orthogonality
\begin{align}\label{eq:orthojnC}
a_{\rm jn}(u,\varphi;\psi_h,\zeta_h) = 0  
\quad \forall (\psi_h,\zeta_h) \in S^{1,1}(\T_h) \times S^{0,0}(\mathcal{K}_h)
\quad \text{with}\quad \operatorname*{supp}\psi_h, \operatorname*{supp}\zeta_h \subset B_{(1+\varepsilon)R}.
\end{align}
Let $\eta \in C_0^\infty(\R^d)$ be a  cut-off function with 
$\text{supp}\, \eta \subseteq B_{( 1+\delta/2) R}$, 
$\eta \equiv 1$ on $B_{( 1+\delta/4) R}$, $0 \le \eta \le 1$, 
and $ \left\| D^j\eta \right\|_{L^\infty(B_{(1+\delta)R})} \lesssim \frac{1}{(\delta R)^j}$ for $j=1,2$.
Here, $0<\delta \leq \varepsilon$ is given such that $\frac{h}{R}\leq\frac{\delta}{16}$.
We note that the condition $\eta \equiv 1$ on $B_{( 1+\delta/4) R}$ is additionally 
imposed due to following estimate \eqref{eq:estKtildeJN}, as the localization of the double-layer operator 
is additionally needed in comparison with the other couplings.

%where $\operatorname*{supp} f, \operatorname*{supp} \varphi_0, \operatorname*{supp} (1/2-K)u_0 \cap B_{(1+\delta)R} = \emptyset$.\\

\noindent
{\bf Step 1:} We start with a localization of the double-layer potential.
More precisely, with a second cut-off function $\widehat{\eta}$ satisfying 
$\widehat{\eta} \equiv 1$ on $B_R$ and $\operatorname{supp}\, \widehat \eta \subseteq B_{( 1+\delta/4) R}$,
$ \left\| \nabla \widehat \eta \right\|_{L^\infty(B_{(1+\delta)R})} \lesssim \frac{1}{\delta R}$,
we split 
\begin{align*}
\widehat\eta \widetilde{K}u = \widehat\eta\widetilde{K}(\eta u) + 
\widehat \eta\widetilde{K}(1-\eta)u =: v_{\rm near} + v_{\rm far}. 
\end{align*}
At first, we estimate the near-field $v_{\rm near} := \widehat\eta\widetilde{K}(\eta u)$.
The mapping properties of the double-layer potential, \eqref{eq:mapping-Ktilde}, together with the fact that 
$\operatorname{supp} \nabla \widehat \eta \subset B_{(1+\delta/4)R}\setminus B_R$
and the trace inequality provide
\begin{align*}
 \norm{\nabla v_{\rm near}}_{L^2(B_R\setminus\Gamma)} \lesssim \norm{\eta u}_{H^{1/2}(\Gamma)} 
+\frac{1}{\delta R}\|\widetilde{K} (\eta u)\|_{L^2(B_{(1+\delta/4)R}\setminus B_R) }
 \lesssim \norm{\eta u}_{H^1(\Omega)} + \frac{1}{\delta R} \|\widetilde {K} (\eta u)\|_{L^2(B_{(1+\delta/4)R}\setminus B_R)}.
\end{align*}
Since $\widehat \eta (1-\eta) \equiv 0$, the far field $v_{\rm far}$ is smooth.
Integration by parts using $\Delta \widetilde{K}((1-\eta) u) = 0$, as
well as $[\gamma_1 \widetilde{K}u] = 0$ and 
$\widehat \eta (1-\eta)\equiv 0$ (therefore no boundary terms appear), leads to
\begin{align*}
  \norm{\nabla v_{\rm far}}_{L^2(B_R\setminus\Gamma)}^2 &=
  \abs{\skp{\nabla \widetilde{K}((1-\eta) u), \nabla(\widehat\eta^2\widetilde{K}((1-\eta) u))}_{L^2(\R^d\setminus\Gamma)}} 
  +\norm{(\nabla \widehat\eta)\widetilde{K}((1-\eta) u)}_{L^2(\R^d)}^2 \\
  &\lesssim \frac{1}{(\delta R)^2}\norm{\widetilde{K}((1-\eta) u)}_{L^2(B_{(1+\delta/4)R}\backslash B_{R})}^2 \\
  &\lesssim \frac{1}{(\delta R)^2}\norm{\widetilde{K}u}_{L^2(B_{(1+\delta/4)R}\backslash B_{R})}^2+
 \frac{1}{(\delta R)^2}\norm{\widetilde{K}(\eta u)}_{L^2(B_{(1+\delta/4)R}\backslash B_{R})}^2.
\end{align*}
Here, we used that $\operatorname*{supp}(\nabla \widehat \eta)\subset B_{(1+\delta/4)R}\backslash B_{R}$.
For the last term, we apply \cite[Lemma~3.7, (i) and (ii)]{FMP16} to obtain
\begin{align*}
\norm{\widetilde{K}(\eta u)}_{L^2(B_{(1+\delta/4)R}\backslash B_{R})} \lesssim 
\sqrt{\delta R}\norm{(1/2-K)(\eta u)}_{L^2(\Gamma)} + 
\sqrt{\delta R}\sqrt{(1+\delta)R}\norm{\nabla\widetilde{K}(\eta u)}_{L^2(B_{(1+\delta/4)R}\setminus\Gamma)}.
\end{align*}
With the mapping properties of $K$, $\widetilde{K}$ from \eqref{eq:mapping-Ktilde}, \eqref{eq:mapping-K} and the multiplicative 
trace inequality this implies
\begin{align*}
\frac{1}{\delta R}\norm{\widetilde{K}(\eta u)}_{L^2(B_{(1+\delta/4)R}\backslash B_{R})} &\lesssim 
\frac{1}{\sqrt{\delta R}}\norm{\eta u}_{L^2(\Gamma)} + 
\sqrt{1+1/\delta}\norm{\eta u}_{H^1(\Omega)} \\
&\lesssim \frac{1}{\sqrt{\delta R}}\norm{\eta u}_{L^2(\Omega)} + 
\frac{1}{\sqrt{\delta R}}\norm{\eta u}^{1/2}_{L^2(\Omega)}\norm{\nabla(\eta u)}^{1/2}_{L^2(\Omega)}+
\sqrt{1+1/\delta}\norm{\eta u}_{H^1(\Omega)}\\
&\lesssim \frac{1}{\delta R}\norm{\eta u}_{L^2(\Omega)} + \norm{\nabla(\eta u)}_{L^2(\Omega)}+
\sqrt{1+1/\delta}\norm{\eta u}_{H^1(\Omega)}.
\end{align*}
Putting the estimates for the near-field and the far-field together, we obtain
\begin{align}\label{eq:estKtildeJN}
 \norm{\nabla \widetilde{K}u}_{L^2(B_R\setminus\Gamma)} &\leq 
  \norm{\nabla v_{\rm near}}_{L^2(B_R\setminus\Gamma)} +   \norm{\nabla v_{\rm far}}_{L^2(B_R\setminus\Gamma)}\nonumber \\
  &\lesssim \sqrt{1+1/\delta}\norm{\eta u}_{H^1(\Omega)} + 
  \frac{1}{\delta R}\norm{u}_{L^2(B_{(1+\delta/4)R}\cap\Omega)} + 
  \frac{1}{\delta R}\norm{\widetilde{K}u}_{L^2(B_{(1+\delta/4)R}\setminus\Gamma)}.
\end{align}

\noindent
{\bf Step 2:} We provide a local ellipticity estimate, i.e., we prove
\begin{align*}
\left\|\nabla (\eta u) \right\|^2_{L^2(\Omega)} 
&+\left\|\nabla (\eta \widetilde{V} \varphi) \right\|^2_{L^2(\R^d) }
+\left\|\nabla  \widetilde{K} u \right\|^2_{L^2(B_R\setminus\Gamma)} 
\lesssim a_{\rm jn}(u, \varphi; \eta ^2 u , \eta ^2 \varphi)+ {\text{terms  in weaker norms}}.
\end{align*}
(See (\ref{eq:ellipticity-jn}) for the precise form).
We start with \eqref{eq:estKtildeJN} to obtain
\begin{align}\label{eq:proofJNC1}
\nonumber \left\|\nabla(\eta u ) \right\|^2_{L^2(\Omega) }  
+\left\|\nabla ( \eta \widetilde{V} \varphi ) \right\|^2_{L^2(\R^d) } 
+\left\|\nabla \widetilde{K} u  \right\|_ {L^2( B_R \setminus\Gamma ) }^2 
&\lesssim 
(1+1/\delta)\left(\left\|\nabla(\eta u ) \right\|^2_{L^2(\Omega)} 
+\left\|\nabla ( \eta \widetilde{V} \varphi ) \right\|^2_{L^2(\R^d) }  \right)\\ 
& \quad+
  \frac{(1+1/\delta)}{(\delta R)^2}\norm{u}_{L^2(B_{(1+\delta)R}\cap\Omega)}^2 + 
  \frac{1}{(\delta R)^2}\norm{\widetilde{K}u}_{L^2(B_{(1+\delta)R}\setminus\Gamma)}^2.
\end{align} 
The last two terms are already in weaker norms, and for the first two terms, we apply \eqref{eq:proofBMcC3}.
Since we assumed $C_{\rm ell}>1/4$ for unique solvability, we choose a $\rho >0$ 
such that $1/4 < \rho/2 < C_{\rm ell}$ and set 
$C_{\rho}:=\min\{1-\frac{1}{2\rho},C_{\rm ell}-\frac{\rho}{2}\} > 0$. Then \eqref{eq:proofBMcC3} implies
\begin{align}\label{eq:proofJNC1b}
\nonumber C_{\rho}\left\|\nabla(\eta u ) \right\|^2_{L^2(\Omega) }  
+C_{\rho}\left\|\nabla ( \eta \widetilde{V} \varphi ) \right\|^2_{L^2(\R^d) } 
&\leq
C_{\rm ell}\left\|\nabla(\eta u ) \right\|^2_{L^2(\Omega) }  
+\left\|\nabla ( \eta \widetilde{V} \varphi) \right\|^2_{L^2(\R^d) }  \\
&\quad- \skp{\nabla ( \eta \widetilde{V} \varphi),\nabla(\eta u )}_{L^2( \Omega) } -
\skp{\nabla  \widetilde{V} \varphi,\nabla (\eta ^2\widetilde{K} u)}_{L^2(\mathbb{R}^d\setminus\Gamma)} \nonumber \\
&\quad+\skp{\nabla  \widetilde{V} \varphi,\nabla (\eta ^2\widetilde{K} u)}_{L^2(\mathbb{R}^d\setminus\Gamma)}.
\end{align} 
The first three terms can be expanded as in Theorem~\ref{th:CaccioppoliBMcC}, where \eqref{eq:proofBMcC4} leads to
\begin{align}\label{eq:proofJNC9}
 \skp{\nabla ( \eta \widetilde{V} \varphi),\nabla(\eta u )}_{L^2( \Omega) }
& =\skp{\nabla  \widetilde{V} \varphi,\nabla (\eta ^2 u)}_{L^2(\Omega)}
+ \text{l.o.t.},
\end{align}
where the omitted terms (cf.~(\ref{eq:proofBMcC4})) 
\begin{equation*}
\text{l.o.t.} = 
\langle (\nabla \eta) \widetilde{V} \varphi, \nabla (\eta u)\rangle_{L^2(\Omega)} - \langle \nabla \widetilde{V} \varphi,\eta (\nabla \eta) u\rangle_{L^2(\Omega)}
\end{equation*}
can be estimated in weaker norms (i.e., $\|\widetilde{V}\varphi\|_{L^2(B_{(1+\delta/2)R}\cap\Omega)}$, $\|u\|_{L^2(B_{(1+\delta/2)R}\cap\Omega)}$)
or lead to terms that are absorbed in the left-hand side
as in the proof of Theorem~\ref{th:CaccioppoliBMcC} (see \eqref{eq:proofBMcCYoung1}, \eqref{eq:proofBMcCYoung2}).
Equations \eqref{eq:proofCSC3a} and \eqref{eq:proofBMcC5} give
\begin{align}\label{eq:proofJNC10}
\skp{\nabla  \widetilde{V} \varphi,\nabla (\eta ^2 u)}_{L^2(\Omega)}
+\skp{\nabla  \widetilde{V} \varphi,\nabla (\eta ^2\widetilde{K} u)}_{L^2(\mathbb{R}^d\setminus\Gamma)} 
=\skp{(1/2+K)u, \eta ^2 \varphi} _{L^2( \Gamma) }.
\end{align}
Therefore, we only have to estimate the last term in \eqref{eq:proofJNC1b}.
We write in the same way as in \eqref{eq:proofJNC9}
\begin{align*}
\skp{\nabla  \widetilde{V} \varphi,\nabla (\eta ^2\widetilde{K}u)}_{L^2(\mathbb{R}^d\setminus\Gamma)} = 
\skp{\nabla (\eta^2\widetilde{V} \varphi),\nabla \widetilde{K}u}_{L^2(\mathbb{R}^d\setminus\Gamma)}+ \text{l.o.t.},
\end{align*}
where, again, the omitted terms 
\begin{equation*}
\text{l.o.t.} =
2 \langle (\nabla (\eta \widetilde{V} \varphi), (\nabla \eta) \widetilde{K} u\rangle_{L^2(\R^d\setminus\Gamma)} 
- 2 \langle (\nabla \eta) \widetilde{V} \varphi, \nabla (\eta \widetilde{K} u)\rangle_{L^2(\R^d\setminus\Gamma)} 
\end{equation*}
can be estimated in weaker norms
(i.e., by $\|\widetilde{K} u\|_{L^2(B_{(1+\delta/2)R}\setminus\Gamma)}$  and $\|\widetilde{V} \varphi\|_{L^2(B_{(1+\delta/2)R}}$) 
or absorbed in the left-hand side.
Now, integration by parts on 
$\R^d\backslash\overline{\Omega}$ and $\Omega$
together with $\Delta \widetilde{K}u = 0$ and $[\gamma_1 \widetilde{K}u] = 0 = [\eta^2\widetilde{V} \varphi]$
implies
\begin{align*}
\skp{\nabla (\eta^2\widetilde{V} \varphi),\nabla \widetilde{K}u}_{L^2(\mathbb{R}^d \setminus\Gamma)} = 
\skp{\eta^2\widetilde{V} \varphi,\Delta \widetilde{K}u}_{L^2(\mathbb{R}^d\backslash \Gamma)}=0.
\end{align*}
Putting everything together into \eqref{eq:proofJNC1b} and in turn into \eqref{eq:proofJNC1}, 
we obtain
\begin{align}
\nonumber \left\|\nabla(\eta u ) \right\|^2_{L^2(\Omega) }  
&+\left\|\nabla ( \eta \widetilde{V} \varphi ) \right\|^2_{L^2(\R^d) } 
 +\left\|\nabla (\eta \widetilde{K}  u )  \right\|_ {L^2( \R^d \setminus \Gamma) }^2 
\\ &\qquad \lesssim (1+1/\delta)\, a_{\rm jn}(u, \varphi; \eta ^2 u , \eta ^2 \varphi) 
+\frac{(1+1/\delta)}{(\delta R)^2} \left\|  \widetilde{K}u\right\| ^2_{L ^ 2( B_{ ( 1+\delta )R } \setminus\Gamma) } \nonumber
\\
\label{eq:ellipticity-jn}
& \qquad \quad+\frac{(1+1/\delta)}{(\delta R)^2}\left\| u \right\| ^2 _{L^2( B_{ ( 1+\delta )R }\cap\Omega)} 
+ \frac{(1+1/\delta)}{(\delta R)^2} \norm{\widetilde{V} \varphi}_{L^2(B_{(1+\delta)R})}^2.
\end{align}
\noindent
{\bf Step 3:} We apply the local orthogonality of $(u,\varphi)$ 
to piecewise polynomials and use approximation properties.

Let $I_h^{\Omega}: C(\overline{\Omega})  \rightarrow S^{1,1}(\T_h)$ 
be the nodal interpolation operator and $I_h^{\Gamma}$ the $L^2(\Gamma)$-orthogonal projection mapping onto
$S^{0,0}(\mathcal{K}_h)$. Then, the orthogonality \eqref{eq:orthojnC} leads to
\begin{align}\label{eq:proofJNC11}
\nonumber 
a_{\rm jn}(u, \varphi; \eta ^2 u , \eta ^2 \varphi)
&=a_{\rm jn}(u, \varphi; \eta^2 u-I_h^\Omega (\eta^2 u), \eta^2 \varphi-I_h^\Gamma (\eta^2 \varphi)) \\
& \nonumber=
\skp{\nabla u, \nabla (\eta ^2 u-I_h ^ \Omega (\eta ^2 u )  )} _{L^2(\Omega)} 
+ \skp{V\varphi , \eta ^2 \varphi-I_h ^ \Gamma (\eta ^2 \varphi )} _{L ^ 2(\Gamma) }  \\
& 
\quad -\skp{\varphi,\eta^2  u - I_h ^ \Omega (\eta ^2 u )} _{L ^ 2(\Gamma) }
+\skp{(1/2 -K)u , \eta^2 \varphi - I_h ^ \Gamma(\eta ^2 \varphi) } _{L ^ 2(\Gamma) } \nonumber \\
&=: T_1 + T_2 +T_3 +T_4.
\end{align}
The terms $T_1$, $T_2$ have already been estimated in the proof of Theorem~\ref{th:CaccioppoliBMcC}, inequalities
\eqref{eq:proofBMcC11}, \eqref{eq:proofBMcC9}, and $T_4$ was treated in \eqref{eq:proofCSC11} in the proof of 
Theorem~\ref{th:CaccioppoliSymm}.

It remains to estimate $T_3$.
%With the jump condition $[ \gamma_1 \widetilde{V}\varphi]= - \varphi$ and 
With $\text{supp}\left( \eta^2 u - I_h ^ \Omega( \eta ^2 u)  \right) \subset B_{(1+\delta/2)R}  $
due to $16h\leq \delta R$, 
we get 
\begin{align*}
\abs{T_3}&=\abs{\skp{ \varphi,\eta^2 u - I_h ^ \Omega(\eta ^2 u)}_{L^2(\Gamma)}} \leq 
%=\abs{\skp{ 
%[\gamma_1 \widetilde{V}\varphi] ,\eta^2 u  - I_h^\Omega(\eta^2 u)}_{L^2(\Gamma)}} 
%\\&\leq
 \left\|\varphi\right\|_{L^2(B_{(1+\delta/2)R}\cap \Gamma)}
 \left\| \eta^2 u - I_h ^\Omega(\eta ^2 u)\right\|_{L^2(\Gamma)}. 
\end{align*}
Lemma~\ref{lem:investV} provides 
\begin{align*}
\left\|\varphi  \right\|_{{L^2}( B_{(1+\delta/2)R}) }&
\lesssim h ^{-1/2} \left\| 	 \nabla \widetilde{V}\varphi \right\|_{L^2( B_{ ( 1+\delta )R })}.
\end{align*}
Therefore, with the super-approximation properties (\ref{eq:super-approximation-Omega}) of $I_h^{\Omega}$, we obtain
\begin{align}\label{eq:proofJNC17}
\nonumber \left|\skp{ \varphi,I_h ^\Omega(\eta^2 u)-\eta ^2 u}_{L^2(\Gamma)} \right|
&\lesssim h ^{-1/2} \left\| \nabla \widetilde{V}\varphi \right\|_{L^2( B_{ ( 1+\delta )R }) }
\left(\frac{h^{3/2}}{ \delta R}  \left\|    \nabla  u \right\|  _ {L^2( B_{(1+ \delta) R}\cap\Omega) }
+ \frac{h ^{3/2}}{( \delta R) ^2}    \left\|     u \right\|  _ {L^2( B_{(1+ \delta)R}\cap\Omega) } \right)
\\
&
 \lesssim    \frac{h}{ \delta R}\left(  \left\| \nabla   \widetilde{V}\varphi  \right\|^2_{L^2 (B_{(1+\delta)R})}     
 +  \left\|  \nabla   u  \right\|^2_{L^2 ( B_{(1+ \delta)R }\cap\Omega  ) } 
 \right)+
 \frac{1}{( \delta R) ^2}  \left\|  u \right\|^2_ {L^2( B_{(1+ \delta)R } )}. 
\end{align}
Putting the estimates of $T_1$, $T_2$, $T_3$, $T_4$ together and using $\delta \lesssim 1$ leads to
\begin{align}\label{eq:proofJNC18}
& \left\| \nabla u \right\|_{L^2(B_R\cap\Omega)}^2 +
\nonumber \left\|\nabla \widetilde{V}  \varphi  \right\|_{L^2(B_R)}^2  
+\left\| \nabla \widetilde{K}  u  \right\|_ {L^2(B_R \setminus \Gamma )}^2  \\
& \qquad 
 \lesssim  
\frac{h}{ \delta^2 R}\left(  \left\| \nabla u  \right\|^2_{L^2 ( B_{(1+ \delta)R } \cap\Omega ) } 
+\left\| 	\nabla \widetilde{V} \varphi  \right\|^2_{L^2 ( B_{(1+ \delta)R }  ) }\right.
 \nonumber \left.\quad+\left\| 	\nabla \widetilde{K} u  \right\|^2_{L^2 ( B_{(1+ \delta)R } \setminus\Gamma ) }\right)  
\\& \qquad \qquad 
+
\frac{1}{\delta^3 R^2} \left(  \left\| 	 u  \right\|^2_{L^2 ( B_{(1+ \delta)R } \cap\Omega ) } \right.
 \left.\quad+\left\| 	\widetilde{V} \varphi  \right\|^2_{L^2 ( B_{(1+ \delta)R }  ) }
+\left\|  \widetilde{K} u  \right\|^2_{L^2 ( B_{(1+ \delta)R } \setminus \Gamma ) }\right).
\end{align}
\textbf{Step 4.} 
Reapplying \eqref{eq:proofJNC18} to the gradient terms 
 with $\delta = \frac{ \varepsilon}{2}$ and suitable boxes, we get the desired result exactly
 as in step 3 of the proof of Theorem~\ref{th:CaccioppoliBMcC}.
\end{proof}

%---------------------------------------
\section{$\mathcal{H}$-matrix approximation to inverse matrices}\label{sec:proofHmatrix}
In this section, we prove the existence of exponentially convergent $\H$-matrix approximants to the inverses of 
the stiffness matrices of the FEM-BEM couplings, as stated in Theorem~\ref{th:H-Matrix approximation of inverses}.

Analyzing the procedure in \cite{FMP15, FMP16,AFM20} shows structural similarities in the derivation of ${\mathcal H}$-matrix 
approximations based on low-dimensional spaces of functions: A single-step approximation is obtained by using a 
Scott-Zhang operator on a coarse grid. Iterating this argument is made possible by a Caccioppoli-inequality, resulting in 
a multi-step approximation. The key ingredients of the argument are collected in properties (A1)--(A3) below. 
We mainly follow \cite{AFM20}.
\subsection{Abstract setting - from matrices to functions}
\label{sec:abstractMatrixtoFunction}
We start by reformulating the matrix approximation problem as a question 
of approximating certain functions  from low dimensional spaces.

Let $\mathbf{X}$ be a Hilbert space of functions. We consider 
variational problems of the form: find $\mathbf{u} \in \mathbf{X}$ such that 
\begin{align*}
a(\mathbf{u},\boldsymbol{\psi}) = \skp{\boldsymbol{f},\boldsymbol{\psi}} \qquad \forall \boldsymbol{\psi} \in \boldsymbol{X}
\end{align*}
for given $a(\cdot,\cdot): \mathbf{X}\times \mathbf{X} \rightarrow \R$, $\boldsymbol{f} \in \mathbf{X}'$. 
Here, the bold symbols may denote vectors, e.g., $\mathbf{u}= (u,\varphi)$  in \eqref{eq:abstractFEMBEM} 
for $\mathbf{X}= H^1(\Omega)\times H^{-1/2}(\Gamma)$,
and $\skp{\cdot,\cdot}$ denotes the appropriate duality bracket. 

For fixed $k$, $\ell \in \N$ (given by the formulation of the problem), we define $\mathbf{L}^2 := L^2(\Omega)^k\times L^2(\Gamma)^\ell$.

\begin{definition}
Let $\mathbf{X}_N \subset \mathbf{X}$ be a finite dimensional subspace of dimension $N$ that is 
also a subspace  $\mathbf{X}_N \subset \mathbf{L}^2$.
Then the linear mapping $\mathcal{S}_N : \mathbf{X}' \rightarrow \mathbf{X}_N$ 
is called the discrete solution operator
if for every $\boldsymbol{f} \in \mathbf{X}'$, there exists a unique function  
$\mathcal{S}_N \boldsymbol{f} \in \mathbf{X}_N$ satisfying
\begin{align}\label{eq:abstractDiscrete}
a(\mathcal{S}_N \boldsymbol{f},\boldsymbol{\psi}) = \skp{\boldsymbol{f},\boldsymbol{\psi}} \qquad \forall \boldsymbol{\psi} \in \boldsymbol{X}_N.
\end{align}
\end{definition}
Let $\{\boldsymbol{\phi}_1,\dots,\boldsymbol{\phi}_N\} \subseteq \mathbf{X}_N$ be a basis of $\mathbf{X}_N$.
%
%By $\mathcal{S}_N$, we denote the solution operator for the discrete variational formulation, i.e, for $\mathbf{f} \in \mathbf{V}'$ the function $\mathcal{S}_N\mathbf{f}$ is the unique 
%solution of 
%\begin{align}\label{eq:abstractDiscrete}
%a(\mathcal{S}_N\mathbf{f},\boldsymbol{\psi}) = \skp{\boldsymbol{f},\boldsymbol{\psi}} \qquad \forall \boldsymbol{\psi} \in \boldsymbol{V}_N.
%\end{align}
We denote the Galerkin matrix  $\mathbf{A} \in \R^{N \times N}$  by
\begin{align}
 \mathbf{A} =\left(  a(\boldsymbol{\phi}_j,\boldsymbol{\phi}_i)\right)  _ {i,j=1} ^N.
\end{align}
The translation of the problem of approximating matrix blocks of $\mathbf{A}^{-1}$ to the problem of approximating certain functions
from low dimensional spaces essentially depends on the following crucial property (A1), the existence of 
a local dual basis.
\begin{itemize}
\item[(A1)] There exist dual functions $\{\boldsymbol{\lambda}_1,\dots,\boldsymbol{\lambda}_N\} \subset \mathbf{L}^2$ satisfying
\begin{align*}
\skp{\boldsymbol{\phi}_i,\boldsymbol{\lambda}_j} = \delta_{ij}, \quad \text{and} \quad 
\Big\|\sum_{j=1}^N \mathbf{x}_j \boldsymbol{\lambda}_j\Big\|_{\mathbf{L}^2}\leq C_{\rm db}(N) \norm{\mathbf{x}}_2 
\end{align*}
for all $i,j \in \{1,\dots,N\}$ and $\mathbf{x} \in \R^N$.
Moreover, we require the $\boldsymbol{\lambda}_i$ to have local support, in the sense that 
$\#\{j \;:\; \operatorname*{supp}(\boldsymbol{\lambda}_i)\cap \operatorname*{supp}(\boldsymbol{\lambda}_j)\neq \emptyset\} \lesssim 1$ 
for all $i \in \{1,\dots,N\}$.
\end{itemize}
We denote the coordinate mappings corresponding to the basis and the dual basis by
\begin{equation*}
\Phi : \left\{
\begin{array}{ccc}
\R^N & \longrightarrow & \mathbf{X}_N \\
\mathbf{x} & \longmapsto & \sum_{j=1}^N \mathbf{x}_j \boldsymbol{\phi}_j
\end{array}
\right.,\quad \quad \quad
\Lambda : \left\{
\begin{array}{ccc}
\R^N & \longrightarrow & \mathbf{L}^2 \\
\mathbf{x} & \longmapsto & \sum_{j=1}^N \mathbf{x}_j \boldsymbol{\lambda}_j
\end{array}
\right..
\end{equation*}
The  Hilbert space transpose of  $\Lambda$ is denoted by  $\Lambda ^T$. Moreover, for $\tau \subset \{1,\dots,N\}$, 
we define the sets $D_j(\tau) := \cup _{i \in \tau } \operatorname*{supp} \boldsymbol{\lambda}_{i,j}$, 
where $\boldsymbol{\lambda}_{i,j}$ is the $j$-th component of $\boldsymbol{\lambda}_{i}$,
and write $\mathbf{L}^2(\tau):= \prod_{j=1}^{k+\ell}L^2(D_j(\tau))$.

In the following lemma,  we derive a representation formula for $\mathbf{A} ^ {-1}$ based 
on three linear operators   $ \Lambda ^T,\, \mathcal{S}_N$ and $\Lambda$. 

\begin{lemma}\label{representation-inverse}(\cite[Lem.~3.10]{AFM20}, \cite[Lem.~3.11]{AFM20})
The restriction of $\Lambda ^T$ to $\mathbf{X}_ N$ is the inverse mapping  $\Phi^{-1}$. More precisely, for all $\mathbf{x}, \mathbf{y} \in \mathbb{R}^N$ and $\mathbf{v} \in \mathbf{X}_N$, we have
\begin{align*}
\skp{\Lambda \mathbf{x},\Phi \mathbf{y}}=\skp{\mathbf{x},\mathbf{y}}_2, \qquad \Lambda ^T \Phi \mathbf{x}=\mathbf{x}, \qquad \Phi \Lambda ^T  \mathbf{v}=\mathbf{v}.
\end{align*}
The mappings $\Lambda$ and $\Lambda^T$ preserve locality, i.e, for $\tau \subset \{1,\dots,N\}$ and $\mathbf{x} \in \R^N$ with $\{i:\mathbf{x}_i \neq 0\} \subset \tau$, we have
$\operatorname*{supp}(\Lambda \mathbf{x}) \subset \prod_j D_j(\tau)$.
For $\mathbf{v} \in \mathbf{L}^2$, we have  
\begin{align*}
\norm{\Lambda^T \mathbf{v}}_{\ell^2(\tau)} \leq \norm{\Lambda}\norm{\mathbf{v}}_{\mathbf{L}^2(\tau)}.
\end{align*}
Moreover, there holds the representation formula
\begin{align*}
\mathbf{A} ^ {-1} \textbf{x}= \Lambda ^T \mathcal{S}_N \Lambda  \textbf{x} \qquad \forall \textbf{x} \in \mathbb{R}^N.
\end{align*}
\end{lemma}
\begin{proof}
For sake of completeness, we provide the derivation of the representation formula from \cite[Lem.~3.11]{AFM20}.
Using that $\Lambda^T = \Phi^{-1}|_{\mathbf{X}_N}$ and the definition of the discrete solution operator, we compute
\begin{align*}
\skp{\mathbf{A}\Lambda^T\mathcal{S}_N \Lambda \mathbf{x},\mathbf{y}}_2 = 
a(\Phi \Lambda^T\mathcal{S}_N \Lambda\mathbf{x},\Phi\mathbf{y}) = 
a(\mathcal{S}_N \Lambda\mathbf{x},\Phi\mathbf{y}) = \skp{\Lambda \mathbf{x},\Phi\mathbf{y}} = 
\skp{\mathbf{x},\mathbf{y}}_2
\end{align*}
for arbitrary $\mathbf{y} \in \R^N$.
\end{proof}

%
%
%We refer to \cite[Lem.~3.10]{AFM20} that the restriction of the Hilbert space transpose $\Lambda^T$ of $\Lambda$  to $\mathbf{V}_N$ is the inverse $\Phi^{-1}$ of the coordinate mapping $\Phi$.
This lemma  is the crucial step  in the  proof of the following lemma.

\begin{lemma}\label{th:H-Matrix approximation of inverses-abstract} Let 
$\mathbf{A}$ be the Galerkin matrix, $\Lambda$ be the coordinate mapping for the dual basis, and 
$\mathcal{S}_N$ be the discrete solution operator. 
Let $\tau \times \sigma \subset \{1,\dots,N\}\times \{1, \dots, N\}$ be an admissible block and $\mathbf{W}_r \subseteq \mathbf{L}^2$ be a finite dimensional space. Then, there exist 
matrices $\mathbf{X}_{\tau\sigma} \in \R^{\abs{\tau} \times r},\mathbf{Y}_{\tau\sigma} \in \R^{\abs{\sigma} \times r}$ of rank $r \leq \dim \mathbf{W}_r$ satisfying
\begin{align*}
\norm{\mathbf{A}^{-1}|_{\tau \times \sigma} -\mathbf{X}_{\tau\sigma}\mathbf{Y}_{\tau\sigma}^T}_2 \leq \norm{\Lambda}^2 
\sup_{\substack{\boldsymbol{f}  \in \mathbf{L}^2:\\ \operatorname*{supp}(\boldsymbol{f}) \subset 
\prod_j D_j(\sigma)}}
 \frac{\inf_{\mathbf{w} \in \mathbf{W}_r}\norm{\mathcal{S}_N \boldsymbol{f}  -\mathbf{w}}_{\mathbf{L}^2(\tau)}}{\norm{\boldsymbol{f}}_{\mathbf{L}^2}}.
\end{align*}
\end{lemma}

\begin{proof}
We use the representation formula  from Lemma \ref{representation-inverse} to prove the asserted estimate. 
With the given space $\mathbf{W}_r$, we define $\mathbf{X}_{\tau\sigma} \in \R^{\abs{\tau}\times r}$ columnwise as vectors from an orthonormal basis 
of the space $\widehat{\mathbf{W}}:=(\Lambda^T \mathbf{W}_r)|_\tau$. 
Then, the product $\mathbf{X}_{\tau\sigma} \mathbf{X}_{\tau\sigma}^T$ is the orthogonal projection onto $\widehat{\mathbf{W}}$.
Defining $\mathbf{Y}_{\tau\sigma} :=  (\mathbf{A}^{-1}|_{\tau \times \sigma})^T \mathbf{X}_{\tau\sigma}$, we can compute for all 
$\mathbf{x} \in \R^N$ with 
$\{i\;:\;\mathbf{x}_i \neq 0\} \subset \sigma$ that
\begin{align*}
 \norm{(\mathbf{A}^{-1}|_{\tau \times \sigma} -\mathbf{X}_{\tau\sigma}\mathbf{Y}_{\tau\sigma}^T)\mathbf{x}|_\sigma}_{\ell^2(\tau)} &= 
  \norm{(\mathbf{I}-\mathbf{X}_{\tau\sigma}\mathbf{X}_{\tau\sigma}^T)(\mathbf{A}^{-1}\mathbf{x})|_\sigma}_{\ell^2(\tau)} = 
  \inf_{\widehat{\mathbf{w}}\in \widehat{\mathbf{W}}}  \norm{(\mathbf{A}^{-1}\mathbf{x})|_\sigma-\widehat{\mathbf{w}}}_{\ell^2(\tau)} \\
  &\stackrel{Lem.~\ref{representation-inverse}}{=}  \inf_{\mathbf{w}\in \mathbf{W}_r}  \norm{\Lambda^T(\mathcal{S}_N\Lambda\mathbf{x}-\mathbf{w})}_{\ell^2(\tau)} \leq 
  \norm{\Lambda} \inf_{\mathbf{w}\in \mathbf{W}_r}  \norm{\mathcal{S}_N\Lambda\mathbf{x}-\mathbf{w}}_{\mathbf{L}^2(\tau)}.
\end{align*}
Dividing both sides by $\norm{\mathbf{x}}_2$,
substituting $\boldsymbol{f}  :=\Lambda \mathbf{x}$ and using that the mapping $\Lambda$ preserves supports, we get the 
desired result.
\end{proof}

Finally, the question of approximating the whole matrix $\mathbf{A}^{-1}$ can be reduced to the question of 
blockwise approximation. For arbitrary matrices 
$\mathbf{M} \in \R^{N\times N}$, and an arbitrary block partition $P$ of $\{1,\dots,N\} \times \{1,\dots,N\}$ 
this follows from
\begin{align*}
\norm{\mathbf{M}}_2 \leq N^2\max\{\norm{\mathbf{M}|_{\tau\times \sigma}}_2 : (\tau,\sigma) \in P\}.
\end{align*}
If the block partition $P$ is based on a cluster tree $\mathbb{T}_{\mathcal{I}}$, the more refined estimate
\begin{align}\label{eq:refinedblockestimate}
\norm{\mathbf{M}}_2 \leq 
C_{\rm sp} \operatorname*{depth}(\mathbb{T}_{\mathcal{I}})\max\{\norm{\mathbf{M}|_{\tau\times \sigma}}_2 : (\tau,\sigma) \in P\}
\end{align}
holds, see 
{\cite{GrasedyckDissertation}, \cite[Lemma 6.5.8]{HackbuschBuch}, \cite{BoermBuch}}. \medskip

In Section~\ref{sec:proofThmHmatrix}, we give explicit definitions of the dual basis for the FEM-BEM coupling model problem.

\subsection{Abstract setting - low dimensional approximation}

We present a general framework that only uses a Caccioppoli type estimate for the construction of 
exponentially convergent low dimensional approximations. 

Let $M \in \N$ be fixed. For $R >0$ let 
$\CB_R:=\lbrace B_i \rbrace _{i=1} ^M$ be a collection of boxes, i.e., $B_i \in \{B_R\cap\Omega,B_R,B_R\backslash\Gamma\}$ for all $i=1,\dots,M$, 
where $B_R$ denotes a box of side length $R$. The choice, which of the three sets is taken for each index $i$, is determined by the application and fixed.

We write $\CB \subset \CB' := \lbrace B_i' \rbrace _{i=1} ^M$ meaning that $B_i \subset B_i'$ for all $i=1,\dots,M$. 
For a parameter $\delta > 0$, we call $ \CB_R^\delta := \{B_i^\delta\}_{i=1}^M$ a collection of $\delta$-enlarged boxes of $\CB_R$, if it satisfies
$$B_i^\delta \in \{B_{R+2\delta}\cap\Omega,B_{R+2\delta},B_{R+2\delta}\backslash\Gamma\} \;\; \forall i=1,\dots,M, \qquad \text{ and } \quad
 \CB^\delta_R  \supset \CB_R,$$
where $B_R$ and $B_{R+2\delta}$ are concentric boxes. Defining  $\operatorname*{diam}(\CB_R) := \max\{\operatorname{diam}(B_i), i=1,\dots,M\},$ we get
\begin{align}\label{eq:diameters}
\diam(\CB_R^\delta)\leq \diam(\CB_R) + 2\sqrt{d}\delta.
\end{align}
In order to simplify notation, we drop the subscript $R$ and write $\CB := \CB_R$ in the following abstract setting.

We use the notation $\mathbf{H}^1(\CB)$ to abbreviate the product space $\mathbf{H}^1(\CB) = \prod_{i=1}^M H^1(B_i)$,
and write $\norm{\mathbf{v}}_{\mathbf{H}^1(\CB)}^2:= \sum_{i=1}^M \norm{\mathbf{v}_i}_{H^1(B_i)}^2$ for the product norm.

\begin{remark}
For the application of the present paper, we chose boxes (or suitable subsets of those) for the sets $B_i$. We also mention that different constructions can be employed as
demonstrated in \cite{AFM20}, where a construction for non-uniform grids 
is presented and where the metric is not the Euclidean one but one that is 
based on the underlying finite element mesh.
\eremk
\end{remark}

In the following, we fix some assumptions on the collections $\CB$ 
of interest and the norm $\triplenorm{\cdot}_{\CB}$ on  
$\mathcal{B}$ we derive our approximation result in. In essence,
we want a norm weaker than than the classical $H^1$-norm that has the correct scaling (e.g., an $L^2$-type norm).

\begin{itemize}
\item[(A2)] Assumptions on the approximation norm $\triplenorm{\cdot}_{\CB}$: 
 For each $\CB$, the Hilbertian norm 
$\triplenorm{\cdot}_{\CB}$ is a norm on $\mathbf{H}^1(\CB)$ and such that 
%We assume that the norm $\triplenorm{\cdot}_{\CB}$ is weaker than the usual $H^1$-norm. 
%i.e.,
%\begin{align*}
%\triplenorm{\Bv}_{\CB} \leq C \norm{\Bv}_{H^1(\CB)}.
%\end{align*}
for any $\delta >0$ and enlarged boxes $\CB^\delta$ and $H>0$ 
there is a discrete space $\mathbf{V}_{H,\CB^\delta} \subset \mathbf{H}^1(\CB^\delta)$ of dimension 
$\dim\mathbf{V}_{H,\CB^\delta} = C (\diam(\CB^\delta)/H)^{Md}$ and a
linear operator
%we require the existence of a linear operator 
$Q_H : \mathbf{H}^1(\CB^\delta) \rightarrow \mathbf{V}_{H,\CB^\delta}$ such that
\begin{align*}
\triplenorm{\Bv-Q_H \Bv}_{\CB} \leq  C_{\rm Qap}H (\norm{\nabla\Bv}_{L^2(\CB^\delta)}+\delta^{-1}\triplenorm{\Bv}_{\CB^\delta})
\end{align*}
with a constant $C_{\rm Qap}>0$ that does not depend on $\CB,\CB^\delta,\delta,$ and $N$.
\end{itemize}

Finally, we require a Caccioppoli type estimate with respect to the norm 
from (A2).  
\begin{itemize}\label{Caccioppoli-abstract}
 \item[(A3)]  Caccioppoli type estimate: For each $\CB$,  
$\delta>0$ and collection $\CB^\delta$  of $\delta$-enlarged boxes
with $\delta \geq C_{\rm Set}(N)$ with a fixed constant $C_{\rm Set}(N)>0$ 
that may depend on $N$, there is a subspace 
$\H_h(\CB^\delta)\subset \mathbf{H}^1(\CB^\delta)$ such that for all 
 $\mathbf{v} \in \H_h(\CB^\delta)$ the inequality
\begin{align}\label{Caccioppoli inequality}
\left\| \nabla{\Bv }\right\|_ {L^2(\CB) } \le 
C_{\rm Cac}\frac{\diam(\CB)^{\alpha-1}}{\delta^{\alpha}} \triplenorm{\Bv}_{\CB^\delta}
\end{align}
holds. 
Here, the constants $C_{\rm Cac}>0$ and $\alpha \ge 1$ do not depend on $\CB,\CB^\delta$, $\delta$, and $N$.

We additionally assume the spaces $\H_h(\CB^\delta)$ to be finite dimensional and nested, i.e., 
$\H_h(\CB') \subset \H_h(\CB)$ for $\CB \subset \CB'$.
\end{itemize}

By $\Pi_{h,\mathcal{B}}$, we denote the orthogonal projection $\Pi_{h,\mathcal{B}}: \mathbf{H}^1(\CB) \rightarrow \H_h(\CB) $ onto that space
with respect to the norm $\triplenorm{\cdot}_{\CB}$,
which is well-defined since, by assumption, $\H_h(\CB) $ is closed.

\begin{lemma} [single-step approximation]
\label{H-Matrix Approximation-abstract} 
\label{H-Matrix Approximation-Abstract}
 Let $2\diam(\Omega) \geq \delta \geq 2C_{\rm Set}(N)$ with the constant $C_{\rm Set}(N)$ from (A3), 
$\CB$  be a given collections of boxes and $\CB \subset \CB^{\delta/2}\subset \CB^\delta$ be enlarged boxes of $\CB$. 
 Let $\triplenorm{\cdot}_{\CB^\delta}$ be a norm on $\mathbf{H}^1(\CB^\delta)$ such that {\rm (A2)} holds for the sets $\CB \subset \CB^{\delta/2}$.  
 Let $\mathbf{v} \in \H_h(\CB^\delta)$ meaning that 
 {\rm (A3)} holds for the sets $\CB^{\delta/2},\CB^\delta$. Then, there exists a space $\mathbf{W}_1$ of dimension 
 $\dim \mathbf{W} _1 \leq C_{\rm ssa}  \left(\frac{\diam(\CB^\delta)}{\delta}\right) ^ {\alpha Md}$
 such that 
\begin{align*}
 \inf_{\mathbf{w}\in \mathbf{W}_1}\triplenorm{\Bv - \mathbf{w}}_{\CB}   \le \frac{1}{2} \triplenorm{\Bv}_{\CB^\delta}.
  \end{align*}
\end{lemma}
\begin{proof}
We set $\mathbf{W}_1 := \Pi_{h,\mathcal{B}}Q_H \H_h(\CB^\delta) \subset \mathbf{V}_{H,\CB^\delta} $.
Since $\mathbf{v} \in \H_h(\CB^\delta)$, we obtain from (A2) and (A3) that
\begin{align}\label{71}
 \triplenorm{\Bv - \Pi_{h,\mathcal{B}}Q_H \Bv }_{\CB} &= \triplenorm{\Pi_{h,\mathcal{B}}(\Bv - Q_H \Bv)}_{\CB} \leq 
 \triplenorm{\Bv - Q_H \Bv}_{\CB} \leq C_{\rm Qap}H( \norm{\nabla\Bv}_{L^2(\CB^{\delta/2})} + 2\delta^{-1}\triplenorm{\Bv}_{\CB^{\delta/2}}) \nonumber \\
 &\leq C_1 C_{\rm Qap}C_{\rm Cac}\frac{\diam(\CB^{\delta/2})^{\alpha-1}}{\delta^\alpha} H \triplenorm{\Bv}_{\CB^\delta}
\end{align}
with a constant $C_1$ depending only on $\Omega$ since $\alpha \geq 1$ and $\delta \leq 2 \diam(\Omega)$.
With the choice $H = \frac{\delta^\alpha}{2C_1C_{\rm Qap}C_{\rm Cac}\diam(\CB^\delta)^{\alpha-1}}$, we get the asserted error bound.
Since $\mathbf{W}_1 \subset \mathbf{V}_{H,\CB^\delta}$ and by choice of $H$, we have
\begin{align*}
\dim \mathbf{W}_1 \le C  \left(  \frac{\diam(\CB^\delta)}{H}\right) ^ {Md} 
\le C\left( 2C_1 C_{\rm Qap}C_{\rm Cac} \frac{\diam(\CB^\delta)^\alpha}{\delta^\alpha} \right) ^ {Md} =: C_{\rm ssa}  \left(\frac{\diam(\CB^\delta)}{\delta} \right) ^ {\alpha Md} , 
\end{align*}
which concludes the proof.
\end{proof}

Iterating the single-step approximation on concentric boxes leads to exponential convergence.

\begin{lemma}[multi-step approximation] 
\label{multi step approximation}
Let $L \in \N$ and 
$\delta \geq 2 C_{\rm Set}(N)$ with the constant $C_{\rm Set}(N)$ from (A3).  Let
 $\CB$ be a collection of boxes and $\CB^{\delta L} \supset \CB$ a collection of $\delta L$-enlarged boxes.
% $\delta L = \operatorname*{dist}(\CB,\partial \CB^{\delta L})$ as well as 
 %$\diam(\CB^{\delta L}) \leq \diam(\CB) + C \delta L$ with a constant $C>0$ that does only depend on $d$. 
Then, there exists a space 
 $\mathbf{W}_L \subseteq \mathcal{H}_h({\CB ^{\delta L}})$ such that 
 for all $\Bv \in \mathcal{H}_h(\CB ^{\delta L}) $ we have
\begin{align*}
 \inf_{\mathbf{w}\in \mathbf{W}_L }\triplenorm{\Bv - \mathbf{w}}_{ \CB}   \le 2 ^ {-L } \triplenorm{\Bv}_{\CB ^{\delta L}},
  \end{align*} 
  and 
\begin{align*}
\dim \mathbf{W} _L  \leq C _ {\rm dim}\Big(L+ \frac{\diam (\CB)}{\delta}\Big)^{\alpha Md+1}.
\end{align*}
\end{lemma}
\begin{proof}
The assumptions on $\CB$ and $\CB^{\delta L}$ allow for the construction of
a sequence of nested enlarged boxes $\CB \subseteq \CB ^\delta \subseteq \CB^{2\delta} \subseteq \ldots \subseteq \CB ^{\delta L}$ satisfying 
%$\delta \leq \operatorname*{dist}(\CB^{\ell\delta},\partial \CB^{(\ell+1)\delta})$ and
$\diam (\CB ^{\ell \delta}) \le \diam (\CB) + C\ell \delta $.

We   iterate the approximation result of Lemma~\ref{H-Matrix Approximation-abstract} on the sets $\CB^{\delta\ell}$, $\ell = L,\ldots,1$. 
For $\ell =L$,  Lemma~\ref{H-Matrix Approximation-abstract} applied with the sets $\CB^{(L-1)\delta} \subset \CB^{\delta L}$  
provides  a subspace 
$\mathbf{V}_1 \subset \mathcal{H}_N(\CB ^{\delta L})$ with 
$
\dim \mathbf{V}_1 \le  C \Big(\frac{\diam (\CB^{\delta L })}{\delta} \Big)^{\alpha Md}
$
such that 
\begin{align} \label{eq:approxchoiceH-abstract}
\inf_{\widehat{\mathbf{v}}_1 \in \mathbf{V}_1}	\triplenorm{\Bv -\widehat{\mathbf{v}}_1 }_{\CB^{(L-1)\delta} }
	\le 2 ^ {-1} \triplenorm{\Bv}_{\CB ^{\delta L}}.
	\end{align}
For  $\widehat{\mathbf{v}}_1 \in \mathbf{V}_1 $, we have  $(\Bv - \widehat{\mathbf{v}}_1)\in 
		\mathcal{H}_N(\CB ^{(L-1)\delta})$, 
so we can use Lemma \ref{H-Matrix Approximation-abstract}  again with the sets $\CB^{(L-2)\delta} \subset \CB^{(L-1)\delta}$, and get a subspace 
	$\mathbf{V}_2$ of $\mathcal{H}_N(\CB ^{(L-2)\delta})$ with 
 $\dim \mathbf{V}_2 \le  C\big(\frac{\diam (\CB ^{(L-1)\delta})}{\delta}\big)^{\alpha Md} $.
	This implies
	 \begin{align} \label{eq:approxchoiceH2-abstract}
	 \inf_{\widehat{\mathbf{v}}_2 \in \mathbf{V}_2} \inf_{\widehat{\mathbf{v}}_1 \in \mathbf{V}_1}	\triplenorm{(\Bv - \widehat{\mathbf{v}}_1) -\widehat{\mathbf{v}}_2 }_{\CB^{(L-2)\delta} }
	 \leq  2^{-1} \inf_{\widehat{\mathbf{v}}_1 \in \mathbf{V}_1}	\triplenorm{\Bv - \widehat{\mathbf{v}}_1}_{\CB^{(L-1)\delta} }
	 \le 2 ^ {-2} \triplenorm{\Bv}_{\CB ^{\delta L}}.
	 	\end{align}
	 	Continuing this process $L -2$ times leads to the subspace 
	 	$\mathbf{W}_L := \bigoplus  \limits _{\ell=1} ^L \mathbf{V}_\ell$ 
	 	of $\mathcal{H}_N({\CB ^{\delta L}})$
	 	with   dimension 
	 	\begin{align*}
	 	\dim \mathbf{W}_L  &\le C  \sum_{\ell=1}^{L}\Big(\frac{\diam (\CB ^{\delta\ell})}{\delta} \Big)^{\alpha Md}
	 	\le C \sum_{\ell=1}^{L}\Big(\frac{\diam (\CB)}{\delta}+ \ell  \Big)^{\alpha Md}\\
	 	& \leq C _ {\rm dim}\Big( L+ \frac{\diam (\CB)}{\delta} \Big)^{\alpha Md+1},
	 	\end{align*}
	 which finishes the proof.
\end{proof}

\subsection{Application of the abstract framework for the FEM-BEM couplings}
\label{sec:proofThmHmatrix}

In this section, we specify the assumptions (A1)--(A3) for the FEM-BEM couplings. 

\subsubsection{The local dual basis}
In the setting of Section~\ref{sec:abstractMatrixtoFunction}, we have $\mathbf{X}= H^1(\Omega) \times H^{-1/2}(\Gamma)$.
In order to suitably represent the data $f,u_0,\varphi_0$ in \eqref{eq:model}, we understand the discrete space 
$S^{1,1}(\mathcal{T}_h) \simeq S^{1,1}_0(\mathcal{T}_h) \times S^{1,1}(\mathcal{K}_h)  \subset L^2(\Omega)\times L^2(\Gamma)$, where 
$S^{1,1}_0(\mathcal{T}_h) := S^{1,1}(\mathcal{T}_h) \cap H^1_0(\Omega)$. Having identified $S^{1,1}(\mathcal{T}_h)$ with 
$S^{1,1}_0(\mathcal{T}_h) \times S^{1,1}(\mathcal{K}_h)$, we view the full FEM-BEM coupling problem as one as approximating 
in 
$S^{1,1}_0(\mathcal{T}_h) \times S^{1,1}(\mathcal{K}_h) \times S^{0,0}(\mathcal{K}_h)$. That is, we set 
$k=1$ and $\ell=2$, and consider $\mathbf{L}^2 = L^2(\Omega) \times L^2(\Gamma)\times L^2(\Gamma)$ for all three FEM-BEM couplings.
The discrete space $\mathbf{X}_N =S^{1,1}_0(\mathcal{T}_h) \times S^{1,1}(\mathcal{K}_h) \times S^{0,0}(\mathcal{K}_h) \subset \mathbf{L}^2$
has dimension $N = n_1+n_2+m$, where $n_1 = \operatorname*{dim}(S^{1,1}_0(\T_h))$, $n_2 = \operatorname*{dim}(S^{1,1}(\mathcal{K}_h))$ ($n_1+n_2 = n$) and $m = \operatorname*{dim}(S^{0,0}(\mathcal{K}_h))$, 
and it remains to show (A1). \medskip
 
The dual functions $\boldsymbol{\lambda}_i$ are constructed by use of $L^2$-dual bases for 
$S^{1,1}(\T_h)$ and $S^{0,0}(\mathcal{K}_h)$. 
% The right-hand side of Proposition~\ref{pro:Hmatrix} also
% suggests, that the boundary degrees of freedom in $S^{1,1}(\T_h)$ need to be treated separately.
% To that end, we split the nodes of $\T_h$ into interior nodes indexed by $\mathcal{I}_\Omega$ 
% and boundary nodes indexed by $\mathcal{I}_\Gamma$. We write 
% $\mathcal{I} = \mathcal{I}_\Omega \cup \mathcal{I}_\Gamma  \cup \mathcal{J}$ with 
% $\mathcal{J}:=\{N+1,\dots,N+M\}$. 
\cite[Sec.~3.3]{AFM20} gives an explicit construction of a suitable dual basis $\{\lambda_i^\Omega\,:\, i=1,\dots,n_1\}$ 
for $S^{1,1}_0(\T_h)$.
This is done elementwise in a discontinuous fashion, i.e., $\lambda_i^\Omega \in S^{1,0}(\T_h) \subset L^2(\Omega)$,
where each $\lambda_i^\Omega$ is non-zero only on one element of $\T_h$ (in the patch of the hat function $\xi_i$), and the function on this element is 
given by the push-forward of a dual shape function on the reference element. Moreover, the local stability 
estimate
\begin{align}
\label{eq:stab-1}
\Big\|\sum_{j=1}^n \mathbf{x}_j \lambda_j^\Omega\Big\|_{L^2(\Omega)}\leq h^{-d/2} \norm{\mathbf{x}}_2 
\end{align}
holds for all $\mathbf{x}\in \R^n$, and we have $\operatorname{supp} \lambda_i^\Omega \subset \operatorname{supp} \xi_i$.
We note that the zero boundary condition is irrelevant for the construction. 
The same can be done for the boundary degrees of freedom, i.e., 
there exists a dual basis $\{\lambda_i^\Gamma\,:\, i=1,\dots,n_2\}$ with 
the analogous stability and support properties.

For the boundary degrees of freedom in $S^{0,0}(\mathcal{K}_h)$, the dual mappings
are given by $\mu_i^\Gamma := \chi_i/\norm{\chi_i}_{L^2(\Omega)}^2$, i.e., the dual basis coincides -- up to scaling -- with the given basis 
$\{\chi_i\,:\, i=1,\dots, m\}$ of $S^{0,0}(\mathcal{K}_h)$. With \eqref{eq:basisb}, this gives 
\begin{align}
\label{eq:stab-2}
\Big\|\sum_{j=1}^m \mathbf{y}_j \mu_j^\Gamma\Big\|_{L^2(\Omega)}\leq h^{-(d-1)/2} \norm{\mathbf{y}}_2
\end{align}
for all $\mathbf{y} \in \R^m$. \medskip

Now, the dual basis is defined as 
$\boldsymbol{\lambda}_i := (\lambda^{\Omega}_i,0,0)$ for $i = 1,\dots,n_1$,  
$\boldsymbol{\lambda}_{i+n_1} := (0,\lambda^{\Gamma}_i,0)$ for $i = 1,\dots,n_2$ and 
$\boldsymbol{\lambda}_{i+n} := (0,0,\mu^{\Gamma}_i)$ for $i = 1,\dots,m$, 
and \eqref{eq:stab-1}, \eqref{eq:stab-2} together with the analogous one
for the $\lambda_i^\Gamma$ show (A1).

% \begin{align*}
% \lambda_i^\Gamma : L^2 (\Gamma) \rightarrow \mathbb{R} ^{\tau},  
% u \mapsto \left\lbrace  \begin{array}{cccc}
% \overline{u}_i & i \in \tau \cap \mathcal{J}\\0 & \rm elsewhere 
% \end{array};\right. \quad
% \Psi_\tau: 
% \mathbb{R}^\tau  \rightarrow S^{0,0}(\mathcal{K}_h), 
% \mathbf{x} \mapsto \sum _{i\in\tau\cap\mathcal{J}}x_i \chi_i,
% \end{align*}
% where $\overline{u}_i$ denotes the mean value of $u$ on the element $K_i$.

\subsubsection{Low dimensional approximation}

{\bf The sets $\CB$, $\CB^\delta$ and the norm $\triplenorm{\cdot}_{\CB}$}

We take $M=3$ and choose collections $\CB = \CB_R := \{B_R\cap\Omega,B_R,B_R\backslash\Gamma\}$, where $B_R$ is a box 
of side length $R$. For $ \ell \in \N$ the enlarged sets $\CB^{\delta \ell}$ then have the form
\begin{align}\label{def:setsconcrete}
\CB^{\delta \ell} = \CB_R^{\delta \ell} := \{B_{R+2\delta \ell}\cap\Omega,B_{R+2\delta \ell},B_{R+2\delta \ell}\backslash\Gamma\} \quad
\end{align}
with the concentric boxes $B_{R+2\delta \ell}$ of side length
$R+2\delta \ell$. 
%\mf{Let $B$, $D \subset \R^d$. For $i=1,3$  we define the set distances away from $\Gamma$ by $\operatorname{dist}_i^*(B,D) := \operatorname{dist}(B\backslash\Gamma,D\backslash\Gamma)$. For $i=2$, we simply set
%$\operatorname{dist}_i^*(B,D) := \operatorname{dist}(B,D)$, where in all cases $\operatorname{dist}$ denotes the Euclidean distance of two sets. Consequently, 
%we obtain
%$\operatorname*{dist}(\CB_R,\partial \CB_R^{\delta \ell}) = \delta \ell$ and}
%$$
%\diam(B_{R+2\delta \ell})= \sqrt{d}(R+2\delta \ell) \leq \diam(B_R) + \sqrt{d}2\delta,
%$$
%i.e., we have $C=2\sqrt{d}$ in \eqref{eq:diameters}.

For $\mathbf{v} = (u,v,w)$, we
use the norm from \eqref{eq:def:triplenormVec} 
$$\triplenorm{\mathbf{v}}_{\CB}:= \triplenorm{(u,v,w)}_{h,R}$$ 
in (A2).
For the Bielak-MacCamy coupling, taking $M=2$ and choosing collections $\CB_R := \{B_R\cap\Omega,B_R\}$ would suffice, however, in order to keep the notation short, we 
can use $M=3$ for this coupling as well by setting the third component to zero, i.e., $\mathbf{v} =(u,v,0)$. \bigskip

\noindent
{\bf The operator $Q_H$ and (A2)}

For the operator $Q_H$, we use a combination of localization and Scott-Zhang interpolation, introduced in \cite{ScottZhang}, on a coarse grid. 
Since the double-layer potential is discontinuous across $\Gamma$, we need to 
employ 
a piecewise Scott-Zhang operator.
Let $\mathcal{R}_H$ be a quasi-uniform (infinite) triangulation of $\mathbb{R} ^d$ (into open simplices $R \in {\mathcal R}_H$) 
with mesh width $H$ that conforms to $\Omega$, i.e., every $R \in {\mathcal{R}_H}$ satisfies either 
$R \subset \Omega$ or $R\subset \Omegaext$ and the restrictions 
${\mathcal R}_H|_{\Omega}$ and ${\mathcal R}_H|_{\Omegaext}$ are $\gamma$-shape regular, regular
triangulations of $\Omega$ and $\Omegaext$ of mesh size $H$, respectively. 

With the Scott-Zhang projections  $I_H^{\rm int}$, $I_H^{\rm ext}$
for the grids $\mathcal{R}_H|_{\Omega}$ and $\mathcal{R}_H|_{{\Omega}^c}$, we define the 
operator 
$I_H^{\rm pw}:H^1(\mathbb{R}^d\setminus \Gamma) \rightarrow 
S^{1,1}_{\rm pw}(\mathcal{R}_H):= \{v\,:\, v|_\Omega \in S^{1,1}({\mathcal R}_H|_{\Omega})\ \mbox{ and } \ 
v|_{\Omegaext} \in S^{1,1}({\mathcal R}_H|_{\Omegaext})\}$ in a piecewise fashion by 
\begin{equation}\label{eq:pwInterpolation}
I_H^{\rm pw} v = \left\{
\begin{array}{l}
 I_H^{\rm int} v \quad \text{on } \Omega, \\
 I_H^{\rm ext} v \quad \text{on } \Omegaext.
 \end{array}
 \right.
\end{equation}
We denote the patch of an element $R \in \mathcal{R} _H$ by 
\begin{align*}
 \omega _R^{\Omega} & := \operatorname{interior} 
\left( \bigcup \left\lbrace \overline{R^\prime}\,\colon\,  R^\prime \in \mathcal{R}_H|_{\Omega} \quad \text{s.t.}\quad \overline{R} \cap \overline{R^ \prime} \neq
 \emptyset \right\rbrace\right),  \\
  \omega _R^{\Omegaext} & := \operatorname{interior} 
\left( \bigcup \left\lbrace \overline{R^\prime}  \,\colon\, R^\prime \in \mathcal{R}_H|_{\Omegaext} \quad \text{s.t.} \quad \overline{R} \cap \overline{R^ \prime} \neq
 \emptyset \right\rbrace\right).
\end{align*}
The Scott-Zhang projection reproduces piecewise affine functions and 
has the following local approximation property for piecewise $H^s$ functions: 
\begin{align}
\left\| v - I_H^{\rm pw} v\right\|^2_{H ^ t( R) } 
\le C H ^ {2 (s-t) }
\begin{cases} 
\abs{v}^2_{H^{s}(\omega_R^\Omega)} &\mbox{ if $R \subset \Omega$} \\
\abs{v}^2_{H^{s}(\omega_R^{\Omegaext})} &\mbox{ if $R \subset \Omegaext$} 
\end{cases}
\quad t, s \in \{0,1\}, \quad 0 \leq t \leq  s\le 1,
\end{align}
with a constant $C$  depending  only on the shape-regularity of $\mathcal{R}_H$ and $d$.

 Let $ \eta \in C^ \infty _0 ( B _{R+2\delta} ) $ be a cut-off function satisfying  
 $ \operatorname{supp} \,\eta \subset   B_{R+\delta}, \, \eta \equiv 1$ on $B_R$ and $\norm{\nabla \eta}_{L^\infty(\R^d)}\lesssim \frac{1}{\delta}$.  
We define the operator 
 \begin{align}
  Q_H\mathbf{v} := (I_H^{\rm int}(\eta \mathbf{v}_1), I_H(\eta \mathbf{v}_2), I_H^{\rm pw}(\eta \mathbf{v}_3)),
 \end{align}
where $I_H$ denotes the classical Scott-Zhang operator for the mesh $\mathcal{R}_H$.
We have
\begin{align*}
\triplenorm{\mathbf{v}- Q_H\mathbf{v}}_{\CB}^2 = \norm{\mathbf{v}_1- I_H^{\rm int}(\eta\mathbf{v}_1)}_{h,R,\Omega}^2 +  \norm{\mathbf{v}_2- I_H(\eta\mathbf{v}_2)}_{h,R}^2  +
 \norm{\mathbf{v}_3- I_H^{\rm pw}(\eta\mathbf{v}_3)}_{h,R,\Gamma^c}^2.  
\end{align*}
Each term on the right-hand side can be estimated with the same arguments. We only work out the details for the second component. Assuming $h\leq H$, and using approximation properties and stability of the Scott-Zhang projection, we get 
\begin{align*}
 \norm{\mathbf{v}_2- I_H(\eta \mathbf{v}_2)}_{h,R}^2  &=  \norm{\eta\mathbf{v}_2- I_H(\eta\mathbf{v}_2)}_{h,R}^2  = 
 h^2 \norm{\nabla(\eta\mathbf{v}_2- I_H(\eta\mathbf{v}_2))}_{L^2(B_R)}^2 + \norm{\eta\mathbf{v}_2- I_H(\eta\mathbf{v}_2)}_{L^2(B_R)}^2  \\
&\lesssim (h^2+H^2) \norm{\nabla(\eta\mathbf{v}_2)}_{L^2(\R^d)}  \lesssim H^2\left(\norm{\nabla\mathbf{v}_2}_{L^2(B_{R+2\delta})}^2 + \delta^{-1}\norm{\mathbf{v}_2}_{L^2(B_{R+2\delta})}^2 \right),
\end{align*}
which shows (A2) for the discrete space $V_{H,\CB^\delta} = S^{1,1}(\mathcal{R}_H)|_{B_{R+2\delta}\cap\Omega} \times  S^{1,1}(\mathcal{R}_H)|_{B_{R+2\delta}}  \times  S^{1,1}_{\rm pw}(\mathcal{R}_H)|_{B_{R+2\delta}} $ of dimension $\dim V_{H,\CB^\delta} \leq C \left(\frac{\diam(B_{R+2\delta})}{H}\right)^{Md}$. 
\bigskip

\noindent
{\bf The Caccioppoli inequalities and (A3)}

Theorem~\ref{th:CaccioppoliBMcC}--Theorem~\ref{th:CaccioppoliJN} provide the Caccioppoli type estimates 
 asserted in (A3) with $\delta = \varepsilon R/2$. 
For the Bielak-MacCamy coupling we have $\alpha=1$ and $C_{\rm Set} =8h$, for the symmetric coupling  $\alpha=1$ and $C_{\rm Set} =16h$.
For the Johnson-N\'ed\'elec we have to take $\alpha =2$ and $C_{\rm Set} =16h$. For $\CB_R = \{B_R\cap\Omega,B_R,B_R\backslash\Gamma\}$,
the spaces $\mathcal{H}_h(\CB_R)$ can be characterized by 
{\begin{align*}
	\mathcal{H}_h(\CB_R)  :=& \{(v,\widetilde{V}\phi,\widetilde{K}v) \in H^1(B_R\cap \Omega)\times H^1(B_R) \times H^1(B_R\backslash \Gamma)
	\,\colon\, \exists \widetilde v \in S^{1,1}(\T_h), 
	\widetilde \phi \in S^{0,0}(\mathcal{K}_h):\\ 
	& \quad 
	\widetilde v|_{B_R \cap \Omega} =v|_{B_R \cap \Omega}, \quad \widetilde{V} \widetilde \phi |_{B_R} =\widetilde{V}\phi |_{ B_R}, \quad
	\widetilde{K} \widetilde v|_{B_R\backslash\Gamma} =\widetilde{K} v|_{ B_R\backslash\Gamma},
	\quad a(v,\phi;\psi_h,\zeta_h) = 0 
	\\ 
	&\quad 
	\forall (\psi_h,\zeta_h) \in S^{1,1}(\T_h) \times S^{0,0}(\mathcal{K}_h), \,
	\operatorname*{supp} \psi_h,\zeta_h \subset B_R \}, 
\end{align*}
where the bilinear form $a(\cdot,\cdot)$ is either $a_{\rm sym}$ or $a_{\rm jn}$.
For the Bielak-MacCamy coupling, it suffices to require 
{\begin{align*}
	\mathcal{H}_h(\CB_R)  :=& \{(v,\widetilde{V}\phi,0) \in H^1(B_R\cap \Omega)\times H^1(B_R) \times H^1(B_R\backslash \Gamma)
	\,\colon\, \exists \widetilde v \in S^{1,1}(\T_h), 
	\widetilde \phi \in S^{0,0}(\mathcal{K}_h):\\ 
	& \quad 
	\widetilde v|_{B_R \cap \Omega} =v|_{B_R \cap \Omega}, \quad \widetilde{V} \widetilde \phi |_{B_R} =\widetilde{V}\phi |_{ B_R},
	\quad a_{\rm bmc}(v,\phi;\psi_h,\zeta_h) = 0 
	\\ 
	&\quad 
	\forall (\psi_h,\zeta_h) \in S^{1,1}(\T_h) \times S^{0,0}(\mathcal{K}_h), \,
	\operatorname*{supp} \psi_h,\zeta_h \subset B_R \}. 
\end{align*}
With these definitions, the closedness and nestedness of the spaces $\mathcal{H}_h(\CB_R)$ clearly holds.

\subsubsection{Proof of Theorem~\ref{th:H-Matrix approximation of inverses}}
As a consequence of the above discussions, the abstract framework of the previous sections can be applied and it remains to put everything together.

%In the following, we apply different $L^2$-orthogonal projections, namely, 
%$\Pi ^{L^2(\Omega)}: L^2 (\Omega) \rightarrow S^{1,1}(\T_h)$, 
%$\Pi _1^{L^2(\Gamma)}: L^2 (\Gamma) \rightarrow S^{1,1}(\mathcal{K}_h)$ and 
%$\Pi _0^{L^2(\Gamma)}: L^2 (\Gamma) \rightarrow S^{0,0}(\mathcal{K}_h)$.

The following proposition constructs the finite dimensional space required from Lemma~\ref{th:H-Matrix approximation of inverses-abstract},
from which the Galerkin solution can be approximated exponentially well.

\begin{Proposition}[low dimensional approximation for the symmetric coupling] 
\label{pro:Hmatrix}
Let $(\tau , \sigma)$ be a cluster pair with bounding boxes $B_{R_\tau}$ and $B_{R_\sigma}$ that satisfy for given $\eta > 0$
$$\eta  \operatorname*{dist}(B_{R_\tau},B_{R_\sigma}) \ge \diam (B_{R_\tau}).$$ 
%Fix $q \in (0,1)$. 
%
Then, for each $L \in \mathbb{N}$, there exists a space 
$\widehat{\mathbf{W}}_L \subset S^{1,1}(\T_h) \times S^{0,0}(\mathcal{K}_h)$ with 
dimension $\dim \widehat{\mathbf{W}}_L \le C_{\rm low} L ^{3d+1}$ such that 
for arbitrary right-hand sides $f \in L^2 (\Omega)$, $v_0 \in L^2(\Gamma)$, and $w_0 \in {L^2(\Gamma)}$
with $\big(\operatorname*{supp} f \cup \operatorname*{supp}v_0 \cup 
\operatorname*{supp}w_0\bigr) \subset B_{R_\sigma},$ 
%that, in the case $d = 2$,  \jmm{warum brauchen wir die compatibility condition?} additionally satisfy the compatibility condition  $\skp{f,1}_{L^2(\Omega)}+\skp{v_0,1}_{L^2(\Gamma)}=0$,
the corresponding Galerkin solution $(u_h,\varphi_h)$ of \eqref{eq:symmcoupling} satisfies 
\begin{align*}
\min\limits_{\left( \widetilde{u}, \widetilde{\varphi}\right)  \in \widehat{\mathbf{W}}_L} 
\left(\left\| u_h- \widetilde{u}\right\|_{L^2 (B_{R_\tau} \cap \Omega)}+
\left\| \varphi_h- \widetilde{\varphi}\right\|_{L^{2} (B_{R_\tau}\cap\Gamma)}\right)  
&\le C_{\rm{box}}h^{-2}2^{-L}\left(\left\|f\right\|_ {L^2 (\Omega)} 
+\left\|v_0 \right\|_ {L^{2} (\Gamma)}  +\left\|w_0 \right\|_ {L^{2} (\Gamma)} \right).
\end{align*}
The constants $C_{\rm low}$, $C_{\rm box}$ depend only on 
	$\Omega$, $d$, $\eta$, and the $\gamma$-shape regularity of the quasi-uniform triangulation 
	$\mathcal{T}_h$ and $\mathcal{K}_h$.
\end{Proposition}
\begin{proof} 
%We distinguish two cases.
%
%\textbf{Case 1.} Let  
%$\frac{h}{R_\tau} \ge \frac{q \rho }{32 k \max \left\lbrace 1 , C _{\text{app}} \right\rbrace }$. 
%Then, we define 
%$X_k:= \left\lbrace (v|_ {B_{R_\tau}\cap \Omega}, \phi|_ {B_{R_\tau}\cap \Gamma})  :\; (v, \phi) \in S^{1,1}(\T_h) \times S^{0,0}(\mathcal{K}_h)\right\rbrace$. 
%The error estimate holds trivially, and
%\begin{align*}
%\dim X_k \lesssim \left( \frac{R_\tau }{h}\right)^{2d} \lesssim 
%\left( \frac{32 k \max \left\lbrace 1 , C _{\text{app}}\right\rbrace }{q\rho}\right) ^{2d}  
%\lesssim \left((1+\eta) q ^ {-1}\right)^{3d} k ^{3d+1}.
%\end{align*}
%
%\textbf{Case 2}. The condition  \eqref{eq:proofHmatrixlemma1} is satisfied with $R =R_ \tau$.
For given $L \in \N$, we choose $\delta:= \frac{R_{\tau}}{2\eta L}$ . Then, we have
\begin{align*}
\operatorname*{dist}(B_{R_\tau+2\delta L},B_ {R_\sigma}) \ge 
\operatorname*{dist} (B_ {R_\tau},B_ {R_\sigma})-  L \delta \sqrt{d} 
\ge \sqrt{d} R_\tau \Big(\frac{1}{\eta}-\frac{1}{2\eta}\Big) >0.
\end{align*} 
With $\CB_{R_\tau} = \{B_{R_\tau}\cap\Omega,B_{R_\tau},B_{R_\tau}\backslash\Gamma\}$ and $\CB^{\delta L}_{R_\tau} = \{B_{R_\tau+2\delta L}\cap\Omega,B_{R_\tau+2\delta L},B_{R_\tau+2\delta L}\backslash\Gamma\}$ from \eqref{def:setsconcrete},
the assumption on the support of the data therefore implies the local orthogonality imposed in the 
space $\mathcal{H}_h(\CB^{\delta L}_{R_\tau})$. 
%Moreover, we have $\operatorname*{dist}(\CB_{R_\tau},\CB^{\delta L}_{R_\tau}) = \delta L$ and $\diam(\CB_{R_\tau}^{\delta L}) \leq \diam(\CB_{R_\tau}) + 2\sqrt{d}L$.
In order to define the space $\widehat{\mathbf{W}}_L$, we distinguish two cases.

{\bf Case $\delta > 2C_{\rm Set}$:} Then, Lemma~\ref{multi step approximation}  applied with the sets $\CB_{R_\tau}^\delta$ and $\CB_{R_\tau}^{\delta L}$
provides a space $\mathbf{W}_L$ of dimension 
\begin{align*}
\dim \mathbf{W} _L  \leq C _ {\rm dim}\Big(L-1+ \frac{\diam (\CB_{R_\tau}^\delta)}{\delta}\Big)^{3d+1} \lesssim 
\Big(L+ \frac{\sqrt{d}R_\tau 2\eta L}{R_\tau}\Big)^{3d+1} \lesssim L^{3d+1}
\end{align*}
with the approximation properties for $\mathbf{v} = (u_h,\widetilde{V}\varphi_h,\widetilde{K}u_h)$
\begin{align}\label{eq:lowdimtemp1}
 \inf_{\mathbf{w}\in \mathbf{W}_L }\triplenorm{\Bv - \mathbf{w}}_{ \CB_{R_\tau}^\delta}   \le 2 ^ {-(L-1) } \triplenorm{\Bv}_{\CB_{R_\tau}^{\delta L}}.
  \end{align} 
Therefore, it remains to estimate the norm $\triplenorm{\cdot}_{\CB}$ from above and below.

With $h\lesssim 1$, the mapping properties of $\widetilde{V}$ and $\widetilde{K}$
from \eqref{eq:mapping-Ktilde}, and the trace inequality 
we can estimate
\begin{align}\label{eq:proofHmatrixpro3}
\nonumber\triplenorm{(u_h,\widetilde{V}\varphi_h,\widetilde{K}u_h)}_{\CB^{\delta L}_{R_\tau}}  
&\lesssim 
\left\|u_h \right\|_ {H^1 (\Omega)}+
\left\|\widetilde{V}\varphi _h \right\|_{H ^1(B_{(1+1/(2\eta))R_\tau})} + 
\left\|\widetilde{K}u_h \right\|_{H ^1(B_{(1+1/(2\eta))R_\tau}\backslash\Gamma)}
\\& \lesssim \left\|u_h \right\|_ {H^1 (\Omega)}+
\left\| \varphi _h\right\|_{H^{-1/2}(\Gamma)}. 
\end{align}
The stabilized form 
$\widetilde{a}_{\rm sym}(u,\varphi;\psi,\zeta) := a_{\rm sym}(u,\varphi;\psi,\zeta) 
+ \skp{1,V\varphi+(\frac{1}{2}-K)u}_{L^2(\Gamma)}\skp{1,V\zeta+(\frac{1}{2}-K)\psi}_{L^2(\Gamma)}$
is elliptic, cf.~\cite{AFFKMP13}. Moreover, \cite[Thm.~{18}]{AFFKMP13} prove that the Galerkin solution also 
solves $\widetilde{a}_{\rm sym}(u_h,\varphi_h;\psi,\zeta) = g_{\rm sym}(\psi,\zeta) + 
\skp{1,w_0}_{L^2(\Gamma)}\skp{1,(\frac{1}{2}-K)\psi+V\zeta}_{L^2(\Gamma)}$. Therefore, we have 
\begin{align}\label{eq:proofHmatrixpro3b}
\nonumber 
\left\| \varphi _h\right\|^2_{H^{-1/2}(\Gamma)}+\left\|u_h \right\|^2_{H^1(\Omega)}  
\lesssim  \widetilde{a}_{\rm sym}(u_h, \varphi _h;u_h, \varphi _h) 
&= \skp{f,u_h}_{L^2(\Omega)} + 
\skp{v_0,u_h}_{L^2(\Gamma)} 
+\skp{w_0 , \varphi _h}_{L^2(\Gamma)} \\ 
&\quad+\skp{1,(1/2-K)u_h+V\varphi_h}_{L^2(\Gamma)}\skp{1,w_0}_{L^2(\Gamma)}. 
%\\&
%\nonumber
%=\skp{\Pi ^{L^2(\Omega)}f,u_h}_{L^2(\Omega)} + \skp{\Pi_1 ^{L^2(\Gamma)}v_0,u_h}_{L^2(\Gamma)} 
%+\skp{\Pi_0 ^{L^2(\Gamma)}w_0 , \varphi _h}_{L^2(\Gamma)} + \\
%&\quad+\skp{1,(1/2-K)u_h+V\varphi_h}_{L^2(\Gamma)}\skp{1,\Pi_0 ^{L^2(\Gamma)}w_0}_{L^2(\Gamma)}. 
\end{align}
The stabilization term can be estimated with the mapping properties of $V$ and $K$ from \eqref{eq:mapping-K} and the trace inequality by 
\begin{align*}
\abs{\skp{1,(1/2-K)u_h+V\varphi_h}_{L^2(\Gamma)}\skp{1,w_0}_{L^2(\Gamma)}} 
&\lesssim \left(\norm{(1/2-K)u_h}_{L^2(\Gamma)}+\norm{V\varphi_h}_{L^2(\Gamma)}\right) 
\norm{w_0}_{L^2(\Gamma)}\\
&\lesssim \norm{w_0}_{L^2(\Gamma)}\left(\norm{u_h}_{H^1(\Omega)}
+\norm{\varphi_h}_{H^{-1/2}(\Gamma)}\right).
\end{align*}
Inserting this in \eqref{eq:proofHmatrixpro3b}, using
the trace inequality and an inverse estimate we further estimate
\begin{align*}
\nonumber 
\left\| \varphi _h\right\|^2_{H^{-1/2}(\Gamma)}+\left\|u_h \right\|^2_{H^1(\Omega)} &\lesssim   
\left(\left\|f\right\|_ {L^2 (\Omega)}
+\left\|v_0 \right\|_ {H^{-1/2} (\Gamma)}\right) \norm{u_h}_{H^1(\Omega)}\\
&\nonumber \quad +\left\|w_0 \right\|_ {L^{2} (\Gamma)} 
\left(\norm{\varphi_h}_{L^{2}(\Gamma)}+\norm{u_h}_{H^1(\Omega)}
+\norm{\varphi_h}_{H^{-1/2}(\Gamma)}\right) 
 \\&
 \nonumber  \le
\left(\left\|f\right\|_ {L^2 (\Omega)}
+\left\|v_0 \right\|_ {H^{-1/2} (\Gamma)}\right) \norm{u_h}_{H^1(\Omega)}  \\
 &
\quad 
+h^{-1/2}\left\|w_0 \right\|_{L^2(\Gamma)}\left( \norm{u_h}_{H^1(\Omega)}
+\norm{\varphi_h}_{H^{-1/2}(\Gamma)}\right). 
 \end{align*}
With Young's inequality and inserting this in \eqref{eq:proofHmatrixpro3}, we obtain the upper bound
\begin{align}\label{eq:proofHmatrixpro4}
\triplenorm{(u_h,\widetilde{V}\varphi_h,\widetilde{K}u_h)}_{\CB_{R_\tau}^{\delta L}}  
&\lesssim \left( \left\|f\right\|_ {L^2 (\Omega)}  
+\left\|v_0 \right\|_ {L^2(\Gamma)} 
+h ^{-1/2}\left\|w_0 \right\|_ {L^2(\Gamma)}\right).
\end{align}
The jump conditions of the single-layer potential and Lemma~\ref{lem:investV} provide 
for arbitrary $\widetilde{\varphi} \in S^{0,0}(\mathcal{K}_h)$
\begin{align}
\left\|\varphi_h -\widetilde{\varphi} \right\| _{L^2 (B_{R_\tau} \cap \Gamma) } =
\left\|[\gamma_1 \widetilde{V}\varphi _h ] 
-[\gamma_1 \widetilde{V} \widetilde{\varphi}]  \right\| _{L^2 (B_{R_\tau} \cap \Gamma) }  
&\lesssim  h^{-1/2}\left\| \nabla ( \widetilde{V}\varphi_h-\widetilde{V}\widetilde{\varphi}) \right\|_{L^2(B_{R+2\delta})} 
\nonumber \\& 
 \lesssim h^{-3/2}\triplenorm{ \widetilde{V}\varphi_h- \widetilde{V}\widetilde{\varphi} }_{h,R+2\delta}.
\end{align}
Finally, we define $\widehat{\mathbf{W}}_L := \{(\widetilde{u},[\gamma_1 \widetilde{v}])\,\colon\,
(\widetilde{u},\widetilde{v},\widetilde{w})\in \mathbf{W}_L\}$. 
Then, the dimension of $\widehat{\mathbf{W}}_L$ is bounded by $\widehat{\mathbf{W}}_L \le C L^{3d+1}$,
and the error estimate follows from \eqref{eq:lowdimtemp1} since 
\begin{align*}
\nonumber \inf\limits_{(\widetilde{u},\widetilde{\varphi}) \in \widehat{\mathbf{W}}_L}
\left( \left\| u_h- \widetilde{u}\right\| _{L^2 (B_{R_\tau}\cap\Omega)} 
+\left\|\varphi_h -\widetilde{\varphi} \right\| _{L^2 (B_{R_\tau} \cap \Gamma) } \right) 
&\lesssim h^{-3/2}\inf\limits_{\mathbf{w} \in \mathbf{W}_L}
\triplenorm{(u_h,\widetilde{V}\varphi_h,\widetilde{K}u_h)-\mathbf{w}}_{\CB_{R_\tau}^{\delta}}
 \\
&
\nonumber \lesssim h^{-3/2} 2 ^{-L} 
\triplenorm{(u_h,\widetilde{V}\varphi_h,\widetilde{K}u_h)}_{\CB_{R_\tau}^{\delta L}}. 
\end{align*}
Applying estimate \eqref{eq:proofHmatrixpro4} finishes the proof for the case $\delta \geq 2 C_{\rm set}.$

{\bf Case $\delta \leq 2 C_{\rm set} = 32h$:} Here, we use the space $\widehat{\mathbf{W}}_L:= S^{1,1}(\mathcal{T}_h)|_{B_{R_\tau}} \times S^{0,0}(\mathcal{K}_h)|_{B_{R_\tau}} $. Since 
$(u_h,\varphi_h)|_{B_{R_\tau}} \in \widehat{\mathbf{W}}_L$ the error estimate holds trivially. For the dimension of $\widehat{\mathbf{W}}_L$, we obtain
\begin{align*}
\operatorname{dim} \widehat{\mathbf{W}}_L \leq C \left(\frac{\diam(B_{R_\tau})}{h}\right)^{2d} \leq C
\left(\frac{32\sqrt{d} R_\tau }{\delta}\right)^{2d} \leq C 
\left(2C_{\rm set}\sqrt{d}2\eta L\right)^{2d} \lesssim L^{2d},
\end{align*}
which finishes the proof.
\end{proof}

\begin{Proposition}[low dimensional approximation for the Bielak-MacCamy coupling] 
\label{pro:HmatrixBMC}
Let $(\tau , \sigma)$ be a cluster pair with bounding boxes $B_{R_\tau}$ and $B_{R_\sigma}$ that satisfy for given $\eta > 0$
$$\eta  \operatorname*{dist}(B_{R_\tau},B_{R_\sigma}) \ge \diam (B_{R_\tau}).$$ 
%Fix $q \in (0,1)$. 
%
Then, for each $L \in \mathbb{N}$, there exists a space 
$\widehat{\mathbf{W}}_L \subset S^{1,1}(\T_h) \times S^{0,0}(\mathcal{K}_h)$ with 
dimension $\dim \widehat{\mathbf{W}}_L \le C_{\rm low} L ^{2d+1}$ such that 
for arbitrary right-hand sides $f \in L^2 (\Omega)$, $\varphi_0 \in L^2(\Gamma)$, and $u_0 \in {L^2(\Gamma)}$
with $\big(\operatorname*{supp} f \cup \operatorname*{supp}\varphi_0 \cup 
\operatorname*{supp}u_0\bigr) \subset B_{R_\sigma},$ 
%that, in the case $d = 2$,  additionally satisfy the compatibility condition  $\skp{f,1}_{L^2(\Omega)}+\skp{\varphi_0,1}_{L^2(\Gamma)}=0$,
the corresponding Galerkin solution $(u_h,\varphi_h)$ of \eqref{eq:BMcC} satisfies 
\begin{align*}
\min\limits_{\left( \widetilde{u}, \widetilde{\varphi}\right)  \in \widehat{\mathbf{W}}_L} 
\left(\left\| u_h- \widetilde{u}\right\|_{L^2 (B_{R_\tau} \cap \Omega)}+
\left\| \varphi_h- \widetilde{\varphi}\right\|_{L^{2} (B_{R_\tau}\cap\Gamma)}\right)  
&\le C_{\rm{box}}h^{-2}2^{-L}\left(\left\|f\right\|_ {L^2 (\Omega)} 
+\left\|\varphi_0 \right\|_ {L^{2} (\Gamma)}  +\left\|u_0 \right\|_ {L^{2} (\Gamma)} \right).
\end{align*}
The constants $C_{\rm low}$, $C_{\rm box}$ depend only on 
	$\Omega$, $d$, $\eta$, and the $\gamma$-shape regularity of the quasi-uniform triangulation 
	$\mathcal{T}_h$ and $\mathcal{K}_h$.
\end{Proposition}
\begin{proof}
The proof is essentially identical to the proof of Proposition~\ref{pro:Hmatrix}. We stress that the bound of the dimension $\dim \widehat{\mathbf{W}}_L \le C_{\rm low} L ^{2d+1}$ is 
better, since no approximation for the double-layer potential is needed, i.e., we can choose $M=2$ in the abstract setting. 
\end{proof}

\begin{Proposition}[low dimensional approximation for the Johnson-N\'ed\'elec coupling] 
\label{pro:HmatrixJN}
Let $(\tau , \sigma)$ be a cluster pair with bounding boxes $B_{R_\tau}$ and $B_{R_\sigma}$ that satisfy for given $\eta > 0$
$$\eta  \operatorname*{dist}(B_{R_\tau},B_{R_\sigma}) \ge \diam (B_{R_\tau}).$$ 
%Fix $q \in (0,1)$. 
%
Then, for each $L \in \mathbb{N}$, there exists a space 
$\widehat{\mathbf{W}}_L \subset S^{1,1}(\T_h) \times S^{0,0}(\mathcal{K}_h)$ with 
dimension $\dim \widehat{\mathbf{W}}_L \le C_{\rm low} L ^{6d+1}$, such that 
for arbitrary right-hand sides $f \in L^2 (\Omega)$, $\varphi_0 \in L^2(\Gamma)$, and $w_0 \in {L^2(\Gamma)}$
with $\big(\operatorname*{supp} f \cup \operatorname*{supp}\varphi_0 \cup 
\operatorname*{supp}w_0\bigr) \subset B_{R_\sigma},$ 
%that, in the case $d = 2$,  additionally satisfy the compatibility condition  $\skp{f,1}_{L^2(\Omega)}+\skp{\varphi_0,1}_{L^2(\Gamma)}=0$,
the corresponding Galerkin solution $(u_h,\varphi_h)$ of \eqref{eq:JN} satisfies 
\begin{align*}
\min\limits_{\left( \widetilde{u}, \widetilde{\varphi}\right)  \in \widehat{\mathbf{W}}_L} 
\left(\left\| u_h- \widetilde{u}\right\|_{L^2 (B_{R_\tau} \cap \Omega)}+
\left\| \varphi_h- \widetilde{\varphi}\right\|_{L^{2} (B_{R_\tau}\cap\Gamma)}\right)  
&\le C_{\rm{box}}h^{-2}2^{-L}\left(\left\|f\right\|_ {L^2 (\Omega)} 
+\left\|\varphi_0 \right\|_ {L^{2} (\Gamma)}  +\left\|w_0 \right\|_ {L^{2} (\Gamma)} \right).
\end{align*}
The constants $C_{\rm low}$, $C_{\rm box}$ depend only on 
	$\Omega$, $d$, $\eta$, and the $\gamma$-shape regularity of the quasi-uniform triangulation 
	$\mathcal{T}_h$ and $\mathcal{K}_h$.
\end{Proposition}
\begin{proof}
The proof is essentially identical to the proof of Proposition~\ref{pro:Hmatrix}. We stress that the bound of the dimension $\dim \widehat{\mathbf{W}}_L \le C_{\rm low} L ^{6d+1}$ is 
worse than for the other couplings, since in the abstract setting, we have to choose $M=3$ and $\alpha=2$, and the bound follows from Lemma~\ref{multi step approximation}.
\end{proof}

Finally, we can prove the existence of 
$\H$-Matrix approximants to the inverse FEM-BEM stiffness matrix.

%\mf{Here, use abstract framework of Sec.~4.1.}

\begin{proof}[Proof of Theorem \ref{th:H-Matrix approximation of inverses}]
We start with the symmetric coupling.
As $\mathcal{H}$ matrices are low rank only on admissible blocks, 
 we set $\mathbf{B_\H}|_{\tau \times \sigma}= \mathbf{A_{\rm sym}^{-1}}|_ {\tau \times \sigma}$ for non-admissible cluster pairs 
and consider
an arbitrary admissible cluster pair $(\tau,\sigma)$ in the following.

%{\bf Step 1.}
%We assume that the compatibility condition $\skp{f,1}_{L^2(\Omega)} + \skp{v_0,1}_{L^2(\Gamma)} = 0$
%is fulfilled by $f$, $v_0$ for $d=2$ in Proposition~\ref{pro:Hmatrix}.
 With a given rank bound $r$, we take $L := \lfloor (r/C_{\rm low})^{1/(3d+1)} \rfloor$.
With this choice, we apply Proposition~\ref{pro:Hmatrix}, which provides a space $\widehat{\mathbf{W}}_L \subset S^{1,1}(\mathcal{T}_h) \times S^{0,0}(\mathcal{K}_h)$
and use this space in Lemma~\ref{th:H-Matrix approximation of inverses-abstract}, which produces matrices $\mathbf{X}_{\tau\sigma},\mathbf{Y}_{\tau\sigma}$ of 
maximal rank $\dim \widehat{\mathbf{W}}_L$, which is by choice of $L$ bounded by
\begin{align*}
\dim \widehat{\mathbf{W}}_L =  C_{\rm low}L^{3d+1} \leq r.
\end{align*}
Proposition~\ref{pro:Hmatrix} can be rewritten in terms of the discrete solution operator of the framework of Section~\ref{sec:abstractMatrixtoFunction}.
 Let $\boldsymbol{f} = (f,v_0,w_0) \in \mathbf{L}^2$ be arbitrary with
$\operatorname*{supp}(\boldsymbol{f}) \subset 
\prod_j D_j(\sigma)$. Then, the locality of the dual functions implies $\big(\operatorname*{supp} f \cup \operatorname*{supp}v_0 \cup 
\operatorname*{supp}w_0\big) \subset B_{R_\sigma}$, and we obtain
\begin{align*}
\inf_{\mathbf{w} \in \widehat{\mathbf{W}}_L}\norm{\mathcal{S}_N \boldsymbol{f}  -\mathbf{w}}_{\mathbf{L}^2(\tau)} &\leq
\inf\limits_{\left( \widetilde{u}, \widetilde{\varphi}\right)  \in \widehat{\mathbf{W}}_L} 
\left(\left\| u_h- \widetilde{u}\right\|_{L^2 (B_{R_\tau} \cap \Omega)}+
\left\| \varphi_h- \widetilde{\varphi}\right\|_{L^{2} (B_{R_\tau}\cap\Gamma)}\right) \\ 
&\lesssim h^{-2}2^{-L}\left(\left\|f\right\|_ {L^2 (\Omega)} 
+\left\|v_0 \right\|_ {L^{2} (\Gamma)}  +\left\|w_0 \right\|_ {L^{2} (\Gamma)} \right) \lesssim h^{-2}2^{-L} \norm{\boldsymbol{f}}_{\mathbf{L}^2}.
\end{align*}
Defining $\mathbf{B}_\H|_{\tau\times \sigma} := \mathbf{X}_{\tau\sigma} \mathbf{Y}_{\tau\sigma}^T$, the estimates \eqref{eq:refinedblockestimate} and $\norm{\Lambda} \lesssim h^{-d/2}$ together with Lemma~\ref{th:H-Matrix approximation of inverses-abstract} then give the error bound
\begin{align*}
\norm{\mathbf{A}_{\rm sym}^{-1}-\mathbf{B}_\H}_2 &\leq C_{\rm sp} \operatorname*{depth}(\mathbb{T}_{\mathcal{I}})\max\{\norm{\mathbf{\mathbf{A}^{-1}-\mathbf{B}_\H}|_{\tau\times \sigma}}_2 : (\tau,\sigma) \in P\} \\
&\leq C_{\rm sp} \operatorname*{depth}(\mathbb{T}_{\mathcal{I}}) \norm{\Lambda}^2 \max_{ (\tau,\sigma) \in P_{\rm far}} 
\sup_{\substack{\boldsymbol{f}  \in \mathbf{L}^2:\\ \operatorname*{supp}(\boldsymbol{f}) \subset 
\prod_j D_j(\sigma)}}
 \frac{\inf_{\mathbf{w} \in \widehat{\mathbf{W}}_L}\norm{\mathcal{S}_N \boldsymbol{f}  -\mathbf{w}}_{\mathbf{L}^2(\tau)}}{\norm{\boldsymbol{f}  }_{\mathbf{L}^2}} \\
&\lesssim  C_{\rm sp} \operatorname*{depth}(\mathbb{T}_{\mathcal{I}}) h^{-(d+2)} 2^{-L} \\
&\leq C_{\rm apx}C_{\rm sp} \operatorname*{depth}(\mathbb{T}_{\mathcal{I}}) h^{-(d+2)} \exp(-b r^{1/(3d+1)}).
\end{align*} 
%{\bf Step 2.} If $d = 2$ and the compatibility condition is not fulfilled, we may subtract a constant $\mathbf{c} = (c,c,0)$ such that 
%$\mathbf{f}-\mathbf{c}$ fulfills the compatibility condition. Then, we may write
%$\norm{\mathcal{S}_N \boldsymbol{f}  -\mathbf{w}}_{\mathbf{L}^2(\tau)} = \norm{\mathcal{S}_N(\boldsymbol{f}-\mathbf{c}) +  \mathcal{S}_N\mathbf{c} -\mathbf{w}}_{\mathbf{L}^2(\tau)}$, increase the dimension of $\widehat{\mathbf{W}}_L$ and in turn for the $\H$-matrix $\mathbf{B}_{\H}$ by 1
%and use Step~1 for $\boldsymbol{f}-\mathbf{c}$. 
This finishes the proof for the symmetric coupling.

The approximations to $\mathbf{A}^{-1}_{\rm bmc}$ and $\mathbf{A}^{-1}_{\rm jn}$ are constructed in exactly 
the same fashion. The different exponentials appear due to the different dimensions of the low-dimensional space
$\widehat{\mathbf{W}}_L$ in Proposition~\ref{pro:HmatrixBMC} and Proposition~\ref{pro:HmatrixJN}.
\end{proof}

%%%%%%%%%%%%%%%%%%%%%%%%%%%%%%%%%%%%%%%%%%%%%%%%%%%%%%%%%%%%%%%
\section{Numerical results}
\label{sec:numerics}
In this section, we provide a numerical example that supports the theoretical 
results from Theorem~\ref{th:H-Matrix approximation of inverses}, i.e, we compute an exponentially convergent 
$\H$-matrix approximant to an inverse FEM-BEM coupling matrix.

If one is only interested in solving a linear system with one (or few) different right-hand sides, rather than computing the 
inverse -- and maybe even its low-rank approximation -- it is more beneficial to use an iterative solver.
The $\H$-matrix approximability of the inverse naturally allows for black-box preconditioning of the 
linear system. \cite{Bebendorf07} constructed $LU$-decompositions in the $\mathcal{H}$-matrix format 
for FEM matrices by approximating certain Schur-complements 
under the assumption that the inverse can be approximated with arbitrary accuracy. 
Theorem~\ref{th:H-Matrix approximation of inverses} provides such an approximation result and the 
techniques of \cite{Bebendorf07,FMP15,FMP16,FMP17} can also be employed 
to prove the existence of $\H$-LU-decompositions for the whole FEM-BEM matrices for each couplings. 

Here, we additionally present a different, computationally more efficient approach by introducing a
black-box block diagonal preconditioner for the FEM-BEM coupling matrices.

We choose the $3d$-unit cube $\Omega = (0,1)^3$ as our geometry, and we set $\mathbf{C} = \mathbf{I}$.
In the following, we only consider the Johnson-N\'ed\'elec coupling, the other couplings can be treated in 
exactly the same way.
\\

In order to guarantee positive definiteness, we study the stabilized system (see \cite[Thm.~{15}]{AFFKMP13} for the assertion
of positive definiteness)
\begin{align}\label{eq:stabsystem}
\left(\begin{pmatrix} \mathbf{A} & -\mathbf{M}^T \\ \frac{1}{2}\mathbf{M} - \mathbf{K} & \mathbf{V} \end{pmatrix}+ \mathbf{s}\mathbf{s}^T\right)
\begin{pmatrix} \mathbf{x} \\ \boldsymbol{\phi} \end{pmatrix}  = \begin{pmatrix} \mathbf{f} \\ \mathbf{g}  \end{pmatrix},
\end{align}
where the stabilization $\mathbf{s}\in\R^{N+M}$ is given by 
$\mathbf{s}_i = \skp{1,(1/2-K)\xi_i}_{L^2(\Gamma)}$ for $i \in \{1,\dots, N\}$ and 
$\mathbf{s}_i = \skp{1,V\chi_i}_{L^2(\Gamma)}$ for $i \in \{N+1,\dots, M\}$. \\

We stress that \cite{AFFKMP13} show that solving the stabilized (elliptic) system is equivalent to solving the 
non stabilized system (with a modified right-hand side). 
By $\mathbf{A}^{\mathrm{st}}:= \mathbf{A} + \mathbf{b}\mathbf{b}^T$, we denote the stabilization of $\mathbf{A}$,
where $\mathbf{b}$ contains the degrees of freedom of $\mathbf{s}$ corresponding to the FEM part. \\

All computations are made using the C-library \texttt{HLiB}, \cite{HLib}, where we employed a geometric clustering
algorithm with admissibility parameter $\eta = 2$ and a leafsize of $25$. \\

\subsection{Approximation to the inverse matrix}
The $\H$-matrices are computed by using a very accurate blockwise low-rank 
approximation to 
\begin{align}\label{eq:defMstab} 
\mathbf{B} := \begin{pmatrix} \mathbf{A} & -\mathbf{M}^T \\ \frac{1}{2}\mathbf{M} - \mathbf{K} & \mathbf{V} \end{pmatrix} + \mathbf{s}\mathbf{s}^T.
\end{align}
%with adaptive cross approximation (\cite{Bebendorf00}).
Then, using $\H$-matrix arithmetics and blockwise projection to rank $r$, the $\mathcal{H}$-matrix inverse is 
computed with a blockwise algorithm using $\H$-arithmetics from \cite{GrasedyckDissertation}. 
In order to not compute the full inverse, we use the upper bound
\begin{align*}
 \norm{\mathbf{B}^{-1}- \mathbf{B}_{\mathcal{H}}}_2 \leq \norm{\mathbf{B}^{-1}}_2  \norm{\mathbf{I}- \mathbf{B}\mathbf{B}_{\mathcal{H}}}_2 
\end{align*}
for the error.

We also compute a second approximate inverse by use of the $\mathcal{H}$-LU decomposition, which can be  
computed using a blockwise algorithm from \cite{Lintner04,Bebendorf05}. Hereby, we use 
$\norm{\mathbf{I}- \mathbf{B}(\mathbf{L}_{\mathcal{H}}\mathbf{U}_{\mathcal{H}})^{-1}}_2 $ to measure the error
without computing the inverse of $\mathbf{B}$.

Figure~\ref{fig:convergenceInverse1} shows convergence of the upper bounds of the error 
and the growth of the storage requirements with respect to the block-rank $r$ for two different
problem sizes. % full matrices 897 MB and 2465 MB
We observe exponential convergence and linear growth in storage for the approximate inverse using $\H$-arithmetics and the approximate inverse using 
the $\H$-LU decomposition, where the $\H$-LU decomposition performs significantly better.   
The observed exponential convergence is even better than the asserted bound from Theorem~\ref{th:H-Matrix approximation of inverses}.

\begin{figure}[ht]
\begin{minipage}{.50\linewidth}
\includegraphics{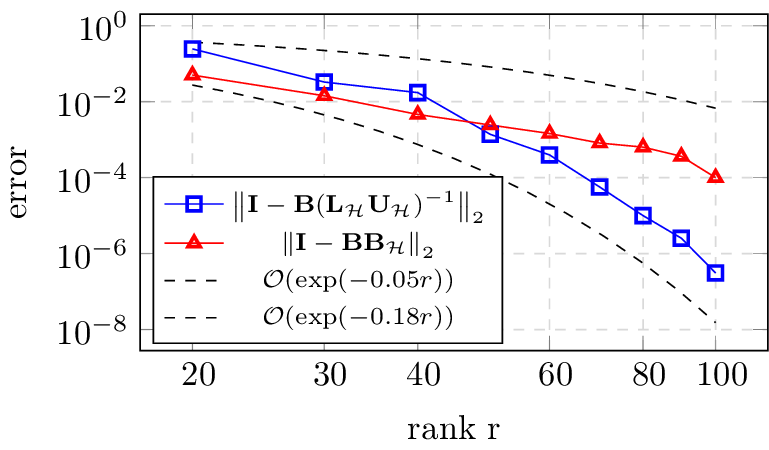}
\end{minipage}
\begin{minipage}{.49\linewidth}
\includegraphics{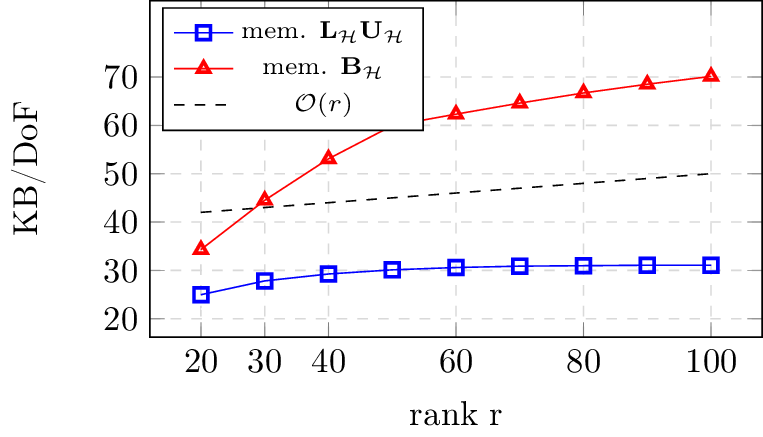}
\end{minipage}

\begin{minipage}{.50\linewidth}
\includegraphics{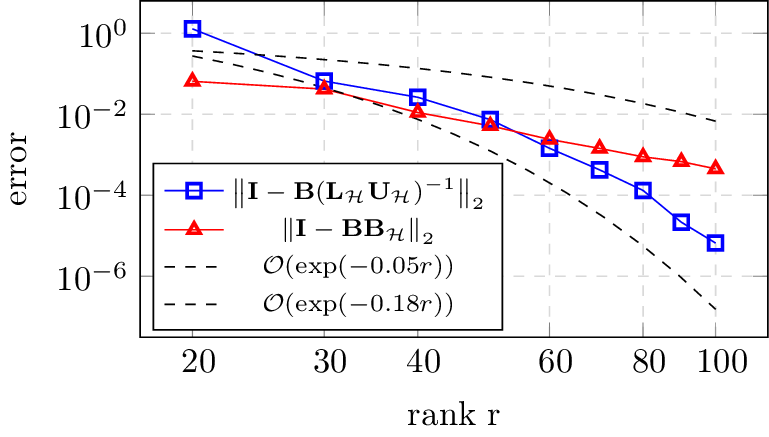}
\end{minipage}
\begin{minipage}{.49\linewidth}
\includegraphics{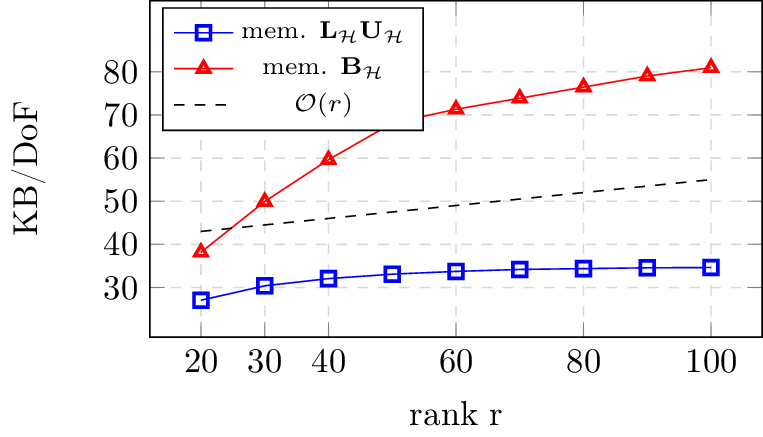}
\end{minipage}
\caption{$\H$-matrix approximation to inverse FEM-BEM matrix; left: error vs.\ block rank $r$; 
right: memory requirement vs.\ block rank $r$; top: $N= 6959$ (FEM-dofs), $M=3888$ (BEM-dofs); 
bottom: $N=10648$, $M=5292$.}
 \label{fig:convergenceInverse1}
\end{figure}

\subsection{Block diagonal preconditioning}

Instead of building an $\H$-LU-decomposition of the whole FEM-BEM matrix, it is significantly cheaper to use  
a block-diagonal preconditioner consisting of $\H$-LU-decompositions for the 
FEM and the BEM part. The efficiency of block-diagonal preconditioners for the FEM-BEM couplings has been 
observed in \cite{MundStephan98,FFPS17}. 

In the following, we consider block diagonal preconditioners of the form
$$
\mathbf{P} = \begin{pmatrix} \mathbf{P}_A & 0 \\ 0 &  \mathbf{P}_V \end{pmatrix},
$$
where $\mathbf{P}_A$ is a good preconditioner for the FEM-block $\mathbf{A}^{\mathrm{st}}$ and 
$\mathbf{P}_V$  is a good preconditioner of the BEM-block $\mathbf{V}$. 

The main result of \cite{FFPS17} is that, provided the 
preconditioners $\mathbf{P}_A$ and $\mathbf{P}_V$ 
fulfill the spectral equivalences 
\begin{align}\label{eq:spectralequivalence}
c_{A}\mathbf{x}^T \mathbf{P}_A \mathbf{x} &\leq \mathbf{x}^T \mathbf{A}^{\mathrm{st}}\mathbf{x} 
\leq C_A \mathbf{x}^T \mathbf{P}_A \mathbf{x} \\
c_V\mathbf{x}^T \mathbf{P}_V \mathbf{x} &\leq \mathbf{x}^T \mathbf{V}\mathbf{x} 
\;\;\leq C_V \mathbf{x}^T \mathbf{P}_V \mathbf{x},
\end{align}
then, $\mathbf{P}$ is a good preconditioner for the full FEM-BEM system. More precisely, 
the condition number $\mathbf{P}^{-1}\mathbf{B}$  (with $\mathbf{B}$ from of \eqref{eq:defMstab}) in the spectral norm can be uniformly bounded by
$$
\kappa_2(\mathbf{P}^{-1}\mathbf{B}) \leq C \frac{\max\{C_A,C_V\}}{\min\{c_A,c_V\}},
$$
where the constant $C$ only depends on the coefficient in the transmission problem.
As a consequence, one expects that the number of GMRES iterations needed to reduce
the residual by a factor remains bounded independent of the matrix size. 

Therefore, we need to provide the preconditioners $\mathbf{P}_A,\mathbf{P}_V$ and prove the spectral equivalences
\eqref{eq:spectralequivalence}. In the following, 
we choose hierarchical $LU$-decompositions as black-box preconditioners, i.e., 
\begin{align*}
\mathbf{P}_A := \mathbf{L}^A_{\H}\mathbf{U}^A_{\H}, \quad 
\mathbf{P}_V := \mathbf{L}^V_{\H}\mathbf{U}^V_{\H},
\end{align*}
where $\mathbf{A}^{\mathrm{st}} \approx \mathbf{L}^A_{\H}\mathbf{U}^A_{\H}$ and
$\mathbf{V} \approx \mathbf{L}^V_{\H}\mathbf{U}^V_{\H}$.
\cite{FMP15,FMP16} prove that such $LU$-decompositions of arbitrary accuracy 
exist for the FEM and the BEM part and the errors, denoted by 
$\varepsilon_A$ and $\varepsilon_V$, converge exponentially in the block-rank of the $\H$-matrices.

With
$\norm{\mathbf{P}_A-\mathbf{A}^{\mathrm{st}}}_2 \leq \varepsilon_A \norm{\mathbf{A}^{\mathrm{st}}}_2$, we estimate
\begin{align}\label{eq:spectralequivA}
\abs{\mathbf{x}^T\mathbf{P}_A\mathbf{x} -\mathbf{x}^T\mathbf{A}^{\mathrm{st}}\mathbf{x}}
\leq \norm{\mathbf{x}}_2^2\norm{\mathbf{P}_A-\mathbf{A}^{\mathrm{st}}}_2 \leq 
\varepsilon_A\norm{\mathbf{x}}_2^2\norm{\mathbf{A}^{\mathrm{st}}}_2
\leq C_1 \varepsilon_A h^{-d} \mathbf{x}^T\mathbf{A}^{\mathrm{st}}\mathbf{x},
\end{align}
where the last step follows from the scaling of the basis of the FEM part and the positive definiteness of 
$\mathbf{A}^{\rm st}$.
In the same way, for $\mathbf{P}_V$ it follows that  
\begin{align}\label{eq:spectralequivV}
\abs{\mathbf{x}^T\mathbf{P}_V\mathbf{x} -\mathbf{x}^T\mathbf{V}\mathbf{x}}\leq \norm{\mathbf{x}}_2^2\norm{\mathbf{P}_V-\mathbf{V}}_2 \leq 
\varepsilon_V\norm{\mathbf{x}}_2^2\norm{\mathbf{V}}_2\leq C_2 \varepsilon_V h^{-d+1} \mathbf{x}^T\mathbf{V}\mathbf{x}.
\end{align}
Choosing the rank of the $\H$-$LU$-decomposition large enough, such that, e.g., 
$C_1\varepsilon_A h^{-d} = \frac{1}{2}$ as well as
$C_2\varepsilon_V h^{-d+1} = \frac{1}{2}$,
then $C_A=C_V = 2$ and $c_A =c_V = \frac{2}{3}$ and the condition number of the preconditioned system is bounded by
$\kappa_2(\mathbf{P}^{-1}\mathbf{B}) \leq 3C$.\newline

Finally, we present a numerical simulation that underlines the usefulness of block-diagonal $\H$-$LU$-preconditioners.

Here, 
the $\mathcal{H}$-$LU$ decompositions are computed with a recursive 
algorithm proposed in \cite{Bebendorf05}.

The following table provides iteration numbers and computation times for the iterative solution of the 
system without and with $\mathcal{H}$-$LU$-block diagonal preconditioner using GMRES. Here, for the 
stopping criterion a bound of $10^{-3}$ for the relative residual is chosen, and the maximal 
rank of the $\mathcal{H}$-$LU$ decomposition is taken to be $r=1$.

\begin{center}
\begin{tabular}{|c|c|c|c|c|c|c|c|}\hline 
$h$ & FEM & BEM & Iterations 
& Iterations & Time solve & Time solve & Time \\ 
 & DOF & DOF &  (without $\mathbf{P}$) 
& (with $\mathbf{P}$) & (without $\mathbf{P}$) 
& (with $\mathbf{P}$) & assembly $\mathbf{P}$ \\ 
\hline $2^{-3}$ & 729 & 768 & 679 & 3 & 3.7 & 0.03 & 2.6 \\ 
\hline $2^{-4}$ & 4913 & 3072 & 3565 & 4 & 315 & 0.9 & 12.2 \\ 
\hline $2^{-5}$ & 35937 & 12288 & 11979 & 5 & 35254 & 30 & 51.9 \\ 
\hline
\end{tabular}
\captionof{table}{
Iteration numbers and computation times (in seconds) for the solution with and without preconditioner with 
block rank $r=1$.}
\end{center}

As expected, the iteration numbers of the preconditioned system is much lower than those of the unpreconditioned 
system and grow very slowly. The computational cost for the preconditioner is theoretically of order
$\mathcal{O}(r^3N\log^3N)$. With the choice $r=1$, we obtain a cheap but efficient preconditioner for the 
FEM-BEM coupling system. 

Table~2 provides the same computations for the case $r=10$.

\begin{center}
\begin{tabular}{|c|c|c|c|c|c|c|c|}\hline 
$h$ & FEM & BEM & Iterations 
& Iterations & Time solve & Time solve & Time \\ 
 & DOF & DOF &  (without $\mathbf{P}$) 
& (with $\mathbf{P}$) & (without $\mathbf{P}$) 
& (with $\mathbf{P}$) & assembly $\mathbf{P}$ \\ 
\hline $2^{-3}$ & 729 & 768 & 679 & 2 & 3.7 & 0.02 & 5.8 \\ 
\hline $2^{-4}$ & 4913 & 3072 & 3565 & 2 & 315  & 0.48 & 24.6 \\ 
\hline $2^{-5}$ & 35937 & 12288 & 11979 & 2 & 35254  & 15.7 & 243.7 \\ 
\hline
\end{tabular}
\captionof{table}{
Iteration numbers and computation times (in seconds) for the solution with and without preconditioner with 
block rank $r=10$.}
\end{center}

A higher choice of rank obviously increases the computational time for the assembly of the preconditioner, but 
leads to lower iteration numbers and faster solution times.

%\nocite{*}
\bibliography{bibliography_2}{}
\bibliographystyle{amsalpha}
\end{document}